\documentclass[reqno, 11pt]{amsart}

\usepackage{amsmath}
\usepackage{amsthm}
\usepackage{amssymb}
\usepackage{mathrsfs}
\usepackage{hyperref}
\usepackage[utf8]{inputenc}
\usepackage{dsfont}
\usepackage{enumitem}
\usepackage{mathabx}

\numberwithin{equation}{section}

\newtheorem{thm}{Theorem}[section]
\newtheorem{lem}[thm]{Lemma}

\theoremstyle{definition}
\newtheorem{Ass}[thm]{Assumption}
\newtheorem{rem}[thm]{Remark}
\newtheorem{cor}[thm]{Corollary}

\title[SVGD dynamics For highly concentrated kernels]{Stein Variational Gradient Descent dynamics for highly concentrated kernels}

\author{José A. Carrillo}
\address{{\it José A. Carrillo:} Mathematical Institute, University of Oxford, Woodstock Road, Oxford, OX2 6GG, United Kingdom}
\email{jose.carrillo@maths.ox.ac.uk}

\author{Jakub Skrzeczkowski}
\address{{\it Jakub Skrzeczkowski: } St John's College, University of Oxford, St Giles, Oxford, OX1 3JP, United Kingdom \& Mathematical Institute, University of Oxford, Woodstock Road, Oxford, OX2 6GG, United Kingdom}
\email{jakub.skrzeczkowski@maths.ox.ac.uk}

\author{Jethro Warnett}
\address{{\it Jethro Warnett: } Mathematical Institute, University of Oxford, Woodstock Road, Oxford, OX2 6GG, United Kingdom}
\email{warnett@maths.ox.ac.uk}

\begin{document}
    
    
    
    \maketitle
    
    \begin{abstract}
        Stein Variational Gradient Descent (SVGD) is a widely used in practice algorithm for scalable sampling with deterministic particle updates. We study its behavior in the singular limit where the kernel bandwidth tends to zero. In this regime, we show that the nonlocal SVGD dynamics converge to a local evolution equation that can be formally interpreted as a Wasserstein gradient flow with quadratic mobility. We analyze this singular limit in two settings: integrable kernels and weighted kernels. In the weighted case, the proof is supported by recently established Stein-log-Sobolev inequalities, which provide the necessary functional control. Overall, our results clarify how SVGD collapses from a nonlocal interacting particle system to a local gradient-flow dynamics as the kernel concentrates.
    \end{abstract}
    
    \setcounter{tocdepth}{1}

    
    \section{Introduction}
    Stein Variational Gradient Descent (SVGD) is a deterministic particle method for approximating a target distribution $\rho_\infty\,\propto\,e^{-V}$ by transporting a collection of particles so as to decrease the Kullback--Leibler divergence to $\rho_\infty$ (also known as the relative entropy) most rapidly within a reproducing-kernel Hilbert space (RKHS) \cite{liu_wang_2016_stein}. Each discrete SVGD update is obtained by projecting the functional gradient of the $\mathrm{KL}$ divergence onto a Stein-admissible vector field in an RKHS. Specifically, the update uses a positive-definite interaction kernel $K:\mathbb{R}^{d}\times\mathbb{R}^{d}\rightarrow\mathbb{R}$ which both regularizes the particle interaction and defines the RKHS geometry used for descent. The SVGD algorithm is an effective deterministic sampler and has motivated a variety of applications including amortized sampling and generator training \cite{feng_wang_liu_2017_learning}, energy-based reinforcement learning \cite{haarnoja_tang_abbeel_levine_2017_reinforcement}, model-predictive control \cite{lambert_ramos_boots_fox_fishman_2021_stein}, constrained or trust-aware sampling \cite{liu_tong_liu_2021_sampling}, and interacting-particle methods for quantization and measure approximation \cite{xu_korba_slepcev_2022_accurate}.\\
    
    Taken to the infinite-particle limit, the empirical particle system of SVGD formally converges to a nonlinear mean-field PDE. Writing $\rho_t$ for the time density of the continuum limit, the mean-field SVGD dynamics can be written as the continuity equation
    \begin{equation}\label{eq:PDE_SVGD_main_intro}
        \partial_t \rho_t(x) = -\mathrm{div}\big( \rho_t(x)\, v[\rho_t](x)\big),
    \end{equation}
    with the kernelized velocity field
    \begin{align*}
        v[\rho](x) = \int_{\mathbb{R}^{d}} \Big( K(x,y)\nabla \log\rho_{\infty}(y) + \nabla_y K(x,y)\Big)\rho(y)\,\mathrm{d} y,
    \end{align*}
    so that the particle interpretation $\dot X_t = v[\rho_t](X_t)$ matches the empirical SVGD update in the large-population limit. This mean-field description has been developed and rigorously justified in several works, including the gradient-flow / KSD-decay viewpoint in \cite{liu_2017_stein}, the scaling-limit and well-posedness results in \cite{lu_lu_nolen_2019_scaling}, non-asymptotic population-limit analyses \cite{korba_salim_arbel_luise_gretton_2021_non}, kernelized Wasserstein-gradient-flow interpretations \cite{chewi_legouic_lu_maunu_rigollet_2020_svgd}, and further convergence and stability improvements and geometric perspectives in \cite{salim_sun_richtarik_2022_convergence,sun_karagulyan_richtarik_2023_convergence,duncan_nusken_szpruch_2023_geometry,carrillo_skrzeczkowski_2023_convergence,he_balasubramanian_sriperumbudur_lu_2024_regularized}. Recently, there was some work done for the finite particle regime in \cite{banerjee_balasubramanian_ghosal_2025_improved,he_balasubramanian_banerjee_ghosal_2026_finite}.\\
    
    This paper investigates the behavior of the mean-field PDE \eqref{eq:PDE_SVGD_main_intro} when the interaction kernel is translation invariant and becomes increasingly concentrated near the origin. This problem was motivated in \cite{lu_lu_nolen_2019_scaling}. Concretely, let
    \begin{equation}\label{eq:general_form_kernel_general_weights}
        K_\sigma(x,y)=w(x)\,k_\sigma(x-y)\,w(y), \qquad k_\sigma(z)=\sigma^{-d}k(z/\sigma),
    \end{equation}
    denote a family of translation-invariant kernels with $k_\sigma$ concentrating to a Dirac mass as $\sigma\to 0$ (here, $k$ is a fixed, integrable function). Our goal is to understand whether, as $\sigma\to 0$, the nonlocal mean-field PDEs \eqref{eq:PDE_SVGD_main_intro} with $K := K_\sigma$ converge in a suitable topology to a well-defined local limit. Furthermore, we want to identify the limiting equation and study its analytic properties. \\
    
    To be more precise, we will consider two type of PDEs
    \begin{align}
        \label{eq:mf_svgd}
        \begin{split}
            \partial_t \varrho&=\mathrm{div}\Big(\varrho\;  (k_\sigma\ast (\nabla \varrho+\varrho\,\nabla V))\Big)\\
            \varrho(0,x) &= \varrho_0(x)
        \end{split}
        \begin{split}
            & \mbox{ on } (0,\infty) \times \mathbb{R}^{d}, \\
            &\mbox{ on }\mathbb{R}^{d},
        \end{split}
    \end{align}
    and
    \begin{align}
        \label{eq:mf_slsi_svgd}
        \begin{split}
            \partial_t \rho&=\mathrm{div}\Big(\rho\; e^{V-\frac{\mathbb{V}}{2}}\; \big(k_\sigma\ast (\nabla \left( \rho\, e^{V} \right) e^{-\frac{\mathbb{V}}{2}})\big)\Big)\\
            \rho(0,x) &= \rho_0(x)
        \end{split}
        \begin{split}
            & \mbox{ on } (0,\infty) \times \mathbb{R}^{d}, \\
            &\mbox{ on }\mathbb{R}^{d}.
        \end{split}
    \end{align}
    Both equations \eqref{eq:mf_svgd} and \eqref{eq:mf_slsi_svgd} use the kernel in \eqref{eq:general_form_kernel_general_weights}, where the first uses the weight $w=1$, and the second uses the weight
    \begin{align}
        \label{eq:slsi_weight}
        w(x)=e^{V(x)-\frac{\mathbb{V}(x)}{2}},
    \end{align}
    where $\mathbb{V}$ is a strictly convex quadratic polynomial such that $V\geq \mathbb{V}$. In both cases, $k_\sigma$ is always assumed to be of the form
    \begin{align*}
        k_\sigma = \omega_\sigma\ast\omega_\sigma,
    \end{align*}
    with some further assumptions to be specified later. The first equation is a standard formulation of the continuous SVGD method \cite{lu_lu_nolen_2019_scaling,duncan_nusken_szpruch_2023_geometry,liu_2017_stein} while the second is motivated by the recent paper \cite{carrillo_skrzeczkowski_warnett_2024_stein} which has proven for the very first time the existence of a kernel that satisfies the so-called Stein-log-Sobolev inequality, conjectured in one space dimention in \cite{duncan_nusken_szpruch_2023_geometry},
    \begin{align}
        \label{eq:slsi}
        \begin{split}
            \lambda\, \mathrm{KL}(\rho\,||\,\rho_\infty)&\leq \mathbb{D}^2(\rho \,||\, \rho_\infty)\\
            &:=\int_{\mathbb{R}^{d}}\int_{\mathbb{R}^{d}}\nabla\left(\frac{\,\mathrm{d} \rho}{\,\mathrm{d} \rho_\infty}\right)(x)\cdot K_\sigma(x,y)\nabla\left(\frac{\,\mathrm{d} \rho}{\,\mathrm{d} \rho_\infty}\right)(y)\,\mathrm{d}\rho_\infty(y)\,\mathrm{d}\rho_\infty(x),
        \end{split}
        \tag{SLSI}
    \end{align}
    where $K_{\sigma}$ is defined in \eqref{eq:general_form_kernel_general_weights} with the weight $w$ given by \eqref{eq:slsi_weight} and $k\in L^1(\mathbb{R}^{d})+L^2(\mathbb{R}^{d})$ with its Fourier transform satisfying the growth bounds
    \begin{equation}\label{eq:fourier_bound_kernel_introduction}
        \frac{1}{\mathcal{D}_0}\,\frac{1}{1+|\xi|^2}\leq \widehat{k}(\xi)\leq \mathcal{D}_1\,\frac{1}{1+|\xi|^2},
    \end{equation}
    and $\lambda=\lambda_0/\mathcal{D}_0$ is independent of $\sigma \in (0,1)$ (see Remark \ref{rem:idea_SLSI_constant_ind_of_scl}). The inequality \eqref{eq:slsi} implies exponential rate of convergence of $\mathrm{KL}(\rho_t\,||\,\rho_\infty) \to 0$ whenever $\rho_t$ solves \eqref{eq:mf_slsi_svgd}.
    
    In the limit $\sigma \to 0$, we expect to obtain for \eqref{eq:mf_svgd}
    \begin{align}
        \label{eq:lim_mf_svgd}
        \begin{split}
            \partial_t \varrho&=\mathrm{div}\Big(\varrho^2\;  \nabla\big(\ln(\varrho)+V\big)\Big)\\
            \varrho(0,x) &= \varrho_0(x)
        \end{split}
        \begin{split}
            & \mbox{ on } (0,\infty) \times \mathbb{R}^{d}, \\
            &\mbox{ on }\mathbb{R}^{d},
        \end{split}
    \end{align}
    and for \eqref{eq:mf_slsi_svgd}
    \begin{align}
        \label{eq:lim_mf_slsi_svgd}
        \begin{split}
            \partial_t \rho&=\mathrm{div}\Big(\rho^2\; e^{2V-\mathbb{V}}\; \nabla\big(\ln(\rho)+V\big)\Big)\\
            \rho(0,x) &= \rho_0(x)
        \end{split}
        \begin{split}
            & \mbox{ on } (0,\infty) \times \mathbb{R}^{d}, \\
            &\mbox{ on }\mathbb{R}^{d}.
        \end{split}
    \end{align}
    
    The target of the paper is to justify the limit $\sigma \to 0$ rigorously. More precisely, we show that the weak solutions $\{\rho^\sigma\}_\sigma$, $\{\varrho^\sigma\}_\sigma$ of both equations \eqref{eq:mf_svgd} and
    \eqref{eq:mf_slsi_svgd} converge to the weak solutions $\rho$, $\varrho$ of both equations \eqref{eq:lim_mf_svgd} and $\eqref{eq:lim_mf_slsi_svgd}$ respectively on bounded time intervals as $\sigma\to 0$ in a suitable weak topology. However, we are only able to show
    exponentially fast convergence of weak solutions $\{\rho^\sigma\}_\sigma$ with respect to the Kullback-Leibler divergence to $\rho_\infty$ as $t\to\infty$, uniformly with respect to $\sigma$. The same exponential rate holds for the solution $\rho$ of the limit equation \eqref{eq:lim_mf_slsi_svgd}, which converges to $\rho_\infty$ as $t\to\infty$. This is mainly due to an uniform in $\sigma$ estimate on the Stein-log-Sobolev inequality \eqref{eq:slsi}. The situation is different for the weak solutions $\{\varrho^\sigma\}_\sigma$ of \eqref{eq:mf_svgd} due to the lack of the corresponding Stein-log-Sobolev inequality \eqref{eq:slsi}. As a consequence, we cannot prove asymptotic convergence of the weak solution $\varrho$ of \eqref{eq:lim_mf_svgd} to $\rho_\infty$ in time. This remains as of now an open problem.\\
    
    The problem is important for the mathematical analysis of the SVGD method. Actually, it was conjectured in \cite[Equation 9]{lu_lu_nolen_2019_scaling} that if the kernel is scaled proportionally to the number of particles $N^\beta K(N^\beta\,\cdot)$, one expects the behavior of the particle system to resemble \eqref{eq:lim_mf_svgd}, and not \eqref{eq:mf_svgd}. In the current work we address this question partially by studying the effect of the concentration of the kernels at the continuous level \eqref{eq:mf_svgd} and \eqref{eq:mf_slsi_svgd}. The general problem of simultaneously studying concentration of the kernels and the mean-field limit from particles remains open and appears challenging. The main difficulty is that, for the interacting particle system, one lacks uniform estimates on the dissipation $\mathbb{D}^2(\rho_t \,||\, \rho_\infty)$, which are essential to our approach. Additionally, for the equations \eqref{eq:mf_svgd} we can apply a Dobrushin type result to justify the joint-limit $N\rightarrow\infty$ and $\sigma\to 0$ of the particle system to the mean-field limit, see for instance \cite[Theorem 1.1]{carrillo_skrzeczkowski_2023_convergence}. This is not possible for \eqref{eq:mf_slsi_svgd} because of the lack of regularity for the kernel $k$ implied by the growth estimates of its Fourier transform in  \eqref{eq:fourier_bound_kernel_introduction}. \\
    
    Our work also addresses new challenges in the study of nonlocal-to-local limits for partial differential equations, where one rigorously justifies the passage to the limit in nonlocal models. Such limits, among others, connect different equations in mathematical physics, such as the Cahn–Hilliard equation \cite{davoli_scarpa_trussardi_2021_nonlocal,elbar_skrzeczkowski_2023_degenerate}, and provide a rigorous foundation for numerical methods like the blob method \cite{carrillo_craig_patacchini_2019_blob,carrillo_esposito_skrzeczkowski_wu_2024_nonlocal,craig_jacobs_turanova_2025_nonlocal}. The challenge in the analysis of nonlocal-to-local limits is to control the quantities related to $\rho^\sigma$ or $\rho^\sigma\ast\omega_\sigma$, uniformly in $\sigma$. Our approach consists of two new novel ideas. First, in Lemma \ref{prop:bounds_gradient} we are able to deduce a uniform estimate on $\nabla \rho^\sigma \ast \omega_\sigma$ by exploiting the uniform bound on dissipation $\mathbb{D}^2(\rho^\sigma\, ||\, \rho_\infty)$ and viewing $(\rho^\sigma\, \nabla V)\ast \omega_\sigma$ as its small perturbation. This estimate is only local in space which poses additional challenges with kernels which are globally supported. Thus, we have to carefully control the tails.\\

    The second idea addresses so-called commutator estimates \cite{constantin_e_titi_1994_onsager,diperna_lions_1989_ordinary,elbar_perthame_skrzeczkowski_2024_limit}. When passing to the limit in the distributional formulation of \eqref{eq:mf_svgd}--\eqref{eq:mf_slsi_svgd}, we have to prove strong convergence of the term $(\nabla \varphi \,  \rho^\sigma)\ast \omega_\sigma$ where $\varphi$ is smooth and compactly supported. A priori, this is difficult because there is no compactness information available on $\rho^\sigma$ itself while $\rho^\sigma\ast\omega_\sigma$ is compact only locally in space. In Lemma \ref{lem:comm_est} we propose a new idea based on expanding test function $\nabla \varphi(x-y)$ into a Taylor's series:
    \begin{align*}
        &|(\nabla \varphi \,  \rho^\sigma)\ast \omega_\sigma - \nabla \varphi \, (\rho^\sigma\ast \omega_\sigma)|\\ &\leq \sum_{k=2}^{r} C_k\, |\nabla^k \varphi(x)|\, (\rho^\sigma\ast(\omega_\sigma\, |y|^{k-1}))+ C_{r+1} \, \|\nabla^{r+1} \varphi\|_{L^{\infty}} \, \rho^\sigma\ast(\omega_\sigma\, |y|^{r}).
    \end{align*}
    The second term converges strongly to 0 because $r$ is large. The first term can be handled by the estimate $|y|^{k-1}\leq |y|^r + 1$ and the Vitali convergence theorem since the test function localizes the reasoning to a compact set. We also introduce a new approach to estimating the commutator $(f\chi)\ast\omega_\sigma-\chi\,(f\ast\omega_\sigma)$ for a test function $\chi$, where $f$ has only $H^{-1}$ regularity in space, by working in Fourier variables, see Lemma \ref{lem:h-1_mod_bound}. This plays an important role in the compactness argument.  \\

    Moreover, the limiting PDE \eqref{eq:lim_mf_svgd} can be viewed as the 2-Wasserstein gradient flow of the functional $\mathcal{F}[\rho] = \int_{\mathbb{R}^{d}} (\rho\, \log \rho + \rho\, V) \,\mathrm{d} x$ with a convex mobility $m(\rho) = \rho^2$. However, a rigorous understanding of PDEs with nonlinear mobilities is still missing. In particular, PDEs of the form $\partial_t \rho = \mathrm{div}(m(\rho)\,\nabla \frac{\delta F}{\delta \rho})$ have rigorous gradient flow interpretation only for concave mobilities $m$ \cite{carrillo_lisini_savare_slepcev_2010_nonlinear} with additional results available for $m(\rho) = \rho^{\alpha}$ and $\alpha\in(0,1)$ \cite{decourcel_elbar_2025_repulsion,carrillo_gomezcastro_vazquez_2022_fast}. For convex mobilities like $m(\rho) = \rho^{\alpha}$ with $\alpha>1$ the available results deal only with $\mathcal{F}[\rho] = \frac{1}{2}\int_{\mathbb{R}^{d}} \rho\, (\mathcal{N}\ast \rho) \,\mathrm{d} x$, where $\mathcal{N}$ is a Newtonian potential \cite{decourcel_elbar_2025_repulsion,carrillo_gomezcastro_vazquez_2022_vortex}. Therefore, we are able to expand on this literature by adding a new class of functionals.\\
    
    The structure of the paper is as follows. In Section~\ref{sect:main_res}, we introduce relevant notation and present the main results of the paper. Sections~\ref{sec:unif_est}, \ref{sec:comm} and \ref{subsec:dt_est} are devoted to establishing crucial a~priori estimates, uniform with respect to the parameter $\sigma \in (0,1)$. Specifically, in Section~\ref{sec:unif_est} we establish some basic estimates, in Section \ref{sec:comm} we show commutator arguments, and in Section \ref{subsec:dt_est} we show how to control the time derivatives. Afterwards, in Section~\ref{sec:weak_conv} we establish strong and weak limits of the solutions $\{\rho^\sigma\}_{\sigma\in (0,1)}$ to some curve $\rho\in C^0([0,\infty);\;\mathscr{P}(\mathbb{R}^{d}))$. Finally, in Section \ref{sec:main_proof} we combine all the previous results to prove Theorems \ref{thm:ex_weak_sol_loc_pde} and \ref{thm:ex_weak_sol_nonloc_pde}. Additionally, in Appendix \ref{app:bessel_pot} we present the Bessel potential as an example of a kernel that satisfies Assumption \ref{assumptions:slsi}. Appendix \ref{app:weak_sol} proves the existence of weak solutions to \eqref{eq:mf_svgd} and Appendix \ref{app:aux} contains necessary auxiliary results that do not fit the rest of the paper.

    
    \section{Main results}\label{sect:main_res}
    
    We need to make assumptions on the kernel and on the potential for both equations \eqref{eq:mf_svgd} and \eqref{eq:mf_slsi_svgd} in order for us to prove the existence of weak solutions. Specifically, for the equation \eqref{eq:mf_slsi_svgd} we make the following assumption. Note that we do not assume that $\omega\geq 0$ in Assumption~\ref{assumptions:slsi}.
    
    \begin{Ass}
        \label{assumptions:slsi}
        There exists a radially symmetric function $\omega\in L^1(\mathbb{R}^{d})$ with unit mass $\int_{\mathbb{R}^{d}}\omega\,\mathrm{d} x=1$ and a fixed integer $m_0>\frac{d}{2}$ such that the following statements be true:
        \begin{enumerate}[label=(A\arabic*)]
            \item The idendity $k=\omega\ast\omega$ holds true.
            \item  \label{ass:four_mom}
            There exists $\mathcal{D}_0,\mathcal{D}_1\geq 1$, \begin{align*}
            p_0\in\begin{cases}
            (1,\frac{6}{5}]&\text{if }d=1,\\
            (1,\frac{6}{5})&\text{if }d=2,\\
            (1,\frac{2(d+1)}{2d+1}]&\text{if }d\geq 3,
        \end{cases}
    \end{align*}
    and a uniformly continuous function $\zeta\in C^0(\mathbb{R}^{d})$ such that the kernel $\omega$  satisfies
    \begin{align}
        \label{ass:fourier}
        &\zeta(x)=\hat{\omega}(x)\,\sqrt{1+|x|^2},\quad\frac{1}{\mathcal{D}_0}\leq \zeta\leq \mathcal{D}_1,\\
        \label{ass:moment}
        &|x|\,\omega\in  L^{p_0}(\mathbb{R}^{d}),
    \end{align}
    where $\hat{\omega}$ is the Fourier transform of $\omega$.
    \item\label{ass:quad}
    The potential $V\in C^1(\mathbb{R}^{d})\cap W_{\mathrm{loc}}^{2,\infty}(\mathbb{R}^{d})\cap H_{\mathrm{loc}}^{m_0}(\mathbb{R}^{d})$ satisfies the following conditions. There exists a strictly convex quadratic and non-negative polynomial $\mathbb{V}$ such that $V\geq \mathbb{V}$.
    \item\label{ass:init_slsi} There exists a probability distribution $\rho_0:\mathbb{R}^{d} \to [0,\infty)$ with unit mass $\int_{\mathbb{R}^{d}} \rho_0 \,\mathrm{d} x= 1$ with
    \begin{align*}
        \int_{\mathbb{R}^{d}} \rho_0\, |\log \rho_0| + \rho_0\, V \,\mathrm{d} x
        < \infty.
    \end{align*}
\end{enumerate}
\end{Ass}

\begin{rem}
    The range of admissible $p_0$ in Assumption \ref{assumptions:slsi} is determined by Lemma \ref{bnd:dt_rho} and the interval $\mathcal{K}_d$. Additionally, the continuity of $\zeta$ is needed in Lemma \ref{lem:h-1_mod_bound}.
\end{rem}

On the other hand, for the equation \eqref{eq:mf_svgd} we must make the upcoming assumption.

\begin{Ass}\label{assumptions:pot}
    There exists a radially symmetric function $\omega\in L^1(\mathbb{R}^{d})$ with unit mass $\int_{\mathbb{R}^{d}}\omega\,\mathrm{d} x=1$, $\omega \geq 0$ and fixed integer $m_0>\frac{d}{2}$ and $p_0>1$ such that the following hold:
    \begin{enumerate}[label=(B\arabic*)]
        \item The idendity $k=\omega\ast\omega$ holds true.
        \item\label{ass:sob_kern} The kernel satisfies
        \begin{align}
            \omega\,(1+|x|^{p_0^\ast\vee m_0})\in L^1(\mathbb{R}^d) \cap L^2(\mathbb{R}^d),
            \label{eq:pot_omega}
        \end{align}
        where we define $p_0^\ast:=\frac{p_0}{p_0-1}$ and use as shorthand $a\vee b=\max\{a,b\}$.
        
        \item\label{ass:pot}
        The non-negative potential $V\in C^2(\mathbb{R}^{d})$ satisfies the following conditions: there exists a constant $C_V \geq 1$ and for all $\alpha, \beta > 0$ there exists a constant $C_{\alpha,\beta} > 0$ such that
        \begin{align}
            & \lim_{R\to\infty} \;\inf_{|x|\geq R} V(x)=\infty,\label{eq:coercive}\\
            &|\nabla V(x)|^{p_0} \leq C_V \, (1+ V(x)),\label{eq:bound_nabla_V_by_V}\\
            \sup_{\theta\in [0,1]}|&\nabla^2 V(\theta x+(1-\theta)y)|^{p_0}\leq C_V\,(1+V(x)+V(y)),
            \label{eq:bound_hessian_V}\\
            \sup_{|y|<\alpha|x|+\beta} &(1+|x|)(|\nabla V(y)|+|\nabla^2 V(y)|)\leq C_{\alpha,\beta}\, (1+V(x)).
            \label{eq:bound_nabla_hessian_V}
        \end{align}
        \item\label{ass:init} There exists a probability distribution $\varrho_0:\mathbb{R}^{d} \to [0,\infty)$ with unit mass $\int_{\mathbb{R}^{d}} \varrho_0 \,\mathrm{d} x= 1$ with
        \begin{align*}
            \int_{\mathbb{R}^{d}} \varrho_0\, |\log \varrho_0| + \varrho_0\, V \,\mathrm{d} x
            < \infty.
        \end{align*}
    \end{enumerate}
\end{Ass}

\begin{rem}
    \label{rem:poly_growth_V}
    It was shown in \cite[Remark 2.3]{lu_lu_nolen_2019_scaling} that under the assumption \eqref{eq:bound_nabla_V_by_V} there exists some constant $C_0>0$ such that the potential $V$ is bounded by some polynomial growth
    \begin{align}
        \label{eq:V_poly_gorwth}
        V(x)\leq C_0\,(1+|x|^{p_0^\ast}).
    \end{align}
    We note that assumptions \eqref{eq:bound_hessian_V} and \eqref{eq:bound_nabla_hessian_V} are only used in the Appendix~\ref{app:weak_sol}. This is done so that we can use \cite[Theorem 2.4]{lu_lu_nolen_2019_scaling}, a key ingredient for our approximation scheme.
\end{rem}

With these assumptions we are able to prove the following theorems.

\begin{thm}
    \label{thm:ex_weak_sol_loc_pde}
    Let Assumption \ref{assumptions:slsi} be true. There exists a subsequence $\{\rho^\sigma\}_\sigma$ of weak solutions to \eqref{eq:mf_slsi_svgd} such that $\rho^\sigma_t\rightharpoonup \rho_t$ narrowly for all $t \in [0,\infty)$, where $\rho\in C^0([0,\infty);\;\mathscr{P}(\mathbb{R}^{d}))$ is a weak solution to \eqref{eq:lim_mf_slsi_svgd}, i.e. for all $T>0$ and for all test functions $\psi\in C_c^\infty([0,T]\times \mathbb{R}^{d})$ we have
    \begin{align}
        \label{eq:weak_sol_loc_pde_slsi}
        \begin{split}
            &\int_{\mathbb{R}^{d}}\psi(T,x)\, \rho_T(x) \,\mathrm{d} x- \int_{\mathbb{R}^{d}}\psi(0,x)\, \rho_0(x) \,\mathrm{d} x\\
            &=
            \int_0^T \int_{\mathbb{R}^{d}}\partial_t\psi\; \rho \,\mathrm{d} x\,\mathrm{d} t-\int_0^T\int_{\mathbb{R}^{d}} \rho\, e^{V-\frac{\mathbb{V}}{2}}\nabla \psi\cdot (\nabla\rho +\rho \nabla V)\, e^{V-\frac{\mathbb{V}}{2}}\,\mathrm{d} x\,\mathrm{d} t.
        \end{split}
    \end{align}
    Moreover, the curve is absolutely continuous with respect to the Wasserstein metric $W_1$ and it enjoys the following regularity, for all $T,R>0$,
    \begin{align}
        \label{eq:regularity_weak_sol_slsi}
        &\rho+|\rho\ln(\rho)|+|\rho V|\in L^\infty(0,T;\, L^1(\mathbb{R}^{d})),\\
        \label{eq:regularity_weak_sol_slsi_2}
        &\rho \,e^{V-\frac{\mathbb{V}}{2}}+|(\nabla \rho+\rho \nabla V) \,e^{V-\frac{\mathbb{V}}{2}}|\in L^2([0,T]\times \mathbb{R}^{d}),\\
        \label{eq:regularity_weak_sol_slsi_3}
        &\rho\, e^{V-\frac{\mathbb{V}}{2}}\in L^2(0,T;\, H^1(B_R)).
    \end{align}
    Additionally, $\rho$ satisfies an energy dissipation type inequality and exponential decay in the Kullback-Leibler distance
    \begin{equation}\label{eq:edi_slsi}
        \mathrm{KL}(\rho_t||\rho_\infty) + \int_0^t |\nabla(\rho_s \, e^{V})e^{-\frac{\mathbb{V}}{2}}|^2\,\mathrm{d} s \leq \mathrm{KL}(\rho_0||\rho_\infty), \quad \mathrm{KL}(\rho_t||\rho_\infty)
        \leq e^{-\lambda_0 t}\, \mathrm{KL}(\rho_0||\rho_\infty),
    \end{equation}
    where $\lambda_0$ is the constant from Theorem \ref{thm:slsi} below.
\end{thm}

\begin{rem}
    \label{rem:ass_expl}
    The Bessel potential of order one satisfies the criteria outlined in Assumption~\ref{assumptions:slsi} (see Appendix \ref{app:bessel_pot} for more details). This potential is characterized by radial symmetry, is non-negative, possesses a singularity at the origin akin to the Newton potential, and exhibits exponential decay as the distance approaches infinity.
\end{rem}

The Assumption \ref{assumptions:slsi} has been chosen such that the kernel $k$ satisfies the so-called Stein-log-Sobolev inequality (SLSI). In particular, for the proof of Theorem \ref{thm:ex_weak_sol_loc_pde} we need to use the following result.

\begin{thm}{(see \cite[Theorem 1.1, Remark 1.2]{carrillo_skrzeczkowski_warnett_2024_stein})}
    \label{thm:slsi}
    Let $k\in L^1(\mathbb{R}^{d})+ L^2(\mathbb{R}^{d})$ be such that
    \begin{align*}
        \frac{1}{\mathcal{D}_0}\frac{1}{1+|\xi|^2}\leq \hat{k}(\xi)\leq \mathcal{D}_1\frac{1}{1+|\xi|^2},
    \end{align*}
    for some constants $\mathcal{D}_0, \mathcal{D}_1\geq 1$. Let $V\in L_\mathrm{loc}^\infty(\mathbb{R}^{d})$ such that there exists a strictly convex quadratic polynomial $\mathbb{V}\in C^\infty(\mathbb{R}^{d})$ such that $V\geq \mathbb{V}$. Let $\rho_\infty\in\mathscr{P}(\mathbb{R}^{d})$ be such that $\rho\propto e^{-V}$. Then, there exists a constant $\lambda_0$ depending only on $V$, $\mathbb{V}$ such that whenever $\rho\in \mathscr{P}(\mathbb{R}^{d})$ satisfies
    \begin{align*}
        \rho\in L^1(\mathbb{R}^{d}),\qquad
        \rho \, e^{V-\frac{\mathbb{V}}{2}}\in L^2(\mathbb{R}^{d}),\qquad
        \rho \, e^{V-\frac{\mathbb{V}}{2}}\nabla\mathbb{V}\in H^{-1}(\mathbb{R}^{d}),
    \end{align*}
    then the Stein-log-Sobolev inequality holds
    \begin{align}
        \label{eq:thm_slsi}
        \frac{\lambda_0}{\mathcal{D}_0}\,\mathrm{KL}(\rho\,||\,\rho_\infty)
        \leq \mathbb{D}^2(\rho\,||\,\rho_\infty)
        =\int_{\mathbb{R}^{d}}\nabla (\rho\, e^V)e^{-\frac{\mathbb{V}}{2}}\cdot (k\ast (\nabla (\rho \, e^V)e^{-\frac{\mathbb{V}}{2}}))\,\mathrm{d} x,
    \end{align}
    where we recall the definition of the dissipation in \eqref{eq:slsi} with $K(x,y)=w(x) k(x-y) w(y)$ and $w=e^{V-\frac{\mathbb{V}}{2}}$.
\end{thm}

\begin{rem}\label{rem:idea_SLSI_constant_ind_of_scl}
    The proof of Theorem \ref{thm:slsi} is based on constructing a particular kernel $k_{0,d}$ that satisfies the Stein-log-Sobolev inequality \eqref{eq:thm_slsi} with constant $\Lambda$ and the Fourier bound $\hat{k}_{0,d}(\xi) \leq \frac{\mathcal{C}_1}{1+|\xi|^2}$. We can compare it with an arbitrary kernel through the chain of inequalities
    \begin{align*}
        \frac{1}{\mathcal{C}_1\,\mathcal{D}_0}\hat{k}_{0,d}(\xi)\leq \frac{1}{\mathcal{D}_0}\frac{1}{1+|\xi|^2}
        \leq \hat{k}(\xi).
    \end{align*}
    Letting $\lambda_0 = \frac{\Lambda}{\mathcal{C}_1}$ (this depends only on $k_{0,d}$), the resulting SLSI coefficient for $k$ is given by $\frac{\lambda_0}{\mathcal{D}_0}$. Moreover, this constant is uniform for all scalings $\{k_\sigma\}_\sigma$, since we have the lower bound
    \begin{align*}
        \frac{1}{\mathcal{D}_0}\frac{1}{1+|\xi|^2}
        \leq \frac{1}{\mathcal{D}_0}\frac{1}{1+\sigma^2 |\xi|^2}\leq\hat{k}(\sigma\xi)=\hat{k}_\sigma(\xi),
    \end{align*}
    so the constant $\mathcal{D}_0$ does not depend on $\sigma$.
\end{rem}

\begin{thm}
    \label{thm:ex_weak_sol_nonloc_pde}
    Let Assumption \ref{assumptions:pot} hold. There exists a subsequence $\{\varrho^\sigma\}_\sigma$ of weak solutions to \eqref{eq:mf_svgd} such that $\varrho^\sigma_t\rightharpoonup \varrho_t$ narrowly for all $t \in [0,\infty)$, where $\varrho\in C^0([0,\infty);\;\mathscr{P}(\mathbb{R}^{d}))$ is a weak solution to \eqref{eq:lim_mf_svgd}, i.e. for all $T>0$ and for all test functions $\psi\in C_c^\infty([0,T]\times \mathbb{R}^{d})$ we have
    \begin{align}
        \label{eq:weak_sol_loc_pde}
        \begin{split}
            &\int_{\mathbb{R}^{d}}\psi(T,x)\, \varrho_T(x) \,\mathrm{d} x- \int_{\mathbb{R}^{d}}\psi(0,x)\, \varrho_0(x) \,\mathrm{d} x\\
            &=
            \int_0^T \int_{\mathbb{R}^{d}}\partial_t\psi\; \varrho \,\mathrm{d} x\,\mathrm{d} t-\int_0^T\int_{\mathbb{R}^{d}} \varrho\, \nabla \psi\cdot (\nabla \varrho+\varrho \nabla V) \,\mathrm{d} x\,\mathrm{d} t.
        \end{split}
    \end{align}
    The curve is absolutely continuous with respect to the Wasserstein metric $W_1$ and it enjoys the following regularity, for all $T,R>0$,
    \begin{align}
        \label{eq:regularity_weak_sol}
        &\varrho+|\varrho\ln(\varrho)|+|\varrho V|\in L^\infty(0,T;\, L^1(\mathbb{R}^{d})),\\
        \label{eq:regularity_weak_sol_2}
        &|\nabla\varrho+\varrho\nabla V|\in L^2([0,\infty)\times \mathbb{R}^{d}),\\
        \label{eq:regularity_weak_sol_3}
        &\varrho\in L^2(0,T;\, H^1(B_R)).
    \end{align}
    Additionally, $\varrho$ satisfies an energy dissipation type inequality
    \begin{equation}\label{eq:lim_edi}
        \mathrm{KL}(\varrho_t\,||\,\rho_\infty) + \int_0^t |\nabla \varrho_s +\varrho_s\nabla V|^2\,\mathrm{d} s \leq \mathrm{KL}(\varrho_0\,||\,\rho_\infty).
    \end{equation}
\end{thm}


\subsection{Notation}

Let us first solidify some conventions and notation used throughout the paper. All objects live on $\mathbb{R}^d$ or on $[0,T]\times\mathbb{R}^d$ with variables $x$ and $(t,x)$ respectively. For $R>0$ we denote the open ball by $B_R:=\{x\in\mathbb{R}^d:\ |x|<R\}$. The Schwartz space is written $\mathcal{S}(\mathbb{R}^d)$, with dual $\mathcal{S}'(\mathbb{R}^d)$. The space $\mathrm{Lip}_L(\mathbb{R}^d)$ consists of $L$--Lipschitz functions on~$\mathbb{R}^d$. Let $\mathscr{P}(\mathbb{R}^{d})$ be the space of Borel probability measures on $\mathbb{R}^d$, endowed with the Kantorovich–Rubinstein distance, which coincides with the $1$–Wasserstein metric:
\[
W_1(\mu,\nu)
:=\sup_{f\in\mathrm{Lip}_1(\mathbb{R}^d)}
\left|\int_{\mathbb{R}^d} f\,\mathrm{d} \mu-\int_{\mathbb{R}^d} f\,\mathrm{d} \nu\right|.
\]
A curve $[0,T]\ni t\mapsto \mu_t\in\mathscr{P}(\mathbb{R}^{d})$ is called absolutely continuous in $(\mathscr{P}(\mathbb{R}^{d}),W_1)$ if
\[
\sup_{t\in[0,T]}\int_{\mathbb{R}^d} |x| \,\mathrm{d}\mu_t < \infty
\]
and if there exists $f\in L^1(0,T)$ such that for all $0\le t\le s\le T$,
\[
W_1(\mu_t,\mu_s)\le \int_t^s f(r)\,\mathrm{d} r.
\]
For $s\in\mathbb{R}$ we use the standard Sobolev space
\[
H^s(\mathbb{R}^d):=\{u\in\mathcal{S}'(\mathbb{R}^d): (1+|\xi|^2)^{s/2}\widehat{u}(\xi)\in L^2(\mathbb{R}^d)\}.
\]
If $X$ is a Banach space of functions or distributions on $\mathbb{R}^d$, the Bochner space $L^p(0,T;X)$ carries the norm
\[
\|f\|_{L^p(0,T;X)}=\Big(\int_0^T \|f(t)\|_X^p\,\mathrm{d} t\Big)^{1/p}.
\]
We abbreviate these norms using subscripts on $t$ and $x$, e.g.
\[
\|\,\cdot\,\|_{L_t^p H_x^1}
:=\|\,\cdot\,\|_{L^p(0,T;H^1(\mathbb{R}^d))},
\qquad
\|\,\cdot\,\|_{H^1}:=\|\,\cdot\,\|_{H^1(\mathbb{R}^d)}.
\]
The same conventions apply to all other function spaces.
For any Banach space $X$, we write $X^\ast$ for its dual.
Convergences are denoted by $f_n\to f$ (strong) and $f_n\rightharpoonup f$ (weak). For $p,q,r\in[1,\infty]$ with $1+\tfrac1r=\tfrac1p+\tfrac1q$ and $f\in L^p(\mathbb{R}^d)$, $g\in L^q(\mathbb{R}^d)$, the Young inequality yields
\[
(f\ast g)(x) := \int_{\mathbb{R}^d} f(x-y)\,g(y)\,\mathrm{d} y \in L^r(\mathbb{R}^d),
\qquad
\|f\ast g\|_{L^r}\le \|f\|_{L^p}\|g\|_{L^q}.
\]
Given measurable functions $\varphi:\mathbb{R}\to\mathbb{R}$, $\omega:\mathbb{R}^d\to\mathbb{R}$, and $f:\mathbb{R}\times \mathbb{R}^d\to\mathbb{R}$, we define temporal and spatial mollification by
\[
(f\ast\varphi)(t,x):=\int_{\mathbb{R}} f(s,x)\,\varphi(t-s)\,\mathrm{d} s,
\qquad
(f\ast\omega)(t,x):=\int_{\mathbb{R}^d} f(t,y)\,\omega(x-y)\,\mathrm{d} y.
\]
The context will always make clear which variable is being mollified.
For measurable sets $E\subseteq\mathbb{R}^d$ we denote the indicator function by $\mathds{1}_E$. For $f\in L^1(\mathbb{R}^d)+L^2(\mathbb{R}^d)$ we use the Fourier transform
\[
\mathcal{F}(f)(\xi)=\,\widehat f(\xi):=\int_{\mathbb{R}^d} e^{-2\pi i x\cdot\xi}\,f(x)\,\mathrm{d} x.
\]
For a multiindex $\alpha\in\mathbb{N}^d$ and a test function $\psi\in C_c^\infty(\mathbb{R}^d)$,
\[
\partial^\alpha\psi
=\partial_{x_1}^{\alpha_1}\cdots\partial_{x_d}^{\alpha_d}\psi,
\qquad
(y^\alpha\psi)(x)=x_1^{\alpha_1}\cdots x_d^{\alpha_d}\psi(x),
\]
and we write
\[
\alpha!:=\alpha_1!\cdots \alpha_d!,
\qquad
|\alpha|:=\alpha_1+\cdots+\alpha_d.
\]
Throughout the paper we use the following interval notation:
\[
\begin{array}{c|ccc}
    d & \mathcal{I}_d & \mathcal{J}_d & \mathcal{K}_d\\ \hline
    1 & [1,\infty] & [1,3] & [6,\infty)\\
    2 & [1,\infty) & [1,3) & (6,\infty)\\
    \ge 3 & [1,\tfrac{2d}{\,d-2\,}] & [1,2+\tfrac{2}{d}] & [2d+2,\infty)
\end{array}
\]
Given two values $a,b\in\mathbb{R}$ we use as shorthand $a\vee b=\max\{a,b\}$. We fix a smooth cutoff $\chi\in C_c^\infty(\mathbb{R}^d)$ such that $\chi=1$ on $B_1$, $\chi=0$ outside $B_2$, and $0\le \chi\le 1$.
For $R>0$ we then define the rescaled cutoff
\begin{equation}\label{eq:definition_cutoff_rescaled}
    \chi_R(x):=\chi(x/R).
\end{equation}


\section{Uniform estimates}
\label{sec:unif_est}
In this section we establish estimates that will be used later. Reader should note that some estimates listed here are only proved for one of the equations \eqref{eq:mf_svgd} and \eqref{eq:mf_slsi_svgd}.
Throughout this paper we assume that $\sigma\in (0,1)$ and we define  $\omega_\sigma(x):=\sigma^{-d}(x/\sigma)$. The following technical lemma will be a crucial part for proving the a priori estimates and the commutator arguments for Theorem \ref{thm:ex_weak_sol_loc_pde}.

\begin{lem}
    \label{lem:h-1_mod_bound}
    Let \eqref{ass:fourier} hold. Let $f\in L^2(0,T; H^{-1}(\mathbb{R}^{d}))$ and $\chi\in C_c^\infty(\mathbb{R}^{d})$. Then, we have that $f\ast\omega_\sigma \in L^2([0,T]\times \mathbb{R}^{d})$ and there exists a constant $C>0$ independent of $\sigma$ and $f$ (but dependent on $\chi$) such that
    \begin{align}
        \label{eq:pert_map_is_bnd}
        \|(f\chi)\ast\omega_\sigma\|_{L_{t,x}^2}\leq C \|f\ast\omega_\sigma\|_{L_{t,x}^2}.
    \end{align}
    Moreover, we have
    \begin{align}
        \label{eq:h-1_comm}
        \lim_{\sigma\to 0}\;\sup_{\substack{f\in L^2(0,T;H^{-1}(\mathbb{R}^{d})) \\ \|f\ast\omega_\sigma\|_{L_{t,x}^2}\leq 1}} \|(f\chi)\ast\omega_\sigma-\chi\,(f\ast\omega_\sigma)\|_{L_{t,x}^2}=0.
    \end{align}
\end{lem}
\begin{rem}The convergence statement \eqref{eq:h-1_comm} is very delicate. One might try to obtain this result by writing
    $$
    (f\chi)\ast\omega_\sigma-\chi\,(f\ast\omega_\sigma) = \int_{\mathbb{R}^{d}} f(y)\,(\chi(y) - \chi(x))\, \omega_{\sigma}(x-y) \,\mathrm{d} y.
    $$
    However, in order to further estimate this term, one would need to impose additional sign assumptions on $f$ and $\omega_{\sigma}$. This does not seem reasonable as $f$ is only a distribution.
\end{rem}
\begin{proof}[Proof of Lemma \ref{lem:h-1_mod_bound}]
    By using the Plancherel Theorem and the convolution formula, it is sufficient to prove \eqref{eq:pert_map_is_bnd} and \eqref{eq:h-1_comm} in the Fourier space. This means we must show
    \begin{align*}
        &\|(\hat{f}\ast\hat{\chi})\,\hat{\omega}_\sigma\|_{L_{t,\xi}^2}\leq C\,\|\hat{f}\,\hat{\omega}_\sigma\|_{L_{t,\xi}^2}\\
        &\lim_{\sigma\to 0}\;\sup_{\substack{f\in L^2(0,T;H^{-1}(\mathbb{R}^{d})) \\ \|\hat{f}\hat{\omega}_\sigma\|_{L_{t,\xi}^2}\leq 1}} \|(\hat{f}\ast\hat{\chi})\,\hat{\omega}_\sigma-(\hat{f}\hat{\omega}_\sigma)\ast\hat{\chi}\|_{L_{t,\xi}^2}=0.
    \end{align*}
    We may substitute $\hat{f}\, \hat{\omega}_\sigma$ by $g\in L^2([0,T]\times \mathbb{R}^{d})$, as $\hat{f}\,\hat{\omega}_\sigma\in L^2([0,T]\times \mathbb{R}^{d})$ by our assumptions. Thus, we can rephrase the statements \eqref{eq:pert_map_is_bnd} and \eqref{eq:h-1_comm} as
    \begin{align*}
        &\|((g/\hat{\omega_\sigma})\ast\hat{\chi})\,\hat{\omega}_\sigma\|_{L_{t,\xi}^2}\leq C\,\|g\|_{L_{t,\xi}^2},\\
        &\lim_{\sigma\to 0}\;\sup_{\substack{g\in L^2([0,T]\times \mathbb{R}^{d}) \\ \|g\|_{L_{t,\xi}^2}\leq 1}} \|((g/\hat{\omega}_\sigma)\ast\hat{\chi})\,\hat{\omega}_\sigma-g\ast\hat{\chi}\|_{L_{t,\xi}^2}=0,
    \end{align*}
    or equivalently in terms of linear operators
    \begin{align*}
        &\|T_\sigma g\|_{L_{t,\xi}^2}\leq C\,\|g\|_{L_{t,\xi}^2},\\
        &\| T_\sigma-T_0\|_{\mathrm{op}}\rightarrow 0\text{ as }\sigma\rightarrow 0\text{ as a linear operator on }L^2([0,T]\times \mathbb{R}^{d}),
    \end{align*}
    where we define the operators on $L^2([0,T]\times \mathbb{R}^{d})$
    \begin{align*}
        T_\sigma g(\xi)
        :=\big(((g/\hat{\omega}_\sigma)\ast\hat{\chi})\,\hat{\omega}_\sigma\big)(\xi)=\int_{\mathbb{R}^{d}}K_\sigma(\xi,\eta)\,g(\eta)\,\mathrm{d}\eta,\quad
        T_0 g:=g\ast\hat{\chi},
    \end{align*}
    with the kernel defined by
    \begin{align*}
        K_\sigma(\xi,\eta):=\frac{\hat{\omega}_\sigma(\xi)}{\hat{\omega}_\sigma(\eta)}\hat{\chi}(\xi-\eta).
    \end{align*}
    
    \underline{Step 1: $T_\sigma$ is a bounded operator.}
    Using $|\eta|^2\leq 2|\xi-\eta|^2+2|\xi|^2$ we find
    \begin{align}
        \label{eq:bnd_ration_xi_eta}
        \sqrt{\frac{1+\sigma^2|\eta|^2}{1+\sigma^2|\xi|^2}}
        \leq \sqrt{\frac{1+2\sigma^2|\xi|^2+2\sigma^2|\xi-\eta|^2}{1+\sigma^2|\xi|^2}}
        \leq \sqrt{2}\,\sqrt{1+\sigma^2|\xi-\eta|^2}.
    \end{align}
    Hence, we discover the uniform upper bound
    \begin{align*}
        &\left|\frac{\hat{\omega}_\sigma(\xi)}{\hat{\omega}_\sigma(\eta)}\hat{\chi}(\xi-\eta)\right|
        \leq \mathcal{D}_0\mathcal{D}_1\sqrt{\frac{1+\sigma^2|\eta|^2}{1+\sigma^2|\xi|^2}}\,|\hat{\chi}(\xi-\eta)|
        \leq \sqrt{2}\,\mathcal{D}_0\mathcal{D}_1\sqrt{1+\sigma^2|\xi-\eta|^2}\,|\hat{\chi}(\xi-\eta)|.
    \end{align*}
    We now estimate $T_{\sigma} g$ by using the same idea as in the Schur's test (see for instance \cite[Theorem 5.2]{halmos_sunder_2012_bounded}). We estimate
    \begin{align}
        \label{eq:TK_bounded_op_1}
        \int_{\mathbb{R}^{d}} |K_\sigma(\xi,\eta)|\, |g(\eta)| \,\mathrm{d} \eta
        &\leq \left(\int_{\mathbb{R}^{d}} |K_\sigma(\xi,\eta)| \,\mathrm{d} \eta \right)^{1/2}\, \left(\int_{\mathbb{R}^{d}} |K_\sigma(\xi,\eta)|\, |g(\eta)|^2 \,\mathrm{d} \eta \right)^{1/2}.
    \end{align}
    Then, by using Fubini's Theorem, and the fact that $\mathcal{\chi}$ is a Schwartz function, the $L^2$ norm is controlled by
    \begin{align}
        \label{eq:TK_bounded_op_2}
        \begin{split}
            \|T_{\sigma}g\|^2_{L^2_{\xi}}
            &\leq \left(\sup_{\xi \in \mathbb{R}^{d}} \int_{\mathbb{R}^{d}} |K_{\sigma}(\xi,\eta)| \,\mathrm{d} \eta \right) \left(\sup_{\eta \in \mathbb{R}^{d}} \int_{\mathbb{R}^{d}} |K_{\sigma}(\xi,\eta)| \,\mathrm{d} \xi \right) \|g\|^2_{L^2_{\xi}}\\
            &\leq 2(\mathcal{D}_0\mathcal{D}_1)^2\left(\int_{\mathbb{R}^{d}}\sqrt{1+\sigma^2|\xi|^2}\,|\hat{\chi}(\xi)|\,\mathrm{d} \xi\right)^2
            \|g\|^2_{L^2_{\xi}}.
        \end{split}
    \end{align}
    Integrating in time yields the desired boundedness \eqref{eq:pert_map_is_bnd}.
    
    \underline{Step 2: $T_\sigma-T_0$ converges to the null-operator.} We use again the same boundedness idea in \eqref{eq:TK_bounded_op_1} and \eqref{eq:TK_bounded_op_2}. This means it suffices to prove the convergence
    \begin{equation}\label{eq:target_operators_norm_convergence_a_la_Schur_test}
        \limsup_{\sigma\rightarrow 0} \;\sup_{\xi\in\mathbb{R}^{d}} \int_{\mathbb{R}^{d}}|K_\sigma(\xi,\eta)-\hat{\chi}(\xi-\eta)|\,\mathrm{d} \eta
        =0,\quad
        \limsup_{\sigma\rightarrow 0} \;\sup_{\eta\in\mathbb{R}^{d}} \int_{\mathbb{R}^{d}}|K_\sigma(\xi,\eta)-\hat{\chi}(\xi-\eta)|\,\mathrm{d} \xi
        =0.
    \end{equation}
    To that end, fix any $\varepsilon>0$. Since $\hat{\chi}$ is Schwartz, choose $R>0$ such that
    \begin{align*}
        \int_{|z|>R}\sqrt{1+|z|^2}\,|\hat{\chi}(z)|\,\mathrm{d} z<\varepsilon.
    \end{align*}
    
    \underline{Step 2.1: Far from the diagonal.} We use the same estimate \eqref{eq:bnd_ration_xi_eta}, then we discover the uniform upper bound
    \begin{align*}
        &\left|\left(\frac{\hat{\omega}_\sigma(\xi)}{\hat{\omega}_\sigma(\eta)}-1\right)\hat{\chi}(\xi-\eta)\right|
        \leq \left(1+\frac{\hat{\omega}_\sigma(\xi)}{\hat{\omega}_\sigma(\eta)}\right)\,|\hat{\chi}(\xi-\eta)|\\
        &\leq \left(1+\mathcal{D}_0\mathcal{D}_1\sqrt{\frac{1+\sigma^2|\eta|^2}{1+\sigma^2|\xi|^2}}\right)\,|\hat{\chi}(\xi-\eta)|
        \leq C\sqrt{1+|\xi-\eta|^2}\,|\hat{\chi}(\xi-\eta)|,
    \end{align*}
    for some constant $C>0$ independent of $\sigma$ and $\varepsilon$ (recalling that $\sigma<1$). Taking the integral now yields
    \begin{align*}
        &\int_{|\xi-\eta|>R} \left|\left(\frac{\hat{\omega}_\sigma(\xi)}{\hat{\omega}_\sigma(\eta)}-1\right)\hat{\chi}(\xi-\eta)\right|\,\mathrm{d} \xi
        \leq C\int_{|\xi-\eta|>R}\sqrt{1+|\xi-\eta|^2}\,|\hat{\chi}(\xi-\eta)|\,\mathrm{d}\xi\\
        &\leq C\int_{|z|>R}\sqrt{1+|z|^2}\,|\hat{\chi}(z)|\,\mathrm{d} z
        \leq C\varepsilon.
    \end{align*}
    The same bound holds if we integrate in $\eta$ instead.
    
    \underline{Step 2.2: Near the diagonal.} Take any $\xi,\eta\in\mathbb{R}^{d}$ such that $|\xi-\eta|\leq R$. Let $\theta$ be a modulus of continuity for $\zeta$, this means $|\zeta(\xi)-\zeta(\eta)|\leq\theta(|\xi-\eta|)$, $\theta$ is non-decreasing and $\theta(t)\rightarrow 0$ as $t\rightarrow 0$. Then, we estimate
    \begin{align*}
        \left|\frac{\zeta(\sigma \xi)}{\zeta(\sigma \eta)}-1\right|
        =\frac{|\zeta(\sigma \xi)-\zeta(\sigma \eta)|}{\zeta(\sigma \eta)}
        \leq \mathcal{D}_0\,\theta(\sigma |\xi-\eta|)
        \leq \mathcal{D}_0\,\theta(\sigma R).
    \end{align*}
    Define the function $\phi(t)=\tfrac{1}{2}\ln(1+|t|^2)$. Since $|\phi'(t)|\le\tfrac{1}{2}$, we estimate
    \begin{align*}
        \left|\ln\frac{\sqrt{1+\sigma^2|\eta|^2}}{\sqrt{1+\sigma^2|\xi|^2}}\right|
        =|\phi(\sigma\eta)-\phi(\sigma\xi)|
        \leq \frac{\sigma}{2}|\xi-\eta|
        \leq \frac{\sigma R}{2}.
    \end{align*}
    We use this to derive two bounds (recalling that $\sigma<1$ and $e^{x}$ is $e^\alpha$-Lipschitz on $(-\infty,\alpha]$)
    \begin{align*}
        \left|\frac{\sqrt{1+\sigma^2|\eta|^2}}{\sqrt{1+\sigma^2|\xi|^2}}-1\right|
        \leq \left|e^{\ln\frac{\sqrt{1+\sigma^2|\eta|^2}}{\sqrt{1+\sigma^2|\xi|^2}}} - e^{0}
        \right|
        \leq  \frac{\sigma R}{2}\,e^{\frac{R}{2}},
        \qquad
        \left|\frac{\sqrt{1+\sigma^2|\eta|^2}}{\sqrt{1+\sigma^2|\xi|^2}}\right|
        \leq e^{\frac{\sigma R}{2}}
        \leq e^{\frac{R}{2}}.
    \end{align*}
    Thus we can bound the error close to the diagonal by
    \begin{align*}
        &\left|\frac{\hat{\omega}_\sigma(\xi)}{\hat{\omega}_\sigma(\eta)}-1\right|\,|\hat{\chi}(\xi-\eta)|
        =\left|\left(\frac{\zeta(\sigma \xi)}{\zeta(\sigma\eta)}-1\right)\frac{\sqrt{1+\sigma^2|\eta|^2}}{\sqrt{1+\sigma^2|\xi|^2}}+\left(\frac{\sqrt{1+\sigma^2|\eta|^2}}{\sqrt{1+\sigma^2|\xi|^2}}-1\right)\right|\,\|\hat{\chi}\|_\infty \\
        &\leq \left(\mathcal{D}_0\,\theta(\sigma R)+\frac{\sigma R}{2}\right)e^{\frac{R}{2}}\,\|\hat{\chi}\|_\infty .
    \end{align*}
    \underline{Step 3: Final estimate.} By combining the bound away from the diagonal and near the diagonal, we compute the integral
    \begin{align*}
        &\limsup_{\sigma\rightarrow 0}\;\sup_{\eta\in\mathbb{R}^d}\int_{\mathbb{R}^{d}}\left|\left(\frac{\hat{\omega}_\sigma(\xi)}{\hat{\omega}_\sigma(\eta)}-1\right)\hat{\chi}(\xi-\eta)\right|\,\mathrm{d} \xi\\
        &\leq \limsup_{\sigma\rightarrow 0}\;\sup_{\eta\in\mathbb{R}^d}\int_{|\xi-\eta|>R}\left|\left(\frac{\hat{\omega}_\sigma(\xi)}{\hat{\omega}_\sigma(\eta)}-1\right)\hat{\chi}(\xi-\eta)\right|\,\mathrm{d} \xi
        +\int_{|\xi-\eta|\leq R}\left|\left(\frac{\hat{\omega}_\sigma(\xi)}{\hat{\omega}_\sigma(\eta)}-1\right)\hat{\chi}(\xi-\eta)\right|\,\mathrm{d} \xi\\
        &\leq \limsup_{\sigma\rightarrow 0}\;C\varepsilon+C\,R^d\left(\mathcal{D}_0\,\theta(\sigma R)+\frac{\sigma R}{2}\right)e^{\frac{R}{2}}\,\|\hat{\chi}\|_\infty
        =C\varepsilon,
    \end{align*}
    for some constant $C>0$ independent of $\sigma$ and $\varepsilon$. The same bound holds if we integrate in $\eta$ instead of $\xi$. As $\varepsilon$ can be chosen as small as desired, we have proven \eqref{eq:target_operators_norm_convergence_a_la_Schur_test}.
\end{proof}

\begin{lem}[Basic gradient flow estimates for \eqref{eq:mf_svgd} and \eqref{eq:mf_slsi_svgd}]
    \label{lem:gf_est} Let Assumption \ref{assumptions:slsi} hold. Let $\{\rho^\sigma\}_\sigma$ be the weak solutions to \eqref{eq:mf_slsi_svgd}. Then, for all $T>0$ the following assertions hold:
    \begin{enumerate}[label=(GF\arabic*)]
        \item\label{est_gf_kl} $\{\rho^\sigma\}_\sigma$, $\{\rho^\sigma\,V\}_\sigma$, $\{\rho^\sigma\,|\log\rho^\sigma|\}_\sigma$ are uniformly bounded in $L^{\infty}(0,T; L^1(\mathbb{R}^d))$;
        \item\label{est_gf_dis} $\{((\nabla \rho^\sigma + \rho^\sigma \, \nabla V)e^{V-\frac{\mathbb{V}}{2}})\ast \omega_\sigma \}_\sigma$ is uniformly bounded in $L^2((0,T)\times \mathbb{R}^d)$.
    \end{enumerate}
    Analogue assertions hold for $\{\varrho^\sigma\}_\sigma$ solving \eqref{eq:mf_svgd} under Assumption \ref{assumptions:pot} if the sequence in \ref{est_gf_dis} is replaced with $\{(\nabla \varrho^\sigma + \varrho^\sigma \, \nabla V)\ast \omega_\sigma \}_\sigma$.
\end{lem}
\begin{proof}
    We focus on the proof of $\{\rho^\sigma\}_\sigma$ and the proof for $\{\varrho^\sigma\}_\sigma$ is similar. We show that the weak solution $\rho^{\sigma}$ satisfies
    \begin{align}
        \label{eq:entropy_identity}
        \mathrm{KL}(\rho_T^\sigma\,||\, \rho_\infty) + \int_0^T \int_{\mathbb{R}^{d}} |((\nabla \rho_t^\sigma + \rho_t^\sigma \, \nabla V)e^{V-\frac{\mathbb{V}}{2}})\ast\omega_\sigma|^2 \,\mathrm{d} x\,\mathrm{d} t \leq \mathrm{KL}(\rho_0 \,||\,\rho_\infty).
    \end{align}
    The identity \eqref{eq:entropy_identity} follows by multiplying \eqref{eq:mf_slsi_svgd} by $\log(\rho^\sigma e^{V})$, exploiting $k_{\sigma}=\omega_{\sigma}\ast\omega_{\sigma}$ and using integration by parts. Then, \eqref{eq:entropy_identity} together with Lemma \ref{lem:control_neg_log} implies bounds \ref{est_gf_kl} and \ref{est_gf_dis}.
\end{proof}

In the following lemma we derive estimates specifically for the weak solutions $\{\rho^\sigma\}_\sigma$ of \eqref{eq:mf_slsi_svgd}. These bounds derive from the Stein-log-Sobolev inequality \eqref{eq:slsi}.

\begin{lem}[Estimates for solutions of \eqref{eq:mf_slsi_svgd}]
    \label{lem:estimates_slsi}
    Let Assumption \ref{assumptions:slsi} hold and let $\{\rho^\sigma\}_\sigma$ be the weak solution to \eqref{eq:mf_slsi_svgd}. Then, for any fixed $T,R>0$ the following assertions are true:
    \begin{enumerate}[label=(S\arabic*)]
        \item \label{bnd:rho}
        $\{\rho^\sigma e^{V-\frac{\mathbb{V}}{2}}\}_\sigma$ is uniformly bounded in $L^2([0,T]\times \mathbb{R}^{d})$;
        \item \label{bnd:h-1} $\{(\nabla\rho^\sigma +\rho^\sigma\nabla V)e^{V-\frac{\mathbb{V}}{2}}\}_\sigma$ is uniformly bounded in $L^2(0,T;\, H^{-1}(\mathbb{R}^{d}))$;
        \item \label{bnd:h1}
        $\{(\rho^{\sigma} e^{V-\frac{\mathbb{V}}{2}}\chi_R)\ast\omega_\sigma\}_\sigma$ is uniformly bounded in $L^2(0,T;\, H^1(\mathbb{R}^{d}))$, where $\chi_R$ is defined in \eqref{eq:definition_cutoff_rescaled}, with the bound depending on $R>0$.
    \end{enumerate}
\end{lem}
\begin{proof}
    We first observe that, by assumption \eqref{ass:fourier}, we have for $\sigma\in(0,1)$
    \begin{equation}\label{eq:estimate_between_generator_h-1_and_square_root_of_kernel}
        \frac{1}{1+|\xi|^2}  \leq \frac{1}{1+\sigma^2\, |\xi|^2} \leq \mathcal{D}_0^2\,|\widehat{\omega}_{\sigma}|^2.
    \end{equation}
    To see \ref{bnd:h-1}, we let $f^\sigma = \nabla(\rho^{\sigma} e^V)e^{-\frac{\mathbb{V}}{2}}$, change to Fourier variables and apply \eqref{eq:estimate_between_generator_h-1_and_square_root_of_kernel} to obtain
    \begin{align*}
        \|f^\sigma\|_{H^{-1}_x}^2 = \int_{\mathbb{R}^{d}} \frac{|\widehat{f^\sigma}(\xi)|^2}{1+|\xi|^2} \,\mathrm{d} \xi \leq \mathcal{D}_0^2 \int_{\mathbb{R}^{d}} |\widehat{f^\sigma}(\xi)|^2 \, |\widehat{\omega}_{\sigma}|^2  \,\mathrm{d} \xi ,
    \end{align*}
    which is bounded in $L^2(0,T)$ by \ref{est_gf_dis}. To see \ref{bnd:rho}, we note that \cite[p. 19-20]{carrillo_skrzeczkowski_warnett_2024_stein} gives a constant $C$ such that for all probability measures $g:\mathbb{R}^{d} \to \mathbb{R}$ satisfying $g\, e^{V-\frac{\mathbb{V}}{2}} \in L^2(\mathbb{R}^{d})$, $\nabla(g e^V)e^{-\frac{\mathbb{V}}{2}} \in H^{-1}(\mathbb{R}^{d})$ we have
    $$
    \| g\, e^{V-\frac{\mathbb{V}}{2}} \|^2_{L^2_x} \leq C\left(1 + \int_{\mathbb{R}^{d}} \frac{|\mathcal{F}(\nabla(g\, e^V)e^{-\frac{\mathbb{V}}{2}})(\xi)|^2}{1+|\xi|^2} \,\mathrm{d} \xi \right).
    $$
    Applying this inequality with $g = \rho^{\sigma}$ we arrive at \ref{bnd:rho} by applying \ref{bnd:h-1}. Next, for the bound \ref{bnd:h1} we first observe that the sequence $\{(\rho^\sigma e^{V-\frac{\mathbb{V}}{2}}\chi_R)\ast\omega_\sigma\}_\sigma$ is uniformly bounded in $L^2((0,T)\times \mathbb{R}^{d})$ by \ref{bnd:rho} and the Young convolution inequality. Next, for the gradient we compute the identity
    \begin{align}
        \label{eq:grad_rho}
        \nabla ((\rho^\sigma e^{V-\frac{\mathbb{V}}{2}}\chi_R)\ast\omega_\sigma)
        =((\nabla \rho^\sigma +\rho^\sigma\nabla V)e^{V-\frac{\mathbb{V}}{2}}\chi_R)\ast \omega_\sigma+(\rho^\sigma e^{V-\frac{\mathbb{V}}{2}}\nabla(\chi_R e^{-\frac{\mathbb{V}}{2}})e^{\frac{\mathbb{V}}{2}})\ast\omega_\sigma.
    \end{align}
    For the first term on the RHS of \eqref{eq:grad_rho}, there exists by Lemma \ref{lem:h-1_mod_bound} a constant $C>0$ independent of $\sigma$ such that the following inequality holds
    \begin{align*}
        \|((\nabla \rho^\sigma +\rho^\sigma\nabla V)e^{V-\frac{\mathbb{V}}{2}}\chi_R)\ast \omega_\sigma\|_{L_{t,x}^2}
        \leq C \|((\nabla \rho^\sigma +\rho^\sigma\nabla V)e^{V-\frac{\mathbb{V}}{2}})\ast \omega_\sigma\|_{L_{t,x}^2}.
    \end{align*}
    We observe that the RHS is uniformly bounded by \ref{est_gf_dis} of Lemma \ref{lem:gf_est}. For the second term on the RHS of \eqref{eq:grad_rho}, we observe that $\nabla(\chi_R e^{-\frac{\mathbb{V}}{2}})e^{\frac{\mathbb{V}}{2}}\in L^\infty(\mathbb{R}^{d})$ so by using the Young convolution inequality and \ref{bnd:rho}, we discover that $\{(\rho^\sigma e^{V-\frac{\mathbb{V}}{2}}\nabla(\chi_R\, e^{-\frac{\mathbb{V}}{2}})e^{\frac{\mathbb{V}}{2}})\ast\omega_\sigma\}_\sigma$ is uniformly bounded in $L^2([0,T]\times\mathbb{R}^{d})$. This concludes the proof of \ref{bnd:h1}.
\end{proof}

\begin{rem}
    We have used Lemma \ref{lem:h-1_mod_bound} in order to derive a $H^1$ bound on the sequence $\{(\rho^{\sigma} e^{V-\frac{\mathbb{V}}{2}}\chi_R)\ast\omega_\sigma\}_\sigma$ so that we can obtain a strongly convergent subsequence in Lemma~\ref{conv:rho_slsi}.
\end{rem}

\begin{lem}[Estimates for solutions of \eqref{eq:mf_svgd}]\label{prop:bounds_gradient}
    Let Assumption \ref{assumptions:pot} hold and let $\{\varrho^\sigma\}_\sigma$ be the weak solution to \eqref{eq:mf_svgd}. Then, for any fixed $T,R>0$, the sequence $\{\varrho^\sigma\ast \omega_\sigma\}_\sigma$ is uniformly bounded in $L^2(0,T; H^1(B_R))$.
\end{lem}
\begin{proof}
    We will employ the estimate in \ref{est_gf_dis} of Lemma \ref{lem:gf_est} and control the sequence $\{(\varrho^\sigma\, \nabla V) \ast \omega_\sigma\}_\sigma$. Our reasoning is divided into three steps.
    
    \smallskip
    
    \underline{Step 1: Pointwise estimate on $\{(\varrho^\sigma\, \nabla V) \ast \omega_\sigma\}$.}
    We first show that there exists a constant $C>0$ independent of $\sigma$ such that
    \begin{align}
        |(\varrho^\sigma\, \nabla V) \ast \omega_\sigma|
        \leq C(1+|x|^{p_0^\ast})\,(\varrho^\sigma\ast\omega_\sigma)+C\, (\varrho^\sigma\ast(|y|^{p_0^\ast\vee m_0}\omega_\sigma)). \label{eq:bound_varrho_nabla_V_omega}
    \end{align}
    To that end, we first employ \eqref{eq:bound_nabla_V_by_V} from Assumption \ref{assumptions:pot} to derive the bound
    \begin{align*}
        |(\varrho^\sigma\, \nabla V) \ast \omega_\sigma|
        \leq (\varrho^\sigma\, [1+|\nabla V|^{p_0}]) \ast \omega_\sigma
        \leq (1+C_V)\,(\varrho^\sigma \ast \omega_\sigma)
        +C_V\,(\varrho^\sigma  V)\ast \omega_\sigma.
    \end{align*}
    Next, by Remark \ref{rem:poly_growth_V} and Assumption \ref{assumptions:pot}, we use the polynomial upper bound \eqref{eq:V_poly_gorwth} to further estimate the second term on the RHS by
    \begin{align*}
        (\varrho^\sigma  V)\ast \omega_\sigma
        \leq C_0\,(\varrho^\sigma\ast \omega_\sigma)
        +C_0\, (\varrho^\sigma |x|^{p_0^\ast})\ast \omega_\sigma.
    \end{align*}
    We use the bound $|x|^{p_0^\ast}\leq 2^{p_0^\ast-1}|x-y|^{p_0^\ast}+2^{p_0^\ast-1}|y|^{p_0^\ast}$ to further bound the second term on the RHS by
    \begin{align*}
        (\varrho^\sigma |x|^{p_0^\ast})\ast \omega_\sigma
        \leq 2^{p_0^\ast-1}\,|x|^{p_0^\ast}\,(\varrho^\sigma \ast \omega_\sigma)+2^{p_0^\ast-1}\,\varrho^\sigma\ast (|y|^{p_0^\ast}\omega_\sigma).
    \end{align*}
    In the last step we need to bound the second term of the RHS again by
    \begin{align*}
        \varrho^\sigma\ast (|y|^{p_0^\ast}\omega_\sigma)
        \leq \varrho^\sigma\ast ((1+|y|^{p_0^\ast\vee m_0})\,\omega_\sigma).
    \end{align*}
    By combining this chain of inequalities, we discover that \eqref{eq:bound_varrho_nabla_V_omega} holds true. We remark that all the above steps require $\omega_{\sigma}\geq 0$.
    \smallskip
    
    \underline{Step 2: Application of pointwise estimate.} We combine the triangle inequality and \eqref{eq:bound_varrho_nabla_V_omega} to derive the estimate
    \begin{align*}
        &\| \nabla \varrho^\sigma \ast \omega_\sigma  \|_{L^2((0,T)\times B_R)}\\ &\leq \| (\nabla \varrho^\sigma + \varrho^\sigma\, \nabla V)\ast\omega_\sigma \|_{L^2((0,T)\times B_R)}
        +C\,\| \varrho^\sigma \ast (|x|^{p_0^\ast\vee m_0}\omega_\sigma)  \|_{L^2((0,T)\times B_R)}  \\
        &\phantom{\leq}+ C\,\| (1+|x|^{p_0^\ast})\,(\varrho^\sigma\ast\omega_\sigma) \|_{L^2((0,T)\times B_R)}.
    \end{align*}
    The first two terms on the RHS are bounded by \ref{est_gf_dis} of Lemma \ref{lem:gf_est} and Lemma \ref{lem:conv_mollifiers_scaled_x} together with assumption \eqref{eq:pot_omega}. For the last one we apply Lemma \ref{lem:interpolation} to find, for some $\delta>0$ to be chosen later, constants $C,C_\delta>0$ independent of $\sigma$ such that
    \begin{align*}
        &\| \nabla \varrho^\sigma \ast \omega_\sigma  \|_{L^2((0,T)\times B_R)}
        \leq C+C\,\| \varrho^\sigma \ast \omega_\sigma  \|_{L^2((0,T)\times B_R)}\\
        &\leq C+C \delta\,\| \nabla\varrho^\sigma \ast \omega_\sigma  \|_{L^2((0,T)\times B_R)}
        +C\,C_\delta\, \| \varrho^\sigma \ast \omega_\sigma  \|_{L^2(0,T;L^1( B_R))}.
    \end{align*}
    Observe that by \ref{est_gf_kl} of Lemma \ref{lem:gf_est} we know that $\{\varrho^\sigma \ast \omega_\sigma\}_\sigma$ is uniformly bounded in $L^2(0,T;L^1( B_R))$. Then, choosing $\delta<1/C$, we conclude the proof.
\end{proof}

\begin{lem}[A posteriori estimates for \eqref{eq:mf_svgd} and \eqref{eq:mf_slsi_svgd}]
    \label{lem:ref_a_priori_est} Let Assumption \ref{assumptions:slsi} hold and let $\{\rho^\sigma\}_\sigma$ be the weak solutions to \eqref{eq:mf_slsi_svgd}. Then, for all fixed $q\in\mathcal{I}_d$ and $r\in\mathcal{J}_d$ the following assertions hold:
    \begin{enumerate}[label=(AP\arabic*)]
        \item \label{bnd:rap_int} $\{(\rho^\sigma e^{V-\frac{\mathbb{V}}{2}}\chi_R)\ast\omega_\sigma\}_\sigma$ is uniformly bounded in $L^\infty(0,T;\, L^1(B_R))\cap L^2(0,T;\, L^q(B_R))$,
        \item \label{bnd:rap_lp} $\{(\rho^\sigma e^{V-\frac{\mathbb{V}}{2}}\chi_R)\ast\omega_\sigma\}_\sigma$ is uniformly bounded in $L^r((0,T)\times B_R)$,
    \end{enumerate}
    where $\chi_R$ is defined in \eqref{eq:definition_cutoff_rescaled}. Analogue assertions hold under Assumption \ref{assumptions:pot} for $\{\varrho^\sigma\}_\sigma$ the weak solutions of \eqref{eq:mf_svgd}, if we replace $\{(\rho^\sigma e^{V-\frac{\mathbb{V}}{2}}\chi_R)\ast\omega_\sigma\}_\sigma$ with $\{\varrho^\sigma\ast\omega_\sigma\}_\sigma$.
\end{lem}
\begin{proof}
    The proofs for $\{\rho^\sigma\}_\sigma$ and $\{\varrho^\sigma\}_\sigma$ are analogue so we focus on the former and briefly point out the differences for the latter. First, we prove \ref{bnd:rap_int}. We use \ref{bnd:h1} of Lemma~\ref{lem:estimates_slsi} (respectively we use Lemma \ref{prop:bounds_gradient}) and observe that the Sobolev embedding theorem guarantees the continuous embedding $H^1(B_R)\hookrightarrow L^q(B_R)$ so that $\{(\rho^\sigma e^{V-\frac{\mathbb{V}}{2}})\ast\omega_\sigma\}_\sigma$ is uniformly bounded in $L^2(0,T;\, L^q(B_R))$. Then, as $\chi_R$ is compactly supported we can use the Young convolution inequality and a change of variables to derive
    \begin{align*}
        \|(\rho^\sigma e^{V-\frac{\mathbb{V}}{2}}\chi_R)\ast\omega_\sigma\|_{L_t^\infty L_x^1}
        \leq \|e^{V-\frac{\mathbb{V}}{2}}\chi_R\|_{L^\infty}\,\|\omega\|_{L^1}\,
        \|\rho^\sigma\|_{L_t^\infty L_x^1},
    \end{align*}
    which concludes the proof of \ref{bnd:rap_int}. In order to prove \ref{bnd:rap_lp}, we perform the interpolation between the bounds in $L^\infty(0,T;\, L^1(B_R))$ and $ L^2(0,T;\, L^q(B_R))$, obtained in \ref{bnd:rap_int}. We obtain that $\{(\rho^\sigma e^{V-\frac{\mathbb{V}}{2}}\chi_R)\ast\omega_\sigma\}_\sigma$ is uniformly bounded in $L^{p_\theta}((0,T)\times B_R)$ whenever
    \begin{align*}
        \frac{1}{p_\theta}
        =\frac{\theta}{2}+\frac{1-\theta}{\infty}
        =\frac{\theta}{q}+\frac{1-\theta}{1}
        \text{ and }\theta\in [0,1].
    \end{align*}
    By solving this equation, we obtain $p_{\theta}= 3-\frac{2}{q}$, yielding that the set of admissible exponents $p_\theta$ is $\mathcal{J}_d$. With this we conclude the proof.
\end{proof}


\section{Commutator arguments}
\label{sec:comm}
In this section, we develop several commutator arguments. In the weak formulations of \eqref{eq:mf_svgd} and \eqref{eq:mf_slsi_svgd}, the test function $\varphi$ appears inside the convolution with the kernel $\omega_\sigma$. This prevents us from applying estimates that are local in space, since such estimates require the test function to be outside the convolution. To overcome this difficulty, we employ a~commutator argument that allows us to move the test function outside the convolution with $\omega_\sigma$. We first carry out this argument for the weak solutions $\{\rho^\sigma\}_\sigma$ of \eqref{eq:mf_slsi_svgd}, and then for the weak solutions $\{\varrho^\sigma\}_\sigma$ of \eqref{eq:mf_svgd}.

\begin{lem}\label{lem:comm_est_slsi}
    Let Assumption \ref{assumptions:slsi} hold, let $\{\rho^\sigma\}_\sigma$ be the weak solutions to \eqref{eq:mf_slsi_svgd}, let $q=\frac{p_0}{p_0-1}$ with $p_0$ defined in \ref{ass:four_mom}, and let $T,R>0$. Then,
    \begin{align}
        \label{eq:conv_com_slsi}
        \lim_{\sigma\to 0}\;\;
        \sup_{\substack{\varphi\in L^q(0,T;\, W_0^{1,\infty}(B_R)) \\ \|\varphi\|_{ L_t^q W_x^{1,\infty}}\leq 1}}\;\;
        \|(\rho^\sigma e^{V-\frac{\mathbb{V}}{2}}\varphi)\ast\omega_\sigma
        -\varphi\,((\rho^\sigma e^{V-\frac{\mathbb{V}}{2}}\chi_R)\ast\omega_\sigma)\|_{L_{t,x}^2}
        =0.
    \end{align}
\end{lem}
\begin{proof}
    First, we observe that as $\mathrm{supp}(\varphi)\subseteq [0,T]\times B_R$, we have that $\varphi=\varphi\,\chi_R$. We estimate pointwisely neglecting the variable $t$ for simplicity:
    \begin{align*}
        &|((\rho^\sigma e^{V-\frac{\mathbb{V}}{2}}\varphi)\ast\omega_\sigma)(x)
        -\varphi(x)\,((\rho^\sigma e^{V-\frac{\mathbb{V}}{2}}\chi_R)\ast\omega_\sigma)(x)|\\
        &\leq \int_{\mathbb{R}^{d}} \rho^\sigma(x-y)\, e^{V(x-y)-\frac{\mathbb{V}(x-y)}{2}}\chi_R(x-y)\;|\varphi(x-y)-\varphi(x)|\;|\omega_\sigma(y)| \,\mathrm{d} y\\
        &\leq
        \; ((\rho^\sigma e^{V-\frac{\mathbb{V}}{2}}\chi_R)\ast |y\,\omega_\sigma|)(x)\;\|\varphi\|_{W_x^{1,\infty}} .
    \end{align*}
    This means it suffices to show
    \begin{align}
        \label{eq:goal_comm_slsi}
        \lim_{\sigma\to 0}\;\; \sup_{\substack{\varphi\in L^q(0,T;\, W_0^{1,\infty}(B_R)) \\ \|\varphi\|_{ L_t^q W_x^{1,\infty}}\leq 1}}\;\; \|((\rho^\sigma e^{V-\frac{\mathbb{V}}{2}}\chi_R)\ast|y\,\omega_\sigma|)\;\|\varphi\|_{W_x^{1,\infty}}\|_{L_{t,x}^2}=0.
    \end{align}
    We know that $\rho^\sigma e^{V-\frac{\mathbb{V}}{2}}\chi_R\in L^\infty(0,T;\,L^1(\mathbb{R}^{d}))\cap L^2(0,T;\,L^2(\mathbb{R}^{d}))$ by \ref{bnd:rho} from Lemma \ref{lem:estimates_slsi} and as $\chi_R$ is compactly supported. Thus, we can use the interpolation theorem to get the inclusion \begin{align}
    \label{eq:interpolation_rs}
    \rho^\sigma e^{V-\frac{\mathbb{V}}{2}}\chi_R\in L^r(0,T;L^s(\mathbb R^d)),
\end{align}
where $s\in (1,2)$ and $r=\frac{s}{s-1}$. We choose $s$ such that
\begin{align}
    \label{eq:rel_sp}
    1+\frac{1}{2}
    =\frac{1}{s}+\frac{1}{p_0}
    \quad
    \Leftrightarrow
    \quad
    \frac{1}{s}=\frac{3}{2}-\frac{1}{p_0},
\end{align}
which is possible as $p_0\in (1,2)$. We use \eqref{eq:interpolation_rs},  \eqref{eq:rel_sp} and apply the Young convolution inequality to get
\begin{align*}
    \|(\rho^\sigma e^{V-\frac{\mathbb{V}}{2}}\chi_R)\ast |y\,\omega_\sigma|\|_{L^r_tL^2_x}\le \|\rho^\sigma e^{V-\frac{\mathbb{V}}{2}}\chi_R\|_{L^r_tL^s_x}\,\|y\,\omega_\sigma\|_{L^{p_0}_x}.
\end{align*}
Next, we wish to apply H\"older in time so that
\begin{align*}
    \|((\rho^\sigma e^{V-\frac{\mathbb{V}}{2}}\chi_R)\ast|y\,\omega_\sigma|)\|\varphi\|_{W^{1,\infty}}\|_{L_{t,x}^2}\leq \|(\rho^\sigma e^{V-\frac{\mathbb{V}}{2}}\chi_R)\ast |y\,\omega_\sigma|\|_{L^r_tL^2_x}\,\|\varphi\|_{L^q_t W_x^{1,\infty}}
\end{align*}
whenever the temporal exponents satisfy
\begin{align*}
    \frac{1}{2}=\frac{1}{r}+\frac{1}{q}.
\end{align*}
But we already have $1/r=1-1/s$, so combining with \eqref{eq:rel_sp} we obtain
\begin{align*}
    \frac{1}{2}-\frac{1}{r}=\frac{1}{2}-\left(1-\frac{1}{s}\right)
    =-\frac{1}{2}+\left(\frac{3}{2}-\frac{1}{p_0}\right)=1-\frac{1}{p_0}
    =\frac{1}{q},
\end{align*}
by the choice $ q=\frac{p_0}{p_0-1}$. Thus in order to prove that \eqref{eq:goal_comm_slsi} holds, it suffices to show $y\,\omega_\sigma\to 0$ in $L^{p_0}(\mathbb{R}^{d})$. To that end, let us compute the $L^{p_0}$--norm of $y\,\omega_\sigma$ explicitly. We use the change of variables $y=\sigma z$ to get
\begin{align*}
    \|y\,\omega_\sigma\|_{L^{p_0}}^{p_0}
    &=\int_{\mathbb{R}^{d}} |y|^{p_0}\,|\omega(y/\sigma)|^{p_0}\,\sigma^{-dp_0}\,\mathrm{d} y
    = \sigma^{p_0-dp_0}\int_{\mathbb{R}^{d}} |z|^{p_0}\,|\omega(z)|^{p_0}\,\sigma^{d}\,\mathrm{d} z
    = \sigma^{p_0-dp_0+d}\,\|y\,\omega\|_{L^{p_0}}^{p_0}.
\end{align*}
The claim now immediately follows from \eqref{ass:moment} of Assumption \ref{assumptions:slsi} and the fact that
\begin{align*}
    1-d\left(1-\frac{1}{p_0}\right)>0
    \quad\Leftrightarrow\quad
    p_0<\frac{d}{d-1},
\end{align*}
where we treat $\infty=\frac{1}{0}$ (note also that $\frac{2(d+1)}{2d+1}<\frac{d}{d-1}$). As this estimate holds uniformly for all $\varphi$ with $\|\varphi\|_{L_t^q W_x^{1,\infty}}\leq 1$, we conclude the proof by duality.
\end{proof}

The following lemma is concerned with solutions $\{\varrho^\sigma\}_\sigma$ to \eqref{eq:mf_svgd} for which we do not have a priori estimates for $\{\varrho^\sigma\}_\sigma$, but only for the mollified quantities $\{\varrho^\sigma \ast \omega_{\sigma}\}_\sigma$. Therefore, we need to establish a stronger commutator result in which the localization is completely removed from the convolution.

\begin{lem}\label{lem:comm_est}
    Let Assumption \ref{assumptions:pot} hold, let $\{\varrho^\sigma\}_\sigma$ be the weak solutions to \eqref{eq:mf_svgd}, let $T,R>0$ and $q\in \mathcal{K}_d$. Then, with $m_0$ as in Assumption \ref{assumptions:pot},
    \begin{align}
        \label{eq:commutator_representation}
        \lim_{\sigma\to 0}\;\;\sup_{\substack{\varphi\in L^q(0,T;\, W_0^{m_0,\infty}(B_R)) \\ \|\varphi\|_{ L_t^q W_x^{m_0,\infty}}\leq 1}}\;\;
        \|(\varrho^\sigma \varphi)\ast\omega_\sigma
        -\varphi\,(\varrho^\sigma\ast\omega_\sigma)\|_{L_{t,x}^2}
        =0.
    \end{align}
\end{lem}
\begin{proof}
    For this proof it is important to observe that there exists $p\in \mathcal{J}_d$ such that $p>2$ with $\tfrac{2}{p}+\tfrac{2}{q}=1$. This follows by checking the value of $p=\frac{2q}{q-2}$ for the boundary values of $q \in \mathcal{K}_d$. We will work with a fixed function $\varphi\in L^q(0,T;\, W_0^{m_0,\infty}(B_R))$ with $\|\varphi\|_{ L_t^q W_x^{m_0,\infty}}\leq 1$. Let us give the proof strategy. We carry out the argument in four steps. First, we approximate $\varphi$ with a smooth compactly supported function, so that in Step 2 we can apply a Taylor's approximation argument. Finally, Steps 3 and 4 are devoted to showing that error terms converge to zero.
    
    \underline{Step 1: Approximation of $\varphi$.} We choose a smooth compactly supported function $\gamma~\in~C_c^\infty(\mathbb{R}^{d})$ such that $\gamma\geq 0$ and $\int_{\mathbb{R}^{d}}\gamma\,\mathrm{d} x=1$. We define the mollification kernel $\gamma_\varepsilon(x):=\varepsilon^{-d}\gamma(x/\varepsilon)$ for any $\varepsilon\in (0,1)$. Then, we smooth the function $\varphi$ in the spatial variable by $\varphi_\varepsilon:=\varphi\ast\gamma_\varepsilon$. For any multiindex $\alpha\in \mathbb{N}^d$ with $|\alpha|\leq m$, we note that
    \begin{align}
        \label{eq:approx_vp_D0}
        &\|\varphi- \varphi_\varepsilon\|_{L_t^q L_x^\infty}
        \leq \|\nabla \varphi\|_{L_t^q L_x^{\infty}}\|y\,\gamma_\varepsilon\|_{L_x^1}\leq \varepsilon\| \varphi\|_{L_t^q W_x^{m_0+1,\infty}}\|y\,\gamma\|_{L_x^1}
        \leq C\,\varepsilon,\\
        \label{eq:approx_vp_bnd}
        &\|\partial^\alpha \varphi^\varepsilon\|_{L_t^q L_x^\infty}\leq
        \|\partial^\alpha \varphi\|_{L_t^q L_x^\infty}\|\gamma\|_{L_x^1}
        \leq 1,
    \end{align}
    where $C\geq 1$ is a constant independent of $\varphi$ and $\varepsilon$. Then, we use the Young convolution inequality, the H\"older inequality, \ref{est_gf_kl} of Lemma \ref{lem:gf_est} and the change of variables formula to estimate
    \begin{align*}
        &\|(\varrho^\sigma\varphi)\ast\omega_\sigma-(\varrho^\sigma\varphi^\varepsilon)\ast\omega_\sigma\|_{L_{t,x}^2}
        \leq \|\varrho^\sigma(\varphi-\varphi^\varepsilon)\|_{L_t^2 L_x^1}\|\omega_\sigma\|_{L^2_x}\\
        &\leq \|\varphi-\varphi^\varepsilon\|_{L_t^q L_x^\infty}\|\varrho^\sigma\|_{L_t^p L_x^1}\|\omega_\sigma\|_{L^2_x}
        \leq C\,\sigma^{-\frac{d}{2}}\|\varphi-\varphi^\varepsilon\|_{L_t^q L_x^\infty}.
    \end{align*}
    Furthermore, by the H\"older inequality, the fact that $\varphi-\varphi^\varepsilon$ is supported on $[0,T]\times \overline{B_{R+1}}$ and \ref{bnd:rap_lp} of Lemma \ref{lem:ref_a_priori_est} we can estimate
    \begin{align*}
        &\|\varphi(\varrho^\sigma\ast\omega_\sigma)-\varphi^\varepsilon(\varrho^\sigma\ast\omega_\sigma)\|_{L_{t,x}^2}
        \leq \|\varphi-\varphi^\varepsilon\|_{L_t^q L_x^\infty}\|\varrho^\sigma\ast\omega_\sigma\|_{L^p(0,T;\,L^2(B_{R+1}))}\\
        &\leq |B_{R+1}|^{\frac{1}{2}- \frac{1}{p}}\|\varphi-\varphi^\varepsilon\|_{L_t^q L_x^\infty}\|\varrho^\sigma\ast\omega_\sigma\|_{L^p((0,T)\times B_{R+1})}
        \leq C\|\varphi-\varphi^\varepsilon\|_{L_t^q L_x^\infty},
    \end{align*}
    where $C>0$ is some constant independent of $\sigma$ and $\varepsilon$. Hence, we deduce the following error estimate
    \begin{equation}\label{eq:commutator_proof_before_Taylors_expansion_regularization}
        \|(\varrho^\sigma\varphi)\ast\omega_\sigma-\varphi\,(\varrho^\sigma\ast\omega_\sigma)\|_{L_{t,x}^2}\leq \|(\varrho^\sigma\varphi^\varepsilon)\ast\omega_\sigma-\varphi^\varepsilon (\varrho^\sigma\ast\omega_\sigma)\|_{L_{t,x}^2}\!+\!2C\,\sigma^{-\frac{d}{2}}\,\|\varphi-\varphi^\varepsilon\|_{L_t^q L_x^\infty}.
    \end{equation}
    Now, by using \eqref{eq:approx_vp_D0} we will choose $\varepsilon=\sigma^d$ so that $\sigma^{-\frac{d}{2}}\,\|\varphi-\varphi^\varepsilon\|_{L_t^q L_x^\infty} \to 0$.
    
    \underline{Step 2: Taylor's expansion.} For a.e. $t\in [0,T]$, we use the Taylor's expansion and derive the pointwise estimate
    \begin{equation}\label{eq:Taylor_proof_commutator_estimate_Step2_sect4}
        \begin{split}
            &((\varrho_t^\sigma\varphi_t^\varepsilon)\ast \omega_\sigma)(x)
            =\int_{\mathbb{R}^{d}}\varrho_t^\sigma(x-y)\varphi_t^\varepsilon(x-y)\omega_\sigma(y)\,\mathrm{d} y\\
            &=\varphi_t^\varepsilon(x) (\varrho_t^\sigma\ast\omega_\sigma)(x)+\sum_{\substack{\alpha\in\mathbb{N}^d \\ 1\leq |\alpha|<m_0}}\frac{(-1)^{|\alpha|}}{\alpha!}\partial^\alpha\varphi_t^\varepsilon(x)\; (\varrho_t^\sigma\ast (y^\alpha \omega_\sigma))(x)
            +R_t^{\sigma,\varepsilon}(x),
        \end{split}
    \end{equation}
    where we define the remainder by
    \begin{align*}
        R_t^{\sigma,\varepsilon}(x)
        =\sum_{\substack{\alpha\in\mathbb{N}^d \\ |\alpha|=m_0}}\frac{m_0}{\alpha!}\int_{\mathbb{R}^{d}} \varrho_t^\sigma(x-y) \,y^\alpha\omega_\sigma(y)\int_0^1 (1-s)^{m_0-1}\partial^\alpha \varphi_t^\varepsilon(x-s y)\,\mathrm{d} s\,\mathrm{d} y.
    \end{align*}

    \underline{Step 3: Convergence of an individual summand $\alpha$.} We will prove that for each $|\alpha|< m_0$
    \begin{equation}\label{eq:step3_commutator_estimate_not_weighted}
        \|\partial^\alpha\varphi^\varepsilon\; (\varrho^\sigma\ast (y^\alpha \omega_\sigma))\|_{L_{t,x}^2} \to 0 \mbox{ as } \sigma \to 0.
    \end{equation}
    We first estimate using the H\"older inequality
    \begin{align}
        \label{eq:est_sum_alpha}
        \begin{split}
            &\|\partial^\alpha\varphi^\varepsilon\; (\varrho^\sigma\ast (y^\alpha \omega_\sigma))\|_{L_{t,x}^2}
            \leq \|\partial^\alpha\varphi^\varepsilon\|_{L_t^q L_x^\infty}\;\|\varrho^\sigma\ast (y^\alpha \omega_\sigma)\|_{L^p(0,T;L^2(B_R))}\\
            &\leq C\,\|\varrho^\sigma\ast (y^\alpha \omega_\sigma)\|_{L^p(0,T;L^2(B_R))} \leq C\,\|\varrho^\sigma\ast (|y|^k \omega_\sigma)\|_{L^p(0,T;L^2(B_R))},
        \end{split}
    \end{align}
    where $C>0$ is some constant independent of $\varepsilon$ and $\sigma$ and $k:=|\alpha|$. This estimate removes the dependence on $\varepsilon$. Next, we observe that
    \begin{equation}\label{eq:claim_conv_in_measure}
        \varrho^\sigma \ast (|x|^{k}\,\omega_\sigma) \to 0 \mbox{ in }L^\infty(0,T;\, L^1(B_R)).
    \end{equation}
    Indeed, this is a consequence of \ref{est_gf_kl} of Lemma \ref{lem:gf_est} and \ref{ass:sob_kern} of Assumption \ref{assumptions:pot} in conjunction with the Young convolution inequality, and then employing change of variables:
    \begin{align*}
        \| \varrho^\sigma \ast (|x|^{k}\,\omega_\sigma)\|_{L^{\infty}(0,T; L^1(\mathbb{R}^d))} &\leq \| \varrho^\sigma \|_{L^{\infty}(0,T; L^1(\mathbb{R}^d))} \, \| |x|^{k}\,\omega_\sigma \|_{L^1(\mathbb{R}^d)} \\
        &=\sigma^k\,\|\varrho_0\|_{L^1(\mathbb{R}^d)} \, \||x|^{k} \, \omega\|_{L^1(\mathbb{R}^d)} \to 0.
    \end{align*}
    As $|x|^k \leq 1 + |x|^{m_0}$ we estimate
    \begin{equation}\label{eq:sequence_of_interest_pointwise_bound}
        0 \leq \varrho^\sigma \ast (|x|^{k}\,\omega_\sigma)  \leq   \varrho^\sigma\ast\omega_\sigma +  \varrho^\sigma \ast (|x|^{m_0}\,\omega_\sigma) .
    \end{equation}
    This estimate inspires the following partition
    \begin{align*}
        A_\sigma&:=\{\varrho^\sigma\ast\omega_\sigma \geq \varrho^\sigma \ast (|x|^{m_0}\,\omega_\sigma) \}\cap [0,T]\times B_R,\\
        B_\sigma&:=\{\varrho^\sigma\ast\omega_\sigma < \varrho^\sigma \ast (|x|^{m_0}\,\omega_\sigma) \}\cap [0,T]\times B_R.
    \end{align*}
    Then on the set $A_\sigma$ we have the estimate
    \begin{align}
        \label{eq:set_A_1}
        \begin{split}
            &\|(\rho^\sigma\ast(|x|^k\omega_\sigma))\mathds{1}_{A_\sigma}\|_{L_t^p L_x^2}^p
            = \int_0^T
            \|(\rho_t^\sigma\ast(|x|^k\omega_\sigma))\mathds{1}_{A_\sigma}\|_{L_x^2}^p\,\mathrm{d} t\\
            &\leq \int_0^T
            \|(\rho_t^\sigma\ast(|x|^k\omega_\sigma))\mathds{1}_{A_\sigma}\|_{L_x^1}^{\theta p}\|(\rho_t^\sigma\ast(|x|^k\omega_\sigma))\mathds{1}_{A_\sigma}\|_{L_x^p}^{(1-\theta) p}\,\mathrm{d} t\\
            &\leq \|(\rho^\sigma\ast(|x|^k\omega_\sigma))\mathds{1}_{A_\sigma}\|_{L_t^\infty L_x^1}^{\theta p}\;\int_0^T
            \|(\rho_t^\sigma\ast(|x|^k\omega_\sigma))\mathds{1}_{A_\sigma}\|_{L_x^p}^{(1-\theta) p}\,\mathrm{d} t,
        \end{split}
    \end{align}
    where we have used interpolation with $\theta=\frac{1}{2}\frac{p-2}{p-1}$. Using \eqref{eq:sequence_of_interest_pointwise_bound}, the inequality determining the set $A_\sigma$, the fact that $(1-\theta)< 1$, and the H\"older inequality we get
    \begin{align}
        \label{eq:set_A_2}
        \begin{split}
            &\|(\rho^\sigma\ast(|x|^k\omega_\sigma))\mathds{1}_{A_\sigma}\|_{L_t^\infty L_x^1}^{\theta p}\int_0^T
            \|(\rho_t^\sigma\ast(|x|^k\omega_\sigma))\mathds{1}_{A_\sigma}\|_{L_x^p}^{(1-\theta) p}\,\mathrm{d} t\\
            &\leq 2\,T^{\theta}\,\|(\rho^\sigma\ast(|x|^k\omega_\sigma))\mathds{1}_{A_\sigma}\|_{L_t^\infty L_x^1}^{\theta p}
            \|\rho^\sigma\ast\omega_\sigma\|_{L^p(0,T;L^p(B_R))}^{(1-\theta)p}.
        \end{split}
    \end{align}
    Now by \ref{bnd:rap_lp} from Lemma \ref{lem:ref_a_priori_est}, $\{\varrho^\sigma\ast\omega_\sigma\}_\sigma$ is uniformly bounded in $L^p((0,T)\times B_R)$, thus using \eqref{eq:claim_conv_in_measure} we get
    \begin{align}
        \label{eq:set_A_3}
        \begin{split}
            \lim_{\sigma\to 0}
            \|(\rho^\sigma\ast(|x|^k\omega_\sigma))\mathds{1}_{A_\sigma}\|_{L_t^\infty L_x^1}^{\theta p}
            \|\rho^\sigma\ast\omega_\sigma\|_{L^p(0,T;L^p(B_R))}^{(1-\theta)p}
            =0.
        \end{split}
    \end{align}
    Let us now focus on the set $B_\sigma$. We have $\varrho^\sigma \ast (|x|^{m_0} \omega_\sigma) \to 0$ as $\sigma \to 0$ in $L^p(0,T; L^2(\mathbb{R}^d))$ by Lemma \ref{lem:conv_mollifiers_scaled_x}, since $m_0> \frac{d}{2}$. Thus, using \eqref{eq:sequence_of_interest_pointwise_bound} and the inequality determining  $B_\sigma$ we get the convergence
    \begin{align}
        \label{eq:set_B}
        &\lim_{\sigma\to 0}\|(\rho^\sigma\ast(|x|^k\omega_\sigma))\mathds{1}_{B_\sigma}\|_{L_t^p L_x^2}^p
        \leq \lim_{\sigma\to 0}2\|(\rho^\sigma\ast(|x|^{m_0}\omega_\sigma))\mathds{1}_{B_\sigma}\|_{L_t^p L_x^2}^p
        =0.
    \end{align}
    Thus by combining \eqref{eq:set_A_1}--\eqref{eq:set_B} we have shown that the RHS in \eqref{eq:est_sum_alpha} converges to zero, thus proving \eqref{eq:step3_commutator_estimate_not_weighted}.
    
    \smallskip
    
    \underline{Step 4: Convergence of remainder.}
    We estimate
    \begin{align*}
        |R_t^{\sigma,\varepsilon}(x)|
        \leq \sum_{\substack{\alpha\in\mathbb{N}^d \\ |\alpha|=m_0}}
        \frac{m_0}{\alpha!}\|\partial^\alpha \varphi_t^\varepsilon\|_{L_{x}^\infty} |(\varrho_t^\sigma\ast (|y^\alpha|\,\omega_\sigma))(x)|
        \leq C\,\|\varphi_t^{\varepsilon}\|_{W_x^{m_0,\infty}} (\varrho_t^\sigma\ast (|y|^{m_0}\omega_\sigma)(x)),
    \end{align*}
    for some constant $C>0$ independent of $\sigma$ and $\varepsilon$. From Lemma~\ref{lem:conv_mollifiers_scaled_x} and $m_0> \frac{d}{2}$ we deduce
    \begin{align*}
        \| R_t^{\sigma,\varepsilon} \|_{L^2_{t,x}} &\leq C\,\|\varphi_t^{\varepsilon}\|_{L^2_t W_x^{m_0,\infty}} \|\varrho_t^\sigma\ast (|y|^{m_0}\omega_\sigma)\|_{L^{\infty}_t L^2_x}\\
        &\leq C\,\|\varphi_t^{\varepsilon}\|_{L^q_t W_x^{m_0,\infty}}\, \|\varrho_t^\sigma \|_{L^{\infty}_t L^1_x}\, \| |y|^{m_0}\omega_\sigma\|_{L^2_x} \leq C\, \| |y|^{m_0}\omega_\sigma\|_{L^2_x} \to 0 \mbox{ as } \sigma\to 0,
    \end{align*}
    since $q>2$ taking into account \eqref{eq:approx_vp_bnd}. Thus, we deduce $R_t^{\sigma,\varepsilon} \to 0$ strongly in $L^2((0,T)\times \mathbb{R}^{d})$. Combining with \eqref{eq:step3_commutator_estimate_not_weighted}, we deduce that every summand in the Taylor's expansion \eqref{eq:Taylor_proof_commutator_estimate_Step2_sect4} converges strongly to 0 in $L^2((0,T)\times\mathbb{R}^{d})$. Hence, the first term on the RHS of \eqref{eq:commutator_proof_before_Taylors_expansion_regularization} converges to 0 while the second converges by the choice of $\varepsilon$ in terms of $\sigma$. As these estimates hold uniformly for all $\varphi$ with $\|\varphi\|_{L_t^q W_x^{m_0,\infty}}\leq 1$, we conclude the proof by duality.
\end{proof}

\begin{rem}
    We emphasize that the commutator argument developed in Lemma \ref{lem:comm_est} plays a crucial role beyond its application in the proof of Theorem \ref{thm:ex_weak_sol_nonloc_pde}. In particular, it will also be used in the proof of Lemma \ref{bnd:dt_rho_chi} to obtain bounds on the temporal derivative.
\end{rem}


\section{Estimate on the temporal derivatives}

\label{subsec:dt_est}

In this section, we prove boundedness of the temporal derivatives of the solutions to \eqref{eq:mf_svgd} and \eqref{eq:mf_slsi_svgd}, as well as of their mollifications with respect to the scaled kernel. These bounds are used, on the one hand, to establish continuity of the corresponding curves and, on the other hand, in combination with the Aubin-Lions lemma in Section \ref{sec:weak_conv} to obtain strong convergence.

\begin{lem}
    \label{bnd:dt_rho}
    Let $\{\rho^\sigma\}_\sigma$ be the weak solution to \eqref{eq:mf_slsi_svgd} and let Assumption \ref{assumptions:slsi} hold. Then, for any fixed $T,R>0$, $k=2$ and $q=\frac{p_0}{p_0-1}$ with $p_0$ defined in \ref{ass:four_mom}, the sequence $\{\partial_t\rho^\sigma\}_\sigma$ is uniformly bounded in the space $(L^q(0,T;\, W_0^{k,\infty}(B_R)))^\ast$.
    The same statement holds true for $\{\varrho^\sigma\}_\sigma$, the weak solution of \eqref{eq:mf_svgd}, whenever Assumption \ref{assumptions:pot} holds, $k=m_0+1$, $q\in \mathcal{K}_d$, and $m_0$ is as in Assumption \ref{assumptions:pot}.
\end{lem}

\begin{proof}
    We concentrate on the bound for $\{\rho^\sigma\}_\sigma$ the weak solution of \eqref{eq:mf_slsi_svgd}, where we use assumption \ref{ass:four_mom} to determine $q$. The same proof strategy holds for $\{\varrho^\sigma\}_\sigma$ the weak solution of \eqref{eq:mf_svgd} up to minor modifications, and here $q$ can be arbitrary in $\mathcal{K}_d$. We derive the uniform boundedness through duality. Take any test function $\psi\in C_c^\infty((0,T)\times B_R)$ and use \eqref{eq:mf_slsi_svgd} (respectively \eqref{eq:mf_svgd}) to derive
    \begin{align}
        \label{eq:test_on_mf_svgd}
        \begin{split}
            &\left\vert\int_0^T \int_{B_R} \psi\;\partial_t\rho_t^\sigma\,\mathrm{d} x\,\mathrm{d} t\right\vert\\
            &\leq\int_0^T \int_{\mathbb{R}^{d}}|(\rho_t^\sigma e^{V-\frac{\mathbb{V}}{2}}\nabla \psi)\ast\omega_\sigma|\; |((\nabla \rho_t^\sigma+\rho_t^\sigma \nabla V)e^{V-\frac{\mathbb{V}}{2}})\ast\omega_\sigma|\,\mathrm{d} x\,\mathrm{d} t.
        \end{split}
    \end{align}
    We apply the triangle inequality to get
    \begin{align}
        \label{eq:comm_trick}
        \begin{split}
            &|(\rho_t^\sigma e^{V-\frac{\mathbb{V}}{2}}\nabla \psi)\ast\omega_\sigma|\\
            &\leq |((\rho_t^\sigma e^{V-\frac{\mathbb{V}}{2}}\nabla \psi)\ast\omega_\sigma)-\nabla \psi\;((\rho_t^\sigma e^{V-\frac{\mathbb{V}}{2}}\chi_R)\ast\omega_\sigma)|
            +|\nabla\psi|\;|(\rho_t^\sigma e^{V-\frac{\mathbb{V}}{2}}\chi_R)\ast\omega_\sigma|.
        \end{split}
    \end{align}
    Then, if we employ the uniform bound \ref{est_gf_dis} of Lemma \ref{lem:gf_est} and implement the commutator argument \eqref{eq:conv_com_slsi} of Lemma \ref{lem:comm_est_slsi} (respectively \eqref{eq:commutator_representation} of Lemma \ref{lem:comm_est}) in \eqref{eq:comm_trick}, we only need to bound the following integral
    \begin{align}
        \label{eq:test_on_mf_svgd_better}
        \int_0^T \int_{B_R}|\nabla \psi|\;|(\rho_t^\sigma e^{V-\frac{\mathbb{V}}{2}}\chi_R)\ast\omega_\sigma|\; |((\nabla \rho_t^\sigma+\rho_t^\sigma \nabla V)e^{V-\frac{\mathbb{V}}{2}})\ast\omega_\sigma|\,\mathrm{d} x\,\mathrm{d} t.
    \end{align}
    By using \ref{bnd:rap_lp} from Lemma \ref{lem:ref_a_priori_est}, \ref{est_gf_dis} from Lemma \ref{lem:gf_est}, and the H\"older inequality, we estimate \eqref{eq:test_on_mf_svgd_better} from above by
    \begin{align*}
        &\|(\rho_t^\sigma e^{V-\frac{\mathbb{V}}{2}}\chi_R)\ast\omega_\sigma\|_{L^r((0,T)\times B_R)}\;
        \|((\nabla \rho_t^\sigma+\rho_t^\sigma \nabla V)e^{V-\frac{\mathbb{V}}{2}}) \ast\omega_\sigma\|_{L^2((0,T)\times B_R)}\;
        \|\nabla\psi\|_{L^q((0,T)\times B_R)}\\
        &\leq C\|\nabla\psi\|_{L^q((0,T)\times B_R)}
        \leq C|B_R|^{\frac{1}{q}}\,\|\psi\|_{L^q(0,T;\, W^{k,\infty}(B_R))},
    \end{align*}
    where $C>0$ is independent of $\sigma$ and $\psi$, and we define $r := \frac{2q}{q-2} \in \mathcal{J}_d \cap (2,\infty)$ so that $\frac{2}{q} + \frac{2}{r} = 1$; note that $\frac{2q}{q-2} \in \mathcal{J}_d \cap (2,\infty)$ for every $q \in \mathcal{K}_d$, and in the case of weak solutions to \eqref{eq:mf_slsi_svgd}, the exponent $q$ is contained in $\mathcal{K}_d$. As $\psi$ is arbitrary, we see that the desired bound holds true.
\end{proof}

\begin{cor}
    \label{bnd:dt_rho_chi_slsi}
    Let Assumption \ref{assumptions:slsi} hold and let $\{\rho^\sigma\}_\sigma$ be the weak solution to \eqref{eq:mf_slsi_svgd}. Then for any fixed $T,R>0$ and $q=\frac{p_0}{p_0-1}$, where $p_0$ is determined by \ref{ass:four_mom}, the sequence $\{\partial_t((\rho^\sigma e^{V-\frac{\mathbb{V}}{2}}\chi_R)\ast\omega_\sigma)\}_\sigma$ is uniformly bounded in the space $(L^q(0,T;\, W^{2,\infty}(\mathbb{R}^{d})))^\ast$.
\end{cor}
\begin{proof}
    We employ a duality argument. Take any test function $\psi\in C_c^\infty((0,T)\times \mathbb{R}^{d})$ and compute
    \begin{align*}
        \left\vert \int_0^T\int_{\mathbb{R}^{d}} \partial_t\,((\rho_t^\sigma e^{V-\frac{\mathbb{V}}{2}}\chi_R)\ast\omega_\sigma)\; \psi\,\mathrm{d} x\,\mathrm{d} t\right\vert
        =\left\vert \int_0^T\int_{\mathbb{R}^{d}} \partial_t\rho_t^\sigma\;  (\psi\ast \omega_\sigma)\, e^{V-\frac{\mathbb{V}}{2}}\chi_R\,\mathrm{d} x\,\mathrm{d} t\right\vert.
    \end{align*}
    Then, by Lemma \ref{bnd:dt_rho}, there exists a constant $C>0$ independent of $\sigma$ such that
    \begin{align*}
        \left\vert \int_0^T\int_{B_{2R}} \partial_t\rho_t^\sigma \; (\psi\ast \omega_\sigma)\, e^{V-\frac{\mathbb{V}}{2}}\chi_R\,\mathrm{d} x\,\mathrm{d} t\right\vert
        \leq C\;\|(\psi\ast \omega_\sigma) \,e^{V-\frac{\mathbb{V}}{2}}\chi_R\|_{L_t^q W_x^{2,\infty}}.
    \end{align*}
    Since $e^{V-\frac{\mathbb{V}}{2}}\chi_R$ is supported on $B_{2R}$ and $\omega_\sigma\in L^1(\mathbb{R}^{d})$, we find with Young's convolution inequality, \ref{ass:quad} and the change of variables formula
    \begin{align*}
        \|(\psi\ast \omega_\sigma) e^{V-\frac{\mathbb{V}}{2}}\chi_R\|_{L_t^q W_x^{2,\infty}}
        &\leq 4\,\|\psi\ast \omega_\sigma \|_{L_t^q W_x^{2,\infty}}\|e^{V-\frac{\mathbb{V}}{2}}\chi_R\|_{W^{2,\infty}}\\
        &\leq 4\,\|\psi\|_{L_t^q W_x^{2,\infty}}\|\omega\|_{L^1}\|e^{V-\frac{\mathbb{V}}{2}}\chi_R\|_{W^{2,\infty}}.
    \end{align*}
    Taking the supremum over all $\|\psi\|_{L_t^q W_x^{2,\infty}}\leq 1$ proves the claim.
\end{proof}

\begin{lem}
    \label{bnd:dt_rho_chi}
    Let Assumption \ref{assumptions:pot} hold and let $\{\varrho^\sigma\}_\sigma$ be the weak solution to \eqref{eq:mf_svgd}. Then, for all fixed $T,R>0$ and $q\in \mathcal{K}_d$, the sequence $\{\partial_t(\varrho^\sigma\ast\omega_\sigma)\}_\sigma$ is uniformly bounded in the space $(L^q(0,T;\, W_0^{m_0+1,\infty}(B_R)))^\ast$, where $m_0$ is as in Assumption \ref{assumptions:pot}.
\end{lem}
\begin{proof}
    Let us give a proof strategy. In Step 1, we reduce the problem to showing that a specific function is uniformly bounded. Then, in Step 2 we find a decomposition of the function via the Taylor approximation. Finally, in Step 3 we employ this decomposition and we close the proof by showing the resulting terms in the decomposition are bounded.
    
    \underline{Step 1: Problem reduction.}  Take any test function $\psi\in C_c^\infty([0,T]\times B_R)$. We use the weak formulation of \eqref{eq:mf_svgd} and the Cauchy-Schwarz inequality to compute
    \begin{align*}
        &\left\vert \int_0^T\int_{B_R} \partial_t\,(\varrho_t^\sigma\ast\omega_\sigma)\; \psi\,\mathrm{d} x\,\mathrm{d} t\right\vert
        =\left\vert \int_0^T\int_{\mathbb{R}^{d}} \partial_t\varrho_t^\sigma\;  (\psi\ast \omega_\sigma) \,\mathrm{d} x\,\mathrm{d} t\right\vert\\
        &=\left\vert \int_0^T\int_{\mathbb{R}^{d}} (\nabla\varrho^\sigma+\varrho^\sigma\nabla V)\ast\omega_\sigma\cdot  ((\nabla\psi\ast \omega_\sigma)\varrho^\sigma)\ast\omega_\sigma \,\mathrm{d} x\,\mathrm{d} t\right\vert\\
        &\leq \|(\nabla\varrho^\sigma+\varrho^\sigma\nabla V)\ast\omega_\sigma\|_{L_{t,x}^2}\|((\nabla\psi\ast \omega_\sigma)\varrho^\sigma)\ast\omega_\sigma\|_{L_{t,x}^2}.
    \end{align*}
    The first term of the RHS is uniformly bounded by \ref{est_gf_dis} of Lemma \ref{lem:gf_est}. Thus, we only need to estimate the remaining term.
    
    \underline{Step 2: Convolution decomposition.} We would like to use a commutator argument for the function $\nabla\psi\ast\omega_\sigma$ as described in Lemma \ref{lem:comm_est}. However, we observe that in general the function $\nabla\psi\ast\omega_\sigma$ is not compactly supported and thus the lemma is not directly applicable.
    
    Hence, we derive a decomposition as in Lemma \ref{lem:comm_est}. We use the Taylor's expansion to find a constant $C>0$ independent of $\sigma$ such that the following pointwise bound holds
    \begin{align*}
        |(\nabla\psi_t\ast\omega_\sigma)(x)|
        &\leq C\||y|^{m_0}\omega_\sigma\|_{L^1}\|\nabla^{m_0+1}\psi_t\|_{L^\infty}+C\sum_{k=0}^{m_0-1}
        \||y|^k\omega_\sigma\|_{L^1}
        |\nabla^{k+1} \psi_t(x)|\\
        &\leq C\|(1+|y|^{m_0})\,\omega\|_{L^1}\left(\sigma^{m_0}\|\nabla^{m_0+1}\psi_t\|_{ L^\infty}+\sum_{k=0}^{m_0-1}
        \sigma^{k}\|\nabla^{k+1} \psi_t\|_{L^\infty}\chi_R(x)\right),
    \end{align*}
    where we used the change of variables formula to derive $\||y|^k\omega_\sigma\|_{L^1}=\sigma^{k}\||y|^k\omega\|_{L^1}$ for all $k\in \{1,\ldots,m_0\}$. Thus, there exists non-negative functions $\varphi_0,\ldots,\varphi_{m_0-1}\in C^0([0,T];\, C_c^\infty(B_{2R}))$ and non-negative function $\Phi_{m_0}\in C^0([0,T];\,C^\infty(\mathbb{R}^{d}))$ such that
    \begin{align}
        \label{eq:conv_decomp}
        &|\nabla\psi\ast \omega_\sigma|
        \leq \sigma^{m_0}\Phi_{m_0}+\sum_{k=0}^{m_0-1} \sigma^{k}\varphi_k,\\
        \label{eq:decomp_norm}
        &\max\{\|\varphi_0\|_{L_t^q W_x^{m_0+1,\infty}},\,\ldots,\,\|\varphi_{m_0-1}\|_{L_t^q W_x^{m_0+1,\infty}},\,\|\Phi_{m_0}\|_{L_t^q L_x^\infty}\}\leq \|\psi\|_{L_t^q W_x^{m_0+1,\infty}}.
    \end{align}
    
    \underline{Step 3: Derivation of the estimate and conclusion.}
    By Step 1, we only need to focus on estimating $\|((\nabla\psi\ast \omega_\sigma)\varrho^\sigma)\ast\omega_\sigma\|_{L_{t,x}^2}$. By using the decomposition \eqref{eq:conv_decomp}, we derive the bound
    \begin{align*}
        \|((\nabla\psi\ast \omega_\sigma)\varrho^\sigma)\ast\omega_\sigma\|_{L_{t,x}^2}
        \leq \sigma^{m_0}\|(\Phi_{m_0}\varrho^\sigma)\ast\omega_\sigma\|_{L_{t,x}^2}+\sum_{k=0}^{m_0-1} \sigma^{k}\|(\varphi_k\varrho^\sigma)\ast\omega_\sigma\|_{L_{t,x}^2}.
    \end{align*}
    To bound an individual summand for $0 \leq k \leq m-1$, we first use the triangle inequality
    \begin{align*}
        \|(\varphi_k\varrho^\sigma)\ast\omega_\sigma\|_{L_{t,x}^2}
        \leq \|(\varphi_k\varrho^\sigma)\ast\omega_\sigma-\varphi_k(\varrho^\sigma\ast\omega_\sigma)\|_{L_{t,x}^2}+\|\varphi_k(\varrho^\sigma\ast\omega_\sigma)\|_{L_{t,x}^2}.
    \end{align*}
    We see by the commutator argument \eqref{eq:commutator_representation} of Lemma \ref{lem:comm_est}, it is sufficient to bound the second term on the RHS. We then compute by the H\"older inequality, \ref{bnd:rap_lp} of Lemma \ref{lem:ref_a_priori_est} and the inequality \eqref{eq:decomp_norm}
    \begin{align*}
        &\|\varphi_k(\varrho^\sigma\ast\omega_\sigma)\|_{L_{t,x}^2}
        \leq \|\varphi_k\|_{L_t^q W_x^{m_0+1,\infty}}\|\varrho^\sigma\ast\omega_\sigma\|_{L^p(0,T;\, L^2(B_{2R}))}\\
        &\leq |B_{2R}|^{\frac{1}{2}-\frac{1}{p}}\|\psi\|_{L_t^q W_x^{m_0+1,\infty}}\|\varrho^\sigma\ast\omega_\sigma\|_{L^p((0,T)\times B_{2R})}
        \leq C\,\|\psi\|_{L_t^q W_x^{m_0+1,\infty}},
    \end{align*}
    for some constant $C>0$, independent of $\sigma$, and with $p:=\frac{2q}{q-2}$, so that $\frac{2}{p}+\frac{2}{q}=1$ and $p\in\mathcal{J}_d$ whenever $q\in\mathcal{K}_d$ (see the beginning of the proof of Lemma~\ref{lem:comm_est}). To bound the remainder, we use the H\"older inequality, then the Young convolution inequality and finally Lemma \ref{lem:conv_mollifiers_scaled_x} to get
    \begin{align*}
        \sigma^{m_0}\|(\Phi_{m_0}\varrho^\sigma)\ast\omega_\sigma\|_{L_{t,x}^2}
        \leq \sigma^{m_0}\|\Phi_{m_0}\|_{L_t^q L_x^\infty}\|\varrho^\sigma\ast\omega_\sigma\|_{L_t^p L_x^2}
        \leq \sigma^{m_0-\frac{d}{2}}\|\Phi_{m_0}\|_{L_t^q L_x^\infty}\|\varrho^\sigma\|_{L_t^p L_x^1}\|\omega\|_{L_x^2}.
    \end{align*}
    As $m_0>\frac{d}{2}$, we see that the remainder term is uniformly bounded for all $\sigma\in (0,1)$. By taking the supremum over all $\|\psi\|_{L_t^q W_x^{m_0+1,\infty}}\leq 1$ in the obtained estimate for $\|((\nabla\psi\ast \omega_\sigma)\varrho^\sigma)\ast\omega_\sigma\|_{L_{t,x}^2}$ we conclude the proof.
\end{proof}


\section{Weak convergences}

\label{sec:weak_conv}

We show that there exist curves $\rho,\varrho \in C^0([0,\infty); \mathscr{P}(\mathbb{R}^{d}))$. Moreover, along a (not relabeled) subsequence, the weak solutions $\{\rho^{\sigma}\}_{\sigma}$ to \eqref{eq:mf_slsi_svgd} and $\{\varrho^{\sigma}\}_{\sigma}$ to \eqref{eq:mf_svgd} converge to $\rho$ and $\varrho$ weakly. These convergences will be used in the proof of Theorem \ref{thm:ex_weak_sol_loc_pde} and Theorem \ref{thm:ex_weak_sol_nonloc_pde}.

\begin{lem}[Basic convergence results]\label{lem:basic_conv}
    Let Assumption \ref{assumptions:slsi} hold and let $\{\rho^\sigma\}_\sigma$ be the weak solution to \eqref{eq:mf_slsi_svgd}. Then, for all fixed $T,R>0$ and $M>\frac{d}{2}+2$ the following convergences hold (up to passing to some subsequence):
    \begin{enumerate}[label=(BC\arabic*)]
        \item\label{bnd:conv_cont_curve}  $\rho^\sigma \to \rho$ as $\sigma\rightarrow 0$ in $C^0([0,T]; H^{-M}(B_R))$;
        \item\label{bnd:conv_l1} $\rho^\sigma_t \rightharpoonup \rho_t$ as $\sigma\rightarrow 0$ weakly in $L^1(\mathbb{R}^{d})$ and narrowly for a.e. $t>0$;
        \item\label{rem:l1_weak_conv}
        $\rho^\sigma \rightharpoonup \rho$ as $\sigma\rightarrow 0$ in $L^1([0,T]\times B_R)$.
    \end{enumerate}
    Analogue convergences hold for the weak solutions $\{\varrho^\sigma\}_\sigma$ to \eqref{eq:mf_svgd} under the Assumption \ref{assumptions:pot} and assuming $M>\frac{d}{2}+m_0+1$, where $m_0$ is as in Assumption \ref{assumptions:pot}.
\end{lem}
\begin{proof}
    We focus on the weak solutions of \eqref{eq:mf_slsi_svgd} and point out minor differences for the weak solutions of \eqref{eq:mf_svgd}. The convergence \ref{bnd:conv_cont_curve} follows because by Lemma \ref{bnd:dt_rho}, $\{\partial_t \rho^\sigma_t\}$ is bounded in $(L^q(0,T;\, W_0^{2,\infty}(B_R)))^\ast$ (respectively $(L^q(0,T;\, W_0^{m_0+1,\infty}(B_R)))^\ast$), which by the Sobolev embedding $H_0^{M}(B_R)\hookrightarrow W_0^{2,\infty}(B_R)$ (respectively $H_0^{M}(B_R)\hookrightarrow W_0^{m_0+1,\infty}(B_R)$) gives a bound in $L^{\frac{q}{q-1}}(0,T;\, H^{-M}(B_R)))$. Hence, by a direct integration in time, $\{\rho^\sigma_t\}$ is bounded in $C^{\beta}([0,T]; H^{-M}(B_R))$ for some $\beta \in (0,1)$ and we can apply classical Arzela-Ascoli argument, see \cite[Equation 4.18]{carrillo_skrzeczkowski_warnett_2024_stein} for an example. For the proof of \ref{bnd:conv_l1}, we will use the bounds in \ref{est_gf_kl} from Lemma \ref{lem:gf_est} and the classical Dunford-Pettis Theorem together with a subsequence argument and identification of the limit for all $t>0$ obtained in \ref{bnd:conv_l1}, see \cite[Equation 4.19]{carrillo_skrzeczkowski_warnett_2024_stein} for an example. We observe that by assumption \ref{ass:quad} (respectively \ref{ass:pot}), the uniform bound
    \begin{align}
        \label{eq:uni_int}
        \sup_{\sigma\in (0,1)}\int_{\mathbb{R}^{d}} \rho_t^\sigma(x) |\ln(\rho_t^\sigma(x))|\,\mathrm{d} x+\int_{\mathbb{R}^{d}} \rho_t^\sigma(x)V(x)\,\mathrm{d} x
        <\infty
    \end{align}
    implies again that $\{\rho^\sigma\}_{\sigma}$ is equi-integrable and tight on $[0,T]\times \mathbb{R}^{d}$. Thus, by the Dunford-Pettis Theorem, the sequence is weakly precompact in $L^1([0,T]\times \mathbb{R}^{d})$. We take any weakly convergence subsequence (not relabeled) $\rho^\sigma\to \mu$ in $L^1([0,T]\times \mathbb{R}^{d})$ and identify the limit. To that end, for any $\psi\in C_c^\infty([0,T]\times \mathbb{R}^{d})$ we have by \ref{bnd:conv_cont_curve}
    \begin{align*}
        &\int_0^T \int_{\mathbb{R}^{d}} \mu\psi\,\mathrm{d} x\,\mathrm{d} t
        =\lim_{\sigma\to 0}\int_0^T \int_{\mathbb{R}^{d}} \rho^\sigma\psi\,\mathrm{d} x\,\mathrm{d} t
        =\lim_{\sigma\to 0}\int_0^T\langle \rho_t^\sigma, \psi_t\rangle_{H^{-M},H^{M}}\,\mathrm{d} t\\
        &=\int_0^T\langle \rho_t, \psi_t\rangle_{H^{-M},H^{M}}\,\mathrm{d} t
        =\int_0^T \int_{\mathbb{R}^{d}}\rho\,  \psi\,\mathrm{d} x\,\mathrm{d} t.
    \end{align*}
    As this limit holds for any converging subsequence we choose, it follows that $\rho^\sigma\rightharpoonup \rho$. Similarly, we show narrow convergence $\rho^\sigma\rightharpoonup \rho$. Indeed, by \eqref{eq:uni_int} we have that $\{\rho^\sigma\}_\sigma$ is tight, and is narrowly precompact by Prokhorov's theorem.
\end{proof}

\begin{lem}[Strong convergence]\label{conv:rho_slsi}
    Let Assumption \ref{assumptions:slsi} hold, let $T,R>0$ and let $\{\rho^\sigma\}_\sigma$ be the weak solutions to \eqref{eq:mf_slsi_svgd} and let $\rho$ be the corresponding curve of probability measures in Lemma \ref{lem:basic_conv}. Then, we have the strong convergence $(\rho^\sigma e^{V-\frac{\mathbb{V}}{2}}\chi_R)\ast\omega_\sigma\rightarrow \rho e^{V-\frac{\mathbb{V}}{2}}\chi_R$ as $\sigma\rightarrow 0$ in $L^2([0,T]\times B_R)$.
    
    \smallskip
    
    Analogously, under the Assumption \ref{assumptions:pot}, $\varrho^\sigma\ast\omega_\sigma \to \varrho$ strongly in $L^2([0,T]\times B_R)$.
\end{lem}
\begin{proof}
    The proofs of $\{\rho^\sigma\}_\sigma$ and $\{\varrho^\sigma\}_\sigma$ are the exact same. We only show $\{\rho^\sigma\}_\sigma$ and point out differences for the other. First, we observe by \ref{bnd:h1} in Lemma \ref{lem:estimates_slsi} (respectively Lemma~\ref{prop:bounds_gradient}) that $\{(\rho_t^\sigma e^{V-\frac{\mathbb{V}}{2}}\chi_R)\ast\omega_\sigma\}_\sigma$ is uniformly bounded in $L^2(0,T;\, H^1(B_R))$, and by Corollary~\ref{bnd:dt_rho_chi_slsi} (respectively Lemma \ref{bnd:dt_rho_chi}) we have for $q=\frac{p_0}{p_0-1}$, with $p_0$ defined in \ref{ass:four_mom}, (respectively for any $q\in \mathcal{K}_d$) and $M>\frac{d}{2}+2$ (respectively $M>\frac{d}{2}+m_0+1$) that $\{\partial_t\,(\rho_t^\sigma e^{V-\frac{\mathbb{V}}{2}}\chi_R)\ast\omega_\sigma\}_\sigma$ is uniformly bounded in $L^q(0,T; H_0^{-M}(B_R))$, where we used the embedding $H_0^{M}(B_R)\hookrightarrow W^{2,\infty}(\mathbb{R}^{d})$ (respectively $H_0^{M}(B_R)\hookrightarrow W_0^{m_0+1,\infty}(B_R)$). The Aubin-Lions lemma yields (up to relabelling) a strongly convergent subsequence $(\rho^\sigma e^{V-\frac{\mathbb{V}}{2}}\chi_R)\ast\omega_\sigma\rightarrow \mu$ for some $\mu~\in~L^2([0,T]\times B_R)$. In order to identify the limit, we take a test function $\psi\in C_c^\infty([0,T]~\times~B_R)$ and note that we have strong convergence $\psi\ast\omega_\sigma\rightarrow \psi$ in $L^\infty([0,T]\times \mathbb{R}^{d})$. Thus, by using \ref{rem:l1_weak_conv} of Lemma \ref{lem:basic_conv} we have the weak convergence $\rho^\sigma\rightharpoonup\rho$ in $L^1([0,T]\times B_{2R})$. Then, we compute the limit
    \begin{align*}
        &\int_0^T\int_{B_R} \mu\, \psi\,\mathrm{d} x\,\mathrm{d} t
        =\lim_{\sigma\rightarrow 0} \int_0^T\int_{B_R} ((\rho^\sigma e^{V-\frac{\mathbb{V}}{2}}\chi_R)\ast\omega_\sigma)\; \psi\,\mathrm{d} x\,\mathrm{d} t\\
        &=\lim_{\sigma\rightarrow 0} \int_0^T\int_{B_{2R}} \rho^\sigma \; (\psi \ast\omega_\sigma)e^{V-\frac{\mathbb{V}}{2}}\chi_R\,\mathrm{d} x\,\mathrm{d} t
        =\int_0^T\int_{B_R} \rho \, e^{V-\frac{\mathbb{V}}{2}}\chi_R \; \psi\,\mathrm{d} x\,\mathrm{d} t.
    \end{align*}
    This proves that $\mu=\rho \, e^{V-\frac{\mathbb{V}}{2}}\chi_R$. Since, every subsequence we could have chosen yields the same limit, the whole sequence must strongly converge and the proof is concluded.
\end{proof}

\begin{lem}[Weak convergence]\label{lem:weak_convergences_slsi}
    Let Assumption \ref{assumptions:slsi} hold and let $\{\rho^\sigma\}_\sigma$ be the weak solutions to \eqref{eq:mf_slsi_svgd} and let $\rho$ be the corresponding curve of probability measures in Lemma \ref{lem:basic_conv}. Then, for all fixed $T,R>0$ we have the weak convergences:
    \begin{enumerate}[label=(O\arabic*)]
        \item\label{conv:weak_convergences_slsi_0} $\rho^\sigma e^{V-\frac{\mathbb{V}}{2}}\rightharpoonup \rho\, e^{V-\frac{\mathbb{V}}{2}}$ as $\sigma\rightarrow 0$ in $L^2([0,T]\times \mathbb{R}^{d})$,
        \item\label{conv:weak_convergences_slsi} 	$(\nabla(\rho^\sigma e^V)e^{-\frac{\mathbb{V}}{2}}\;\chi_R)\ast\omega_\sigma\rightharpoonup \nabla(\rho\, e^V)e^{-\frac{\mathbb{V}}{2}}\;\chi_R$ as $\sigma\rightarrow 0$ in $L^2((0,T)\times B_R)$,
        \item\label{conv:weak_convergences_slsi_1} 	$(\nabla(\rho^\sigma e^V)e^{-\frac{\mathbb{V}}{2}})\ast\omega_\sigma\rightharpoonup \nabla(\rho \, e^V)e^{-\frac{\mathbb{V}}{2}}$ as $\sigma\rightarrow 0$ in $L^2((0,T)\times \mathbb{R}^{d})$.
    \end{enumerate}
\end{lem}
\begin{proof}
    The proof of every convergence follows the same idea. We take a weakly converging subsequence with the desired limit. Since every subsequence that we could have chosen will yields the same limit, the whole sequence must weakly converge.
    
    \underline{The weak convergence \ref{conv:weak_convergences_slsi_0}.}
    By the Banach-Alaoglu Theorem and \ref{bnd:rho} of Lemma \ref{lem:estimates_slsi}, we may assume that $\rho^\sigma e^{V-\frac{\mathbb{V}}{2}}\rightharpoonup \mu$ as $\sigma\rightarrow 0$ weakly in $L^2([0,T]\times \mathbb{R}^{d})$. Now take any test function $\psi\in C_c^\infty([0,T]\times B_R)$ and compute
    \begin{align*}
        \int_0^T \int_{\mathbb{R}^{d}} \mu\,\psi\,\mathrm{d} x\,\mathrm{d} t
        =\lim_{\sigma\rightarrow 0}\int_0^T \int_{B_R} \rho^\sigma e^{V-\frac{\mathbb{V}}{2}}\psi\,\mathrm{d} x\,\mathrm{d} t
        =\int_0^T \int_{\mathbb{R}^{d}} \rho e^{V-\frac{\mathbb{V}}{2}}\psi\,\mathrm{d} x\,\mathrm{d} t,
    \end{align*}
    where we have used the fact that $e^{V-\frac{\mathbb{V}}{2}}\psi\in L^\infty([0,T]\times \mathbb{R}^{d})$ and is supported on $[0,T]\times B_R$, and that $\rho^\sigma \rightharpoonup \rho$ as $\sigma\rightarrow 0$ in $L^1([0,T]\times B_R)$ by \ref{rem:l1_weak_conv} of Lemma \ref{lem:basic_conv}. Then, as $R$ and $\psi$ is arbitrary, it follows that $\mu=\rho\, e^{V-\frac{\mathbb{V}}{2}}$.
    
    \underline{The weak convergence \ref{conv:weak_convergences_slsi}.} We observe the identity
    \begin{align*}
        (\nabla(\rho^\sigma e^V)e^{-\frac{\mathbb{V}}{2}}\;\chi_R)\ast\omega_\sigma=\nabla (\rho^\sigma e^{V-\frac{\mathbb{V}}{2}}\chi_R)\ast\omega_\sigma-(\rho^\sigma e^{V-\frac{\mathbb{V}}{2}} \nabla (\chi_R e^{-\frac{\mathbb{V}}{2}})\,e^{\frac{\mathbb{V}}{2}})\ast\omega_\sigma.
    \end{align*}
    We use Lemma \ref{conv:rho_slsi}, \ref{bnd:h1} of Lemma \ref{lem:estimates_slsi} and the Banach--Alaoglu to show that we have $\nabla (\rho^\sigma e^{V-\frac{\mathbb{V}}{2}}\chi_R)\ast\omega_\sigma\rightharpoonup \nabla (\rho \,e^{V-\frac{\mathbb{V}}{2}}\chi_R)$ as $\sigma\rightarrow 0$ in $L^2([0,T]\times B_R;\,\mathbb{R}^{d})$, up to passing to some subsequence (not relabeled). We again use the Banach--Alaoglu Theorem, \ref{bnd:rho} of Lemma~\ref{lem:estimates_slsi} and the fact that $\nabla (\chi_R e^{-\frac{\mathbb{V}}{2}})\, e^{\frac{\mathbb{V}}{2}}\in C_c^\infty(\mathbb{R}^{d};\,\mathbb{R}^{d})$ to show there exists a weakly convergent subsequence (not relabeled) such that $(\rho^\sigma e^{V-\frac{\mathbb{V}}{2}} \nabla (\chi_R e^{-\frac{\mathbb{V}}{2}})\,e^{\frac{\mathbb{V}}{2}})\ast\omega_\sigma\rightharpoonup \rho\, e^{V-\frac{\mathbb{V}}{2}} \nabla (\chi_R e^{-\frac{\mathbb{V}}{2}})\,e^{\frac{\mathbb{V}}{2}}$ as $\sigma\rightarrow 0$ in $L^2([0,T]\times \mathbb{R}^{d};\,\mathbb{R}^{d})$ (use the same identification of the limit argument as in Lemma \ref{conv:rho_slsi}). The claim now follows from the fact that
    \begin{align*}
        \nabla (\rho\, e^V)\,e^{-\frac{\mathbb{V}}{2}}\chi_R
        =\nabla (\rho \,e^{V-\frac{\mathbb{V}}{2}}\chi_R)-\rho\, e^{V-\frac{\mathbb{V}}{2}}\nabla(\chi_R\, e^{-\frac{\mathbb{V}}{2}})\,e^{\frac{\mathbb{V}}{2}}.
    \end{align*}
    
    \underline{The weak convergence \ref{conv:weak_convergences_slsi_1}.} From \ref{est_gf_dis} of Lemma \ref{lem:gf_est} and the Banach-Alaoglu theorem, we may assume that we have $(\nabla (\rho^\sigma e^V)\, e^{-\frac{\mathbb{V}}{2}})\ast\omega_\sigma\rightharpoonup \mu$ as $\sigma\rightarrow 0$ weakly in $L^2([0,T]\times \mathbb{R}^{d};\,\mathbb{R}^{d})$. Take any test function $\psi\in C_c^\infty([0,T]\times B_R;\,\mathbb{R}^{d})$ and compute
    \begin{align*}
        &\int_0^T \int_{\mathbb{R}^{d}} \mu\cdot\psi\,\mathrm{d} x\,\mathrm{d} t
        =\lim_{\sigma\rightarrow 0}\int_0^T \int_{\mathbb{R}^{d}} (\nabla (\rho^\sigma e^V)\, e^{-\frac{\mathbb{V}}{2}})\ast\omega_\sigma\cdot\psi\,\mathrm{d} x\,\mathrm{d} t\\
        &=\lim_{\sigma\rightarrow 0}\int_0^T \int_{\mathbb{R}^{d}} (\nabla (\rho^\sigma e^V)\, e^{-\frac{\mathbb{V}}{2}}\chi_R)\ast\omega_\sigma\cdot\psi\,\mathrm{d} x\,\mathrm{d} t
        =\int_0^T \int_{\mathbb{R}^{d}} \nabla (\rho^\sigma e^V)\, e^{-\frac{\mathbb{V}}{2}}\cdot\psi\,\mathrm{d} x\,\mathrm{d} t,
    \end{align*}
    where we have used the commutator argument from Lemma \ref{lem:h-1_mod_bound} in conjunction with \ref{bnd:h-1} from Lemma \ref{lem:estimates_slsi}, the fact that $\psi=\psi\, \chi_R$ and the weak convergence \ref{conv:weak_convergences_slsi}. By the fact that $\psi$ and $R$ are arbitrary, we deduce that $\mu=\nabla (\rho^\sigma e^V)\, e^{-\frac{\mathbb{V}}{2}}$.
\end{proof}

\begin{lem}[Weak convergence]\label{lem:weak_convergences}
    Let Assumption \ref{assumptions:pot} hold and let $\{\varrho^\sigma\}_\sigma$ be the weak solutions to \eqref{eq:mf_svgd} and let $\varrho$ be the corresponding curve of probability measures in Lemma~\ref{lem:basic_conv}. Then, we have the weak convergence $(\nabla\varrho^\sigma+\varrho^\sigma\nabla V)\ast\omega_\sigma\rightharpoonup \nabla\varrho+\varrho\nabla V$ as $\sigma\rightarrow 0$ in $L^2((0,T)\times \mathbb{R}^{d})$.
\end{lem}
\begin{proof}
    By \ref{est_gf_dis} of Lemma \ref{lem:gf_est} we know that $\{(\nabla\varrho^\sigma+\varrho^\sigma\nabla V)\ast\omega_\sigma\}_\sigma$ is uniformly bounded in $L^2((0,T)\times B_R)$. Consequently, from Lemma \ref{prop:bounds_gradient} we know that both $\{\nabla\varrho^\sigma\ast\omega_\sigma\}_\sigma$ and $\{\varrho^\sigma\nabla V\ast\omega_\sigma\}_\sigma$ are uniformly bounded in $L^2((0,T)\times B_R)$.
    
    \smallskip
    
    By Banach--Alaoglu Theorem, we can extract a subsequence (not relabeled) and two functions $\mu,\nu\in L^2([0,1]\times B_R)$ such that $\nabla\varrho^\sigma\ast\omega_\sigma \rightharpoonup \mu$ and $\varrho^\sigma\nabla V\ast\omega_\sigma \rightharpoonup \nu$ as $\sigma\rightarrow 0$ in $L^2([0,T]\times B_R)$. Let $\psi\in C_c^\infty([0,T]\times B_R)$ be a test function. We compute for $i\in \{1,\ldots,d\}$
    \begin{align*}
        \int_0^T \int_{B_R}
        (\partial_i \varrho^\sigma\ast\omega_\sigma)\;\psi\,\mathrm{d} x\,\mathrm{d} t
        =-\int_0^T \int_{B_R}
        \varrho^\sigma \; (\partial_i \psi\ast\omega_\sigma)\,\mathrm{d} x\,\mathrm{d} t.
    \end{align*}
    To pass to the limit, we must apply a commutator argument. We estimate the difference by using the fact that $\partial_i\psi=\partial_i\psi\;\chi_R$, the Young convolution inequality and the change of variables formula to estimate
    \begin{align*}
        &\|((\partial_i \psi\;\chi_R)\ast\omega_\sigma)-\chi_R\, (\partial_i \psi\ast\omega_\sigma)\|_{L_{t,x}^\infty}
        \leq \|\chi_R\|_{W^{1,\infty}}\||\partial_i \psi|\ast |y\,\omega_\sigma|\|_{L_{t,x}^\infty}\\
        &\leq \|\chi_R\|_{W^{1,\infty}}\|\partial_i \psi\|_{L_{t,x}^\infty}\|y\,\omega_\sigma\|_{L^1}
        \leq C\,\sigma,
    \end{align*}
    where $C>0$ is a constant independent of $\sigma$ and $i$. Then, as $\varrho^\sigma\rightharpoonup \varrho$ in $L^1([0,T]\times B_{2R})$ by \ref{rem:l1_weak_conv} of Lemma \ref{lem:basic_conv}, and $(\partial_i \psi\ast\omega_\sigma)\rightarrow \partial_i \psi$ in $L^\infty([0,T]\times B_{2R})$, we can derive the limit
    \begin{align*}
        &\int_0^T \int_{B_R}
        \mu\;\psi\,\mathrm{d} x\,\mathrm{d} t
        =\lim_{\sigma\rightarrow 0}\int_0^T \int_{B_R}
        (\partial_i \varrho^\sigma\ast\omega_\sigma)\;\psi\,\mathrm{d} x\,\mathrm{d} t\\
        &=\lim_{\sigma\rightarrow 0}-\int_0^T \int_{\mathbb{R}^{d}}
        \varrho^\sigma \; ((\partial_i \psi\;\chi_R)\ast\omega_\sigma)\,\mathrm{d} x\,\mathrm{d} t
        =\lim_{\sigma\rightarrow 0}-\int_0^T \int_{B_{2R}}
        \varrho^\sigma \; \chi_R(\partial_i \psi\ast\omega_\sigma)\,\mathrm{d} x\,\mathrm{d} t\\
        &=\int_0^T \int_{B_R}
        \partial_i \varrho\;\psi\,\mathrm{d} x\,\mathrm{d} t.
    \end{align*}
    As $\psi$ is arbitrary, we have shown that $\mu=\nabla \varrho$.
    
    \smallskip
    
    The same argument applies to $\{(\varrho^\sigma
    \nabla V)\ast\omega_\sigma\}_\sigma$. Taking any $\psi\in C_c^\infty([0,T]\times B_R)$ and repeating the previous computations, we have
    \begin{align*}
        \int_0^T\int_{B_R} ((\varrho^\sigma
        \nabla V)\ast\omega_\sigma)\,\psi\,\mathrm{d} x\,\mathrm{d} t
        = \int_0^T\int_{\mathbb{R}^{d}} \varrho^\sigma
        \nabla V\, (\psi\ast\omega_\sigma)\,\mathrm{d} x\,\mathrm{d} t
        \rightarrow \int_0^T\int_{B_R} \varrho\nabla V\,\psi \,\mathrm{d} x\,\mathrm{d} t,
    \end{align*}
    because $\varrho^\sigma
    \nabla V \rightharpoonup \varrho
    \nabla V$ as $\sigma\rightarrow 0$ in $L^1([0,T]\times \mathbb{R}^{d})$ by Lemma \ref{lem:weak_conv_of_rho_nablaV}. Therefore $\nu=\varrho\nabla V$. Adding the two weak limits, we conclude the desired weak limit. Since, every subsequence we could have chosen yields the same limit, the whole sequence must weakly converge. This completes the proof.
\end{proof}


\section{Proof of Theorem \ref{thm:ex_weak_sol_loc_pde} and Theorem \ref{thm:ex_weak_sol_nonloc_pde}}

\label{sec:main_proof}

We only show the proof for Theorem \ref{thm:ex_weak_sol_loc_pde}. The proof for Theorem \ref{thm:ex_weak_sol_nonloc_pde} is exactly analogue, up to small differences that we point out. Let us give a proof strategy. We divide the proof into 5 steps. In Step 1, we prove the convergence of solutions to \eqref{eq:mf_slsi_svgd} towards a solution to the equation \eqref{eq:lim_mf_slsi_svgd} (respectively solutions of \eqref{eq:mf_svgd} towards a solution to the equation \eqref{eq:lim_mf_svgd}). Then in Step 2, we prove the energy dissipation inequalities \eqref{eq:edi_slsi} (respectively \eqref{eq:lim_edi}). In Step 3, we prove the regularity claims \eqref{eq:regularity_weak_sol_slsi}--\eqref{eq:regularity_weak_sol_slsi_3} (respectively \eqref{eq:regularity_weak_sol}--\eqref{eq:regularity_weak_sol_3}). Finally, in Steps 4 and 5 we focus on the weak solution of \eqref{eq:lim_mf_slsi_svgd}, where we are able to prove the exponential decay in the Kullback-Leibler divergence in \eqref{eq:edi_slsi}, as well as the absolute continuity in the Wasserstein metric.

\underline{Step 1: Existence of solutions.}
In this step, we will show that the limit $\rho$ of Lemma \ref{lem:basic_conv} satisfies for all $T,R>0$ and any test functions $\psi\in C_c^\infty([0,T]\times B_R)$ the equation
\begin{align}
    \label{eq:weak_eq_to_pass_lim_target}
    \begin{split}
        &\int_{\mathbb{R}^{d}} \psi_T \rho_T\,\mathrm{d} x-\int_{\mathbb{R}^{d}} \psi_0 \rho_0\,\mathrm{d} x\\
        &=
        \int_0^T \int_{\mathbb{R}^{d}} \partial_t\psi\; \rho_t\,\mathrm{d} x\,\mathrm{d} t
        -\int_0^T\int_{\mathbb{R}^{d}} \rho e^{V-\frac{\mathbb{V}}{2}}\nabla \psi\cdot \nabla (\rho e^V) e^{-\frac{\mathbb{V}}{2}}\,\mathrm{d} x\,\mathrm{d} t.
    \end{split}
\end{align}
Indeed, the weak formulation of \eqref{eq:mf_slsi_svgd} gives us the equation
\begin{align}
    \label{eq:weak_eq_to_pass_lim}
    \begin{split}
        &\int_{\mathbb{R}^{d}} \psi_T \rho_T^\sigma\,\mathrm{d} x-\int_{\mathbb{R}^{d}} \psi_0 \rho_0\,\mathrm{d} x\\
        &=
        \int_0^T \int_{\mathbb{R}^{d}} \partial_t\psi\; \rho_t^\sigma\,\mathrm{d} x\,\mathrm{d} t
        -\int_0^T\int_{\mathbb{R}^{d}} (\rho^\sigma e^{V-\frac{\mathbb{V}}{2}}\nabla \psi)\ast \omega_\sigma\cdot (\nabla (\rho^\sigma e^V) e^{-\frac{\mathbb{V}}{2}})\ast \omega_\sigma\,\mathrm{d} x \,\mathrm{d} t.
    \end{split}
\end{align}
Then, the first term in \eqref{eq:weak_eq_to_pass_lim} converges by \ref{bnd:conv_l1} from Lemma \ref{lem:basic_conv}, and the third term in \eqref{eq:weak_eq_to_pass_lim} converges by \ref{rem:l1_weak_conv} of Lemma \ref{lem:basic_conv}. This means we only need to justify the limit for the last term in \eqref{eq:weak_eq_to_pass_lim}. Because $\{((\nabla\rho^\sigma+\rho^\sigma\nabla V)e^{V-\frac{\mathbb{V}}{2}})\ast\omega_\sigma\}_\sigma$ is uniformly bounded in $L^2([0,T]\times \mathbb{R}^{d})$ by \ref{est_gf_dis} from Lemma \ref{lem:gf_est}, we can employ the commutator argument \eqref{eq:conv_com_slsi} from Lemma \ref{lem:comm_est_slsi} (respectively \eqref{eq:commutator_representation} from Lemma \ref{lem:comm_est}). Thus, we get the identity
\begin{align*}
    &\lim_{\sigma\to 0}\int_0^T\int_{\mathbb{R}^{d}} (\rho^\sigma e^{V-\frac{\mathbb{V}}{2}}\nabla \psi)\ast \omega_\sigma\cdot (\nabla (\rho^\sigma e^V) e^{-\frac{\mathbb{V}}{2}})\ast \omega_\sigma\,\mathrm{d} x \,\mathrm{d} t\\
    &=\lim_{\sigma\rightarrow 0}\int_0^T\int_{B_R} \nabla \psi\;((\rho^\sigma e^{V-\frac{\mathbb{V}}{2}}\chi_R)\ast \omega_\sigma)\cdot (\nabla (\rho^\sigma e^V)e ^{-\frac{\mathbb{V}}{2}})\ast \omega_\sigma \,\mathrm{d} x\,\mathrm{d} t.
\end{align*}
Next, we observe that $(\rho^\sigma e^{V-\frac{\mathbb{V}}{2}}\chi_R)\ast\omega_\sigma \rightarrow \rho\, e^{V-\frac{\mathbb{V}}{2}}\chi_R$ as $\sigma\rightarrow 0$ in $L^2([0,T]\times B_R)$ by Lemma~\ref{conv:rho_slsi}, and that $(\nabla (\rho^\sigma e^V)e^{\frac{-\mathbb{V}}{2}})\ast\omega_\sigma\rightharpoonup \nabla (\rho\, e^V)e^{\frac{-\mathbb{V}}{2}}$ as $\sigma\rightarrow 0$ in $L^2([0,T]\times B_R)$ by \ref{conv:weak_convergences_slsi_1} of Lemma \ref{lem:weak_convergences_slsi} (respectively Lemma \ref{lem:weak_convergences}). Therefore, these convergences justify the limit for the last term in \eqref{eq:weak_eq_to_pass_lim}. Note that for any test function $\psi$, we can always choose $R>0$ sufficiently large such that $\psi=\psi\chi_R$. With this we have proven the identity \eqref{eq:weak_eq_to_pass_lim_target} as desired.

\smallskip

\underline{Step 2: Energy dissipation inequality.} Using the energy dissipation inequality \eqref{eq:entropy_identity} from Lemma \ref{lem:gf_est}, we derive the following inequality, for all $\sigma\in (0,1)$,
\begin{align*}
    \mathrm{KL}(\rho_T^\sigma\,||\,\rho_\infty)+\big\|\big((\nabla\rho^\sigma+\rho^\sigma\nabla V)e^{V-\frac{\mathbb{V}}{2}}\big)\ast\omega_\sigma\big\|_{L_{t,x}^2}^2\leq \mathrm{KL}(\rho_0\,||\,\rho_\infty).
\end{align*}
We show that when passing to the limit $\sigma\rightarrow 0$ the following inequality holds
\begin{equation}\label{eq:energy_dissipation_inequality_limit_scl_to_0}
    \mathrm{KL}(\rho_T\,||\,\rho_\infty)+\|(\nabla\rho+\rho\nabla V)e^{V-\frac{\mathbb{V}}{2}}\|_{L_{t,x}^2}^2\leq \mathrm{KL}(\rho_0\,||\,\rho_\infty).
\end{equation}
Indeed, for the Kullback-Leibler term we use \cite[Theorem 2.34]{ambrosio_fusco_pallara_2000_functions} and \ref{bnd:conv_l1} of Lemma \ref{lem:basic_conv} to obtain
\begin{align*}
    \mathrm{KL}(\rho_T || \rho_\infty)
    \leq \liminf_{\sigma\rightarrow 0}\mathrm{KL}(\rho_T^\sigma || \rho_\infty).
\end{align*}
For the second term, we have $((\nabla\rho^\sigma+\rho^\sigma\nabla V)e^{V-\frac{\mathbb{V}}{2}})\ast\omega_\sigma\rightharpoonup (\nabla\rho+\rho\nabla V)\, e^{V-\frac{\mathbb{V}}{2}}$ as $\sigma\rightarrow 0$ in $L^2([0,T]\times \mathbb{R}^{d})$ by \ref{conv:weak_convergences_slsi_1} in Lemma \ref{lem:weak_convergences_slsi} (respectively Lemma \ref{lem:weak_convergences}) so the weak lower semincontinuity of a norm implies
\begin{align*}
    \|(\nabla\rho+\rho\nabla V)e^{V-\frac{\mathbb{V}}{2}}\|_{L_{t,x}^2}^2\leq \liminf_{\sigma\rightarrow 0}\|((\nabla\rho^\sigma+\rho^\sigma\nabla V)e^{V-\frac{\mathbb{V}}{2}})\ast\omega_\sigma\|_{L_{t,x}^2}^2.
\end{align*}

\smallskip

\underline{Step 3: Regularity properties.} The first property of \eqref{eq:regularity_weak_sol_slsi} (respectively \eqref{eq:regularity_weak_sol}) follows from conservation of mass, the remaining follow from the energy dissipation inequality \eqref{eq:energy_dissipation_inequality_limit_scl_to_0} and Lemma \ref{lem:control_neg_log}. The properties \eqref{eq:regularity_weak_sol_slsi_2} (respectively \eqref{eq:regularity_weak_sol_2}) follow from the convergences of Lemma \ref{lem:weak_convergences_slsi} (respectively Lemma \ref{lem:weak_convergences}). The last property \eqref{eq:regularity_weak_sol_slsi_3} (respectively \eqref{eq:regularity_weak_sol_3}) follows from the convergence in Lemma \ref{conv:rho_slsi}, the uniform bound \ref{bnd:h1} in Lemma \ref{lem:estimates_slsi} (respectively Lemma \ref{prop:bounds_gradient}) and the Banach--Alaoglu Theorem.

\underline{Step 4: Exponential convergence under Kullback-Leibler divergence.} This step only applies to the weak solution of \eqref{eq:lim_mf_slsi_svgd}. We consider a fixed $\sigma\in (0,1)$. We multiply the equation \eqref{eq:mf_slsi_svgd} with $\ln(\rho^\sigma e^V)$, integrate on $\mathbb{R}^{d}$ and use integration by parts to compute
\begin{align*}
    \frac{\,\mathrm{d}}{\,\mathrm{d} t}\mathrm{KL}(\rho_t^\sigma\,||\,\rho_\infty)
    =-\|(\nabla\rho_t^\sigma+\rho_t^\sigma\nabla V)e^{V-\frac{\mathbb{V}}{2}}\ast\omega_\sigma\|_{L_{x}^2}^2
    =-\mathbb{D}^2(\rho_t^\sigma\,||\,\rho_\infty).
\end{align*}
We observe that \ref{bnd:rho} and \ref{bnd:h-1} from Lemma \ref{lem:estimates_slsi} imply that
$$
\rho_t^\sigma\, e^{V-\frac{\mathbb{V}}{2}} \, \frac{\nabla \mathbb{V}}{2} = \nabla(\rho_t^\sigma\, e^V) e^{-\frac{\mathbb{V}}{2}} - \nabla(\rho_t^\sigma\, e^{V-\frac{\mathbb{V}}{2}}) \in H^{-1}(\mathbb{R}^{d}) \mbox{ for a.e. } t.
$$
Therefore, we can apply Theorem \ref{thm:slsi} to obtain the Stein-log-Sobolev inequality
\begin{align*}
    -\lambda_\sigma\,\mathrm{KL}(\rho_t^\sigma\,||\,\rho_\infty)
    \leq \frac{\,\mathrm{d}}{\,\mathrm{d} t}\mathrm{KL}(\rho_t^\sigma\,||\,\rho_\infty),
\end{align*}
where we define $\lambda_\sigma:=\lambda_0/\mathcal{D}_\sigma$ for some constant $\lambda_0$, depending only on $V$ and $\mathbb{V}$, and a constant $\mathcal{D}_\sigma\geq 1$ that satisfies
\begin{align*}
    \frac{1}{\mathcal{D}_\sigma}\frac{1}{1+|\xi|^2}
    \leq \hat{k}_\sigma(\xi).
\end{align*}
Then, by using the Gr\"onwall inequality gives
\begin{align*}
    \mathrm{KL}(\rho_t^\sigma \,||\; \rho_t^\infty)\leq e^{-\lambda_\sigma t}\,\mathrm{KL}(\rho_0\,||\,\rho_\infty).
\end{align*}
The last step is to determine the behavior of $\lambda_\sigma$ as $\sigma\to 0$. Since $\omega\in L^1(\mathbb{R}^{d})$ and $\int_{\mathbb{R}^{d}}\omega\,\mathrm{d} x~=~1$, we have that $\hat{k}$ is continuous, $\hat{k}(0)=1$ and $\widehat{k_\sigma}(x)=\hat{k}(\sigma x)$. Then, by Lemma \ref{lem:optimal_slsi_val} and \eqref{ass:fourier} from Assumption \ref{assumptions:slsi}, there exists for all $\varepsilon\in (0,1)$, some $\sigma_\ast\in (0,1)$ such that for all $\sigma\in (0,\sigma_\ast)$ we have that
\begin{align*}
    \hat{k}_\sigma(\xi)\geq \frac{1-\varepsilon}{1+|\xi|^2}.
\end{align*}
This shows that as $\sigma\to 0$ we have that $\lambda_\sigma\to\lambda_0$. Finally, using the same lower continuity argument as in the previous step, we immediately deduce by passing to the limit $\sigma\rightarrow 0$ that
\begin{align*}
    \mathrm{KL}(\rho_t \,||\; \rho_t^\infty)\leq\liminf_{\sigma \rightarrow 0}\mathrm{KL}(\rho_t^\sigma \,||\; \rho_t^\infty)\leq e^{-\lambda_0 t}\,\mathrm{KL}(\rho_0\,||\,\rho_\infty).
\end{align*}

\underline{Step 5: Absolute continuity of solution with respect to $W_1$.} This step only applies to the weak solution of \eqref{eq:lim_mf_slsi_svgd}. First, we observe from \eqref{eq:regularity_weak_sol_slsi} and the fact that $V\geq \mathbb{V}$, that
\begin{align*}
    \sup_{t\in [0,T]}\int_{\mathbb{R}^{d}} |x|\,\mathrm{d} \rho_t(x)
    <\infty.
\end{align*}
In addition, by \eqref{eq:regularity_weak_sol_slsi_2} and the H\"older inequality, $\rho \, e^{V-\frac{\mathbb{V}}{2}}\; \nabla (\rho \, e^V)e^{-\frac{\mathbb{V}}{2}}\in L^1([0,T]\times\mathbb{R}^{d})$. Let $0\leq t_1 \leq t_2$ and let $\psi\in C_c^\infty(\mathbb{R}^{d})$ any test function, then from \eqref{eq:weak_sol_loc_pde} we easily deduce
\begin{equation}\label{eq:weak_form_simplified_only_space_all_times}
    \int_{\mathbb{R}^{d}}\psi(x) \rho_{t_2}(x) \,\mathrm{d} x - \int_{\mathbb{R}^{d}}\psi(x) \rho_{t_1}(x) \,\mathrm{d} x=
    -\int_{t_1}^{t_2} \int_{\mathbb{R}^{d}} \rho \,e^{V-\frac{\mathbb{V}}{2}}\; \nabla (\rho\, e^V)e^{-\frac{\mathbb{V}}{2}}\cdot \nabla \psi\,\mathrm{d} x\,\mathrm{d} s.
\end{equation}
Indeed, \eqref{eq:weak_form_simplified_only_space_all_times} can be proved by evaluating \eqref{eq:weak_sol_loc_pde} with test function $\psi$ on $[0,t_2]$ and $[0,t_1]$ and taking the difference. We extend the identity to all $\psi$ that are $1$-Lipschitz. Indeed, take any sequence $\{\psi_n\}_{n\geq 1} \subset C_c^{\infty}(\mathbb{R}^{d})$ such that $\psi_n\rightarrow\psi$, $\nabla\psi_n\rightarrow\nabla\psi$ almost everywhere and $\|\nabla\psi\|_{L^\infty}\leq 1$. We may assume $|\psi_n(x)|\leq 2|\psi(0)|+2|x|$ so that the limit $\int_{\mathbb{R}^{d}}\psi_n(x)\, \rho_{t}(x) \,\mathrm{d} x \to \int_{\mathbb{R}^{d}}\psi(x)\, \rho_{t}(x) \,\mathrm{d} x$ as $n\to\infty$ holds by the dominated convergence. We can now take the absolute value in \eqref{eq:weak_form_simplified_only_space_all_times} and pass to the supremum over all $\psi\in \mathrm{Lip}(\mathbb{R}^{d})$ with $\mathrm{Lip}(\psi)\leq 1$ to obtain the absolute contiuity of the curve
\begin{align*}
    W_1(\rho_{t_1},\rho_{t_2})
    \leq \int_{t_1}^{t_2} \int_{\mathbb{R}^{d}} \big\vert \rho\, e^{V-\frac{\mathbb{V}}{2}}\; \nabla (\rho\, e^V)e^{-\frac{\mathbb{V}}{2}} \big\vert \,\mathrm{d} x \,\mathrm{d} t \qquad \mbox{ for all } t_1, t_2 \in [0,T].
\end{align*}
With this we conclude the proof. \hfill $\Box$

\appendix


\section{The Bessel potential of order $\alpha$}
\label{app:bessel_pot}

Let $d\ge 1$ and $\alpha>0$. We use the standard definition \cite[Equation (4, 1)]{aronszajn_smith_1961_theory}
\begin{align*}
    G_\alpha(x)
    = \frac{1}{2^{\frac{d+\alpha-2}{2}}\pi^{\frac{d}{2}}\Gamma\!\left(\frac{\alpha}{2}\right)}
    \,K_{\frac{d-\alpha}{2}}\!(|x|)\,|x|^{\frac{\alpha-d}{2}},
\end{align*}
with $K_\nu$ the modified Bessel function of the second kind. Its Fourier transform and semigroup property \cite[Equation (4, 6) and (4, 7)]{aronszajn_smith_1961_theory} are
\begin{align*}
    \widehat{G_\alpha}(\xi)=\frac{1}{(1+4\pi^2|\xi|^2)^{\alpha/2}},
    \qquad
    G_{\alpha+\beta}=G_\alpha\ast G_\beta\quad\text{for}\quad\alpha,\beta>0.
\end{align*}
Hence the fundamental solution $k$ of the screened Poisson equation $(-\Delta)k+k=\delta_0$ satisfies $\hat k(\xi)=(1+4\pi^2|\xi|^2)^{-1}$ and therefore $k=G_2=G_1\ast G_1$. Throughout we set $\omega:=G_1$.

The basic asymptotics of $\omega$ are up to multiplicative constants (depending only on $d$) \cite[Equation (4.2)]{aronszajn_smith_1961_theory}
\begin{align*}
    \omega(x)\sim
    \begin{cases}
        -\log|x|, & d=1,\\
        |x|^{\,1-d}, & d>1,
    \end{cases}
    \quad (|x|\to 0),
    \qquad\qquad
    \omega(x)\sim |x|^{-\frac d2}\,e^{-\frac{|x|}{2\pi}}
    \quad (|x|\to\infty).
\end{align*}

We derive an exact radial formula for $\nabla\omega$. Writing $r=|x|$ and using the classic identity $K_\nu'(z)=-K_{\nu+1}(z)+\frac{\nu}{z}K_\nu(z)$,
\begin{align*}
    \nabla\omega(x)
    = \frac{x}{r}\,\frac{d}{dr}\!\left(\frac{1}{2^{\frac{d-1}{2}}\pi^{\frac{d}{2}}\Gamma\!\left(\frac{1}{2}\right)}\,r^{\frac{1-d}{2}}K_{\frac{d-1}{2}}(r)\right)
    = -\frac{1}{2^{\frac{d-1}{2}}\pi^{\frac{d}{2}}\Gamma\!\left(\frac{1}{2}\right)}\,
    r^{-\frac{d+1}{2}}\,K_{\frac{d+1}{2}}(r)\;x.
\end{align*}

The asymptotic properties of the the gradient thus follow from those of $K_{\frac{d+1}{2}}$ \cite[Equation (3, 4) and (3, 5)]{aronszajn_smith_1961_theory}. The leading orders (again up to multiplicative constants) are
\begin{align*}
    |\nabla\omega(x)|
    \sim |x|^{-d} \quad (|x|\to 0),
    \qquad\qquad
    |\nabla\omega(x)|
    \sim |x|^{-\frac d2}\,e^{-|x|} \quad (|x|\to\infty).
\end{align*}

These two asymptotic behaviors immediately give the $L^p$ information. For $\omega$,
\begin{align*}
    \omega\in
    \begin{cases}
        L^p(\mathbb{R}^d) \ \text{for all } p\in[1,\infty), & d=1,\\[1mm]
        L^p(\mathbb{R}^d) \ \text{iff } p\in\left[1,\frac{d}{d-1}\right), & d>1,
    \end{cases}
\end{align*}
since the only restriction comes from the singularity at the origin, while the tail is exponentially integrable. For the gradient, the near–$0$ behaviour $|\nabla\omega(x)|\sim |x|^{-d}$ shows that $\nabla\omega\notin L^p_{\mathrm{loc}}(\mathbb{R}^d)$ for every $p\ge 1$ (and is exponentially integrable at infinity for all $p$). In particular, we have that $\omega\notin W^{1,p}(\mathbb{R}^d)$ for every $p\in[1,\infty)$.


\section{Existence of weak solutions to \eqref{eq:mf_svgd}}
\label{app:weak_sol}

The goal of this section is to prove the following theorem. This extends the existence result of \cite[Theorem 2.4]{lu_lu_nolen_2019_scaling} to kernels with lower regularity.

\begin{thm}
    \label{thm:mf_scl}
    Under Assumption \ref{assumptions:pot}, there exists a weak solution $\varrho^\sigma\in C^0([0,\infty);\,\mathscr{P}(\mathbb{R}^{d}))$ to the following PDE
    \begin{align}
        \label{eq:mf_scl}
        \left\{\begin{array}{ll}
        \partial_t \varrho^\sigma
        =\mathrm{div}(\varrho^\sigma\; k_\sigma\ast (\nabla\varrho^\sigma+\varrho^\sigma\nabla V)) & \text{ on }(0,\infty)\times\mathbb{R}^{d},  \\
        \varrho^\sigma =\varrho_0 & \text{ on }\{0\}\times\mathbb{R}^{d},
    \end{array}\right.
\end{align}
where $\sigma\in (0,1)$, $k_\sigma:=\omega_\sigma\ast\omega_\sigma$ and $\omega_\sigma(x):=\sigma^{-d}\omega(x/\sigma)$.
\end{thm}
We prove existence by approximating with a mollifier $\varphi_\delta$, where $\varphi\in C_c^\infty(B_1)$, $\varphi\geq 0$, $\int_{\mathbb{R}^{d}}\varphi\,\mathrm{d} x=1$ and $\varphi_\delta(x)=\delta^{-d}\varphi(x/\delta)$ for $\delta\in (0,1)$. We define the mollified data
\begin{align*}
    V^\delta:=V\ast \varphi_\delta,\quad
    \varrho_\infty^\delta\,\propto\, e^{-V^\delta},\quad
    \omega_{\sigma,\delta}
    :=\omega_\sigma\ast \varphi_\delta,\quad
    k_{\sigma,\delta}:=\omega_{\sigma,\delta}\ast \omega_{\sigma,\delta}.
\end{align*}
We approximate solutions to \eqref{eq:mf_scl} by solutions of
\begin{align}
    \label{eq:mf_scl_delta}
    \left\{\begin{array}{ll}
    \partial_t \varrho^{\sigma,\delta}
    =\mathrm{div}(\varrho^{\sigma,\delta}\; k_{\sigma,\delta}\ast (\nabla\varrho^{\sigma,\delta}+\varrho^{\sigma,\delta}\nabla V^\delta)) & \text{ on }(0,\infty)\times\mathbb{R}^{d},  \\
    \varrho^{\sigma,\delta} =\varrho_0 & \text{ on }\{0\}\times\mathbb{R}^{d}.
\end{array}\right.
\end{align}

\begin{lem}[{Existence of classical solutions to \eqref{eq:mf_scl_delta}}]
    \label{lem:ex_cls_sol_mf_scl_del}
    Let Assumption \ref{assumptions:pot} hold. Then, there exists a unique weak solution $\varrho^{\sigma,\delta}\in C^0([0,\infty);\,\mathscr{P}(\mathbb{R}^{d}))$ to \eqref{eq:mf_scl_delta}. Moreover, for all $T>0$, we have
    \begin{align*}
        &(1+V^\delta)\varrho^{\sigma,\delta}\in L^\infty(0,T;\,L^1(\mathbb{R}^{d})).
    \end{align*}
\end{lem}
\begin{proof}
    The goal is to prove that the mollified data satisfies the conditions of \cite[Theorem 2.4]{lu_lu_nolen_2019_scaling}. To that end, we first check that $V^\delta$ satisfies the same properties as in \ref{ass:pot} in Assumption \ref{assumptions:pot}. Indeed, we verify each property separably.
    \begin{itemize}
        \item We have $V^\delta\in C^\infty(\mathbb{R}^{d})$, $V^\delta\geq 0$ and $V^\delta(x)\rightarrow \infty$ whenever $|x|\rightarrow \infty$.
        \item We use \eqref{eq:bound_nabla_V_by_V} and the Jensen inequality to show
        \begin{align*}
            &|\nabla V^\delta(x)|^{p_0}
            \leq\int_{\mathbb{R}^{d}} |\nabla V(x-y)|^{p_0} \varphi_\delta(y)\,\mathrm{d} y \\
            &=C_V\int_{\mathbb{R}^{d}} (1+V(x-y)) \varphi_\delta(y)\,\mathrm{d} y
            =C_V(1+V^\delta(x)).
        \end{align*}
        \item Take any two values $x,y\in \mathbb{R}^{d}$ and $\theta\in [0,1]$. Then we use \eqref{eq:bound_hessian_V} and the Jensen inequality to compute
        \begin{align*}
            &|\nabla^2 V^\delta((1-\theta) x+\theta y)|^{p_0}
            \leq \int_{\mathbb{R}^{d}} |\nabla^2 V((1-\theta) x+\theta y-z)|^{p_0} \varphi_\delta(z)\,\mathrm{d} z\\
            &\leq C_V\int_{\mathbb{R}^{d}} (1+V(x-z)+V(y-z)) \varphi_\delta(z)\,\mathrm{d} z
            \leq C_V\,(1+V^\delta(x)+V^\delta(y)).
        \end{align*}
        \item Take any $\alpha,\beta>0$ and pick any vectors $x,y\in\mathbb{R}^{d}$ with $|y| <\alpha|x|+\beta$. Then, by using \eqref{eq:bound_nabla_hessian_V} and the fact that $\varphi$ is supported on $B_1$
        \begin{align*}
            &(1+|x|)\,(|\nabla V^\delta (y)|+|\nabla^2 V^\delta(y)|)\\
            &\leq 2\int_{\mathbb{R}^{d}} (1+|x-z|)\,(|\nabla V(y-z)|+|\nabla^2 V(y-z)|)\varphi_\delta(z)\,\mathrm{d} z\\
            &\leq 2C_{\alpha,\alpha+\beta+1}\int_{\mathbb{R}^{d}} (1+V(x-z))\varphi_\delta(z)\,\mathrm{d} z
            =2C_{\alpha,\alpha+\beta+1}(1+V^\delta(x)).
        \end{align*}
    \end{itemize}
    Secondly, we observe by the Young convolution inequality that $k_{\sigma,\delta}\in C^4(\mathbb{R}^{d})$ with bounded derivatives. Lastly, we observe the growth condition with \eqref{eq:bound_nabla_V_by_V}
    \begin{align*}
        &\left|\frac{d}{dt}(1+V(x+th))^{\frac{p_0-1}{p_0}}\right|
        \leq \frac{p_0-1}{p_0}(1+V(x+th))^{-\frac{1}{p_0}}|\nabla V(x+th)|\,|h|\\
        &\leq \frac{p_0-1}{p_0}(1+V(x+th))^{-\frac{1}{p_0}} C_V^{\frac{1}{p_0}}\;(1+V(x+th))^{\frac{1}{p_0}}\,|h|
        =\frac{p_0-1}{p_0}C_V^{\frac{1}{p_0}}\,|h|.
    \end{align*}
    This means we get
    \begin{align*}
        (1+V(x+h))^{\frac{p_0-1}{p_0}}
        \leq (1+V(x))^{\frac{p_0-1}{p_0}}+\frac{p_0-1}{p_0}C_V^{\frac{1}{p_0}}\,|h|.
    \end{align*}
    If we raise to the power of $\tfrac{p_0}{p_0-1}>1$, then we can show
    \begin{align*}
        V(x+h)\leq C(1+V(x)+|h|^{\frac{p_0}{p_0-1}}),
    \end{align*}
    for some constant $C>0$ independent of $h$ and $\delta$. We have thus shown
    \begin{align}
        \label{eq:moll_V_bnd_by_V}
        V^\delta\leq \kappa\,(1+V)
    \end{align}
    for some constant $\kappa>0$ independent of $\delta$. In this case, we can compute the initial value
    \begin{align*}
        \int_{\mathbb{R}^{d}} V^\delta \,\mathrm{d} \varrho_0
        \leq \kappa\int_{\mathbb{R}^{d}} (1+V) \,\mathrm{d}\varrho_0
        <\infty.
    \end{align*}
    Thus all the conditions of \cite[Theorem 2.4]{lu_lu_nolen_2019_scaling} have been met and the proof is concluded.
\end{proof}

\begin{lem}
    \label{lem:unif_bnds_approx}
    Let Assumption \ref{assumptions:pot} hold and let $\{\varrho^{\sigma,\delta}\}_\delta$ be the weak solutions to \eqref{eq:mf_scl_delta}. Then, for any fixed $T,R>0$ and $q\in\mathcal{I}_d$ and $r>2$ the following assertions are true:
    \begin{enumerate}[label=(T\arabic*)]
        \item\label{bnd:mf_scl_del_kl} $\{\varrho^{\sigma,\delta}\}_\delta$, $\{\varrho^{\sigma,\delta} V^\delta\}_\delta$, $\{\varrho^{\sigma,\delta}\,|\log\varrho^{\sigma,\delta}|\}_\delta$ are uniformly bounded in $L^\infty(0,T;\,L^1(\mathbb{R}^{d}))$;
        \item\label{bnd:mf_scl_del_dis} $\{(\nabla\varrho^{\sigma,\delta}+\varrho^{\sigma,\delta}\nabla V^\delta)\ast \omega_{\sigma,\delta}\}_\delta$ is uniformly bounded in $L^2((0,T)\times\mathbb{R}^{d})$;
        \item\label{bnd:mf_scl_del_nabla} $\{\varrho^{\sigma,\delta}\ast \omega_{\sigma,\delta}\}_\delta$ is uniformly bounded in $L^2(0,T;H^1(B_R))$;
        \item\label{bnd:mf_scl_del_lp} $\{\varrho^{\sigma,\delta}\ast \omega_{\sigma,\delta}\}_\delta$ is uniformly bounded in $L^2(0,T;L^q(B_R))$;
        \item\label{bnd:mf_scl_del_dt} $\{\partial_t \varrho^{\sigma,\delta}\}_\delta$ is uniformly bounded in $(L^r(0,T; W^{1,\infty}(B_R)))^\ast$.
    \end{enumerate}
\end{lem}
\begin{proof}
    The proofs of \ref{bnd:mf_scl_del_kl}--\ref{bnd:mf_scl_del_lp} are exactly the same as the proofs of \ref{est_gf_kl}, \ref{est_gf_dis} of Lemma \ref{lem:gf_est}, Lemma \ref{prop:bounds_gradient}, \ref{bnd:rap_int} of Lemma \ref{lem:ref_a_priori_est} respectively.
    
    We show \ref{bnd:mf_scl_del_dt} by duality. Take any test function $\psi\in C_c^\infty([0,T]\times \mathbb{R}^{d})$ and use the weak formulation \eqref{eq:mf_scl_delta} to compute
    \begin{align*}
        &\int_0^T\int_{\mathbb{R}^{d}}\partial_t\varrho^{\sigma,\delta}\;\psi\,\mathrm{d} x\,\mathrm{d} t
        =\int_0^T\int_{\mathbb{R}^{d}}(\varrho^{\sigma,\delta}\nabla\psi)\ast\omega_{\sigma,\delta}\cdot (\nabla\varrho^{\sigma,\delta}+\varrho^{\sigma,\delta}\nabla V^\delta)\ast\omega_{\sigma,\delta}\,\mathrm{d} x\,\mathrm{d} t\\
        &\leq \|(\nabla\varrho^{\sigma,\delta}+\varrho^{\sigma,\delta}\nabla V^\delta)\ast \omega_{\sigma,\delta}\|_{L_{t,x}^2}
        \|(\varrho^{\sigma,\delta}\nabla\psi)\ast\omega_{\sigma,\delta}\|_{L_{t,x}^2}.
    \end{align*}
    By the uniform bound \ref{bnd:mf_scl_del_dis}, it suffices to bound the second term. We estimate it by using the Young convolution inequality and the H\"older inequality
    \begin{align*}
        \|(\varrho^{\sigma,\delta}\nabla\psi)\ast\omega_{\sigma,\delta}\|_{L_{t,x}^2}
        \leq \|\varrho^{\sigma,\delta}\nabla\psi\|_{L_t^2 L_x^1}\|\omega_{\sigma,\delta}\|_{L^2}
        \leq \|\psi\|_{L_t^r W_x^{1,\infty}}\|\varrho^{\sigma,\delta}\|_{L_t^p L_x^1}\|\omega_{\sigma,\delta}\|_{L^2},
    \end{align*}
    where we use the fact that there exists $p\in(2,\infty)$ such that $\tfrac{2}{p}+\tfrac{2}{r}=1$. The uniform bound now follows from the fact that $\{\omega_{\sigma,\delta}\}_\delta$ is uniformly bounded in $L^2(\mathbb{R}^{d})$ and the uniform bound \ref{bnd:mf_scl_del_kl}.
\end{proof}

\begin{lem}[Convergences]
    \label{lem:conv_mf_scl_del}
    Let Assumption \ref{assumptions:pot} hold and let $\varrho^{\sigma,\delta}$ be the classical solution in Lemma \ref{lem:ex_cls_sol_mf_scl_del}. Then, for all $T, R>0$, for every $M>\frac{d}{2}+1$ and any test function $\psi\in C_c^\infty([0,T]\times B_R)$, we have up to passing to a subsequence (not relabeled) the following convergences
    \begin{enumerate}[label=(\roman*)]
        \item\label{conv:init_varrho} $\mathrm{KL}(\varrho_0\,||\,\varrho_\infty^\delta)\rightarrow \mathrm{KL}(\varrho_0 \,||\,\varrho_\infty)$ as $\delta\rightarrow 0$;
        
        \item\label{conv:varrho_scl_del_hm} $\varrho^{\sigma,\delta}\rightarrow\varrho^\sigma$ as $\delta\rightarrow 0$ in $C^0([0,T];\, H^{-M}(B_R))$;
        
        \item\label{conv:varrho_scl_del_pro} $\varrho_t^{\sigma,\delta}\rightarrow\varrho_t^\sigma$ as $\delta\rightarrow 0$ narrowly for all $t\in[0,T]$;
        \item\label{conv:varrho_scl_del_strong} $(\varrho^{\sigma,\delta} \psi)\ast\omega_{\sigma,\delta}\rightarrow(\varrho^\sigma\psi)\ast\omega_\sigma$ as $\delta\rightarrow 0$ in $L^2([0,T]\times \mathbb{R}^{d})$;
        \item\label{conv:varrho_scl_del_weak} $(\nabla\varrho^{\sigma,\delta}+\varrho^{\sigma,\delta}\nabla V^\delta)\ast\omega_{\sigma,\delta}\rightharpoonup(\nabla\varrho^\sigma+\varrho^\sigma\nabla V)\ast\omega_\sigma$ as $\delta\rightarrow 0$ weakly in $L^2([0,T]\times \mathbb{R}^{d})$.
    \end{enumerate}
\end{lem}
\begin{proof}
    We divide the proof into five steps. Each step proves one of the claims.
    
    \underline{Step 1: Proof of \ref{conv:init_varrho}.} This follows by \eqref{eq:moll_V_bnd_by_V} and the dominated convergence theorem.
    
    \underline{Step 2: Proof of \ref{conv:varrho_scl_del_hm}.} We follow the exact same proof strategy as for \ref{bnd:conv_cont_curve} from Lemma~\ref{lem:basic_conv}. We simply use the uniform bounds \ref{bnd:mf_scl_del_kl} and \ref{bnd:mf_scl_del_dt} from Lemma \ref{lem:unif_bnds_approx}, and the embedding $H_0^{M}(B_R)\hookrightarrow W_0^{1,\infty}(B_R)$.
    
    \underline{Step 3: Proof of \ref{conv:varrho_scl_del_pro}.}
    By multiplying the equation \eqref{eq:mf_scl_delta} with $\log(\varrho^{\sigma,\delta}e^{V^\delta})$, integrating on $[0,t]\times\mathbb{R}^{d}$, exploiting $k_{\sigma,\delta}=\omega_{\sigma,\delta}\ast\omega_{\sigma,\delta}$ and using integration by parts we get
    \begin{align*}
        \mathrm{KL}(\varrho_t^{\sigma,\delta}\,||\, \varrho_\infty^\delta) + \int_0^t \int_{\mathbb{R}^{d}} |(\nabla \varrho_s^{\sigma,\delta} + \varrho_s^{\sigma,\delta} \, \nabla V)\ast\omega_{\sigma,\delta}|^2 \,\mathrm{d} x\,\mathrm{d} s \leq \mathrm{KL}(\varrho_0 \,||\,\varrho_\infty^\delta).
    \end{align*}
    Thus, by using \ref{conv:init_varrho} we derive a uniform bound on the Kullback-Leibler divergence and get
    \begin{align*}
        \sup_{\delta\in (0,1)}\int_{\mathbb{R}^{d}}\varrho_t^{\sigma,\delta}|\log \varrho_t^{\sigma,\delta}|\,\mathrm{d} x
        +\int_{\mathbb{R}^{d}}\varrho_t^{\sigma,\delta}\, W\,\mathrm{d} x
        <\infty,
    \end{align*}
    where $W\in C^0(\mathbb{R}^{d})$ is a coercive function independent of $\delta$ with $V^\delta\geq W$ and is defined by
    \begin{align*}
        W(x):=\inf_{|x-z|<1} V(z).
    \end{align*}
    Thus, the sequence $\{\varrho_t^{\sigma,\delta}\}_\delta$ is equi-integrable and tight. From the Dunford-Pettis theorem, we have up to passing to some subsequence $\varrho_t^{\sigma,\delta}\rightharpoonup \mu$ in $L^1(\mathbb{R}^{d})$ for some $\mu\in L^1(\mathbb{R}^{d})$. Using the uniform convergence in \ref{conv:varrho_scl_del_hm} we see that $\mu=\varrho_t^\sigma$. As this argument holds for any chosen subsequence, we deduce that the convergence \ref{conv:varrho_scl_del_pro} holds.
    
    \underline{Step 4: Proof of \ref{conv:varrho_scl_del_strong}.}
    First, we observe the strong convergence $\omega_{\sigma,\delta}\to \omega_\sigma$ in $L^2(\mathbb{R}^{d})$. Moreover, for all $\varepsilon>0$ we may choose $\lambda_\varepsilon\in C_c^\infty(\mathbb{R}^{d})$ and $\delta_\ast\in (0,1)$ such that for all $\delta\in (0,\delta_\ast)$ we have
    \begin{align*}
        \|\omega_{\sigma,\delta}-\omega_\sigma\|_{L_x^2}+\|\omega_{\sigma}-\lambda_\varepsilon\|_{L_x^2}\leq \varepsilon.
    \end{align*}
    Then, by the Young convolution inequality and the triangle inequality we get
    \begin{align}
        \label{eq:omega_scl_del_smth_cmp_app}
        \begin{split}
            &\|(\varrho^{\sigma,\delta}\psi)\ast\omega_{\sigma,\delta}-(\varrho^{\sigma,\delta}\psi)\ast\lambda_\varepsilon\|_{L_{t,x}^2}
            \leq \|\psi\|_{L^\infty(\mathbb{R}^{d})}\|\varrho^{\sigma,\delta}\|_{L_t^2 L_x^1}\|\omega_{\sigma,\delta}-\lambda_\varepsilon\|_{L_x^2}
            \leq C\varepsilon,\\
            &\|(\varrho^\sigma\psi)\ast\omega_\sigma-(\varrho^\sigma\psi)\ast\lambda_\varepsilon\|_{L_{t,x}^2}
            \leq \|\psi\|_{L^\infty(\mathbb{R}^{d})}\|\varrho^\sigma\|_{L_t^2 L_x^1}\|\omega_\sigma-\lambda_\varepsilon\|_{L_x^2}
            \leq C\varepsilon,
        \end{split}
    \end{align}
    for some constant $C>0$ independent of $\varepsilon$ and $\delta$ as a result of \ref{bnd:mf_scl_del_kl} from Lemma \ref{lem:unif_bnds_approx}. We define the continuous operator $T\!:\!L^2(0,T;\, H_0^{-M}(B_R))\to L^2([0,T]\times \mathbb{R}^{d})$ by $T(f)=f\ast\lambda_\varepsilon$. Indeed, by the Plancherel theorem we have, for any $f\in L^2(0,T;\, H_0^{-M}(B_R))$,
    \begin{align*}
        &\| T(f)\|_{L_{t,x}^2 }^2=\|f\ast\lambda_\varepsilon\|_{L_{t,x}^2}^2
        =\int_0^T \|(1+|\xi|^2)^{-M/2}\hat{f}_t\;(1+|\xi|^2)^{M/2}\hat{\lambda}_\varepsilon\|_{L_{x}^2}^2\,\mathrm{d} t\\
        &\leq \|(1+|\xi|^2)^{M/2}\hat{\lambda}_\varepsilon\|_{L_x^\infty}^2\|f\|_{L_t^2 H_x^{-M}}^2.
    \end{align*}
    Thus, we get strong convergence $(\varrho^{\sigma,\delta}\psi)\ast\lambda_\varepsilon\to (\varrho^\sigma\psi)\ast\lambda_\varepsilon$ in $L^2([0,T]\times \mathbb{R}^{d})$ by \ref{conv:varrho_scl_del_hm}. Finally, we use \eqref{eq:omega_scl_del_smth_cmp_app} and get the estimate
    \begin{align*}
        &\limsup_{\delta\to 0}\|(\varrho^{\sigma,\delta}\psi)\ast\omega_{\sigma,\delta}-(\varrho^\sigma\psi)\ast\omega_\sigma\|_{L_{t,x}^2}
        \\
        &\leq \lim_{\delta\to 0}\|(\varrho^{\sigma,\delta}\psi)\ast\lambda_\varepsilon-(\varrho^\sigma\psi)\ast\lambda_\varepsilon\|_{L_{t,x}^2}+2\,C\,\varepsilon
        =2\,C\,\varepsilon.
    \end{align*}
    As $\varepsilon$ is arbitrary, it shows that we get the convergence $(\varrho^{\sigma,\delta}\psi)\ast\omega_{\sigma,\delta}\to (\varrho^\sigma\psi)\ast\omega_\sigma$ as $\delta\to 0$ in $L^2([0,T]\times \mathbb{R}^{d})$. This proves the desired convergence.
    
    \underline{Step 5: Proof of \ref{conv:varrho_scl_del_weak}.} The proof follows the same argument as in Lemma \ref{lem:weak_convergences}.
\end{proof}

\begin{proof}[Proof of Theorem \ref{thm:mf_scl}]
    Let $\varrho^\sigma$ be the curve of probabilities from Lemma \ref{lem:conv_mf_scl_del}. We claim that $\varrho^\sigma$ satisfies the weak formulation of \eqref{eq:mf_scl}. Indeed, let $\psi\in~ C_c^\infty([0,T]\times \mathbb{R}^{d})$ be a~test function. We multiply it with the PDE in \eqref{eq:mf_scl_delta}, we integrate on $[0,T]\times \mathbb{R}^{d}$ and use integration by parts to derive the equality
    \begin{align*}
        &\int_{\mathbb{R}^{d}}\psi_T\,\mathrm{d} \varrho_T^{\sigma,\delta}-\int_{\mathbb{R}^{d}}\psi_0\,\mathrm{d} \varrho_0\\
        &=\int_0^T \int_{\mathbb{R}^{d}} \partial_t\psi\; \varrho_t^{\sigma,\delta}\,\mathrm{d} x\,\mathrm{d} t
        -\int_0^T \int_{\mathbb{R}^{d}}(\varrho_t^{\sigma,\delta}\nabla\psi)\ast\omega_{\sigma,\delta}\cdot  (\nabla \varrho_t^{\sigma,\delta}+\varrho_t^{\sigma,\delta}\nabla V^\delta)\ast\omega_{\sigma,\delta}\,\mathrm{d} x\,\mathrm{d} t.
    \end{align*}
    The first term and third term converge by \ref{conv:varrho_scl_del_pro} and \ref{conv:varrho_scl_del_hm} of Lemma \ref{lem:conv_mf_scl_del} respectively. For the last term, the desired convergence comes by combining \ref{conv:varrho_scl_del_strong} and \ref{conv:varrho_scl_del_weak} of Lemma \ref{lem:conv_mf_scl_del}. With this we have proven the theorem.
\end{proof}


\section{Auxiliary Results}
\label{app:aux}
\begin{lem}[see for instance {\cite[Lemma D.2]{carrillo_skrzeczkowski_warnett_2024_stein}}]\label{lem:control_neg_log}
    For any given pair of measurable functions $\rho,V: \mathbb{R}^d \to [0,\infty)$ we have
    \begin{equation}\label{eq:pointwise_lower_bound}
        \frac{1}{2}\,\rho\, \log \rho + \rho\, V(x) \geq - \frac{1}{e} \, e^{-V}.
    \end{equation}
    Moreover, if $\rho(x)\,\mathrm{d} x$ is a probability measure and $e^{-V}\in L^1(\mathbb{R}^{d})$, then
    \begin{equation}\label{eq:lower_bound_negative_log}
        \frac{1}{4}\int_{\mathbb{R}^d} \rho\, |\log \rho| \,\mathrm{d} x \leq \frac{1}{4}\int_{\mathbb{R}^d} \rho\, \log \rho \,\mathrm{d} x + \int_{\mathbb{R}^d} \rho\, V(x) \,\mathrm{d} x + \frac{1}{e} \int_{\mathbb{R}^d} e^{-V(x)} \,\mathrm{d} x.
    \end{equation}
\end{lem}

\begin{lem}\label{lem:interpolation}
    For all $\delta > 0$, there exists a constant $C_{\delta}$ such that for all $f \in L^2(0,T; H^1(B_R))$
    \begin{equation}\label{eq:time_int_inequality}
        \|f\|_{L^2((0,T)\times B_R)} \leq \delta \, \|\nabla f\|_{L^2((0,T)\times B_R)} + C_{\delta}\, \|f\|_{L^2(0,T; L^1(B_R))}.
    \end{equation}
\end{lem}

\begin{proof}
    It is sufficient to prove that for all $\delta > 0$, there exists a constant $C_{\delta}$ such that for all $f \in H^1(B_R)$
    \begin{equation}\label{eq:time_int_inequality_just_space}
        \|f\|_{L^2(B_R)} \leq \delta \, \|\nabla f\|_{L^2(B_R)} + C_{\delta}\, \|f\|_{L^1(B_R)}.
    \end{equation}
    Then, \eqref{eq:time_int_inequality} follows from \eqref{eq:time_int_inequality_just_space} by taking the square, integrating in time and taking the square root. To see \eqref{eq:time_int_inequality_just_space}, we use classical compactness argument. If \eqref{eq:time_int_inequality_just_space} fails to be true, there exists $\delta>0$ such that for all $n$ there exists a function $f_n$ such that
    $$
    \|f_n\|_{L^2(B_R)} > \delta \, \|\nabla f_n\|_{L^2(B_R)} + n\, \|f_n\|_{L^1(B_R)}.
    $$
    Letting $g_n = f_n/\|f_n\|_{L^2(B_R)}$, we see that
    \begin{equation}\label{eq:inequality_int_ineq_normalized}
        1 > \delta \, \|\nabla g_n\|_{L^2(B_R)} + n\, \|g_n\|_{L^1(B_R)}
    \end{equation}
    so that $\{g_n\}$ is bounded in $H^1(B_R)$. Hence, it has a subsequence $g_{n_k} \to g$ converging strongly in $L^2(B_R)$ so that $\|g\|_{L^2(B_R)} = 1$ and $g \neq 0$. However from \eqref{eq:inequality_int_ineq_normalized} we see that $\|g_{n_k}\|_{L^1(B_R)} \leq \frac{1}{n_k}$ so that $g = 0$. This is in contradiction with $\|g\|_{L^2(B_R)} = 1$.
\end{proof}

\begin{lem}\label{lem:conv_mollifiers_scaled_x}
    Let $r\geq \frac{d}{2}$. Suppose that $\omega_{\varepsilon} = \frac{1}{\varepsilon^d}\, \omega\left(\frac{x}{\varepsilon}\right)$ and $\int_{\mathbb{R}^d} \omega^2(x)\, |x|^{2r} \,\mathrm{d} x < \infty$. Let $\{\rho_{\varepsilon}\}$ be a sequence of functions on $(0,T)\times\mathbb{R}^d$ bounded in $L^{\infty}(0,T; L^1(\mathbb{R}^d))$. Then, the sequence $\{\rho_{\varepsilon} \ast (\omega_{\varepsilon}\,|x|^r)\}$ is bounded in $L^{\infty}(0,T; L^2(\mathbb{R}^d))$:
    $$
    \| \rho_{\varepsilon} \ast (\omega_{\varepsilon}\,|x|^r) \|_{L_t^{\infty}L_x^2} \leq  \varepsilon^{r-\frac{d}{2}}\, \| \rho_{\varepsilon}\|_{L_t^\infty L_x^1}\, \| \omega\,|x|^r\|_{L^2(\mathbb{R}^d)}.
    $$
    In particular, if  $r > \frac{d}{2}$, then
    $\rho_{\varepsilon} \ast (\omega_{\varepsilon}\,|x|^r) \to 0$ in $L^{\infty}(0,T; L^2(\mathbb{R}^d))$.
\end{lem}
\begin{proof} We compute the norm
    $\| \omega_{\varepsilon}\,|x|^{r} \|_{L^2(\mathbb{R}^d)}$:
    $$
    \| \omega_{\varepsilon}\,|x|^{r} \|_{L^2(\mathbb{R}^d)}^2 = \int_{\mathbb{R}^d} \frac{1}{\varepsilon^{2d}} \, \omega^2 \left(\frac{x}{\varepsilon}\right) |x|^{2r} \,\mathrm{d} x = \varepsilon^{2r-d} \int_{\mathbb{R}^d} \omega^2(x) \, |y|^{2r} \,\mathrm{d} y = \varepsilon^{2r-d} \, \| \omega\,|x|^r\|_{L^2(\mathbb{R}^d)}^2,
    $$
    so that the claim follows by the Young's convolutional inequality.
\end{proof}

\begin{lem}\label{lem:weak_conv_of_rho_nablaV}
    Let $\{\rho_{\varepsilon}\}_{\varepsilon}$ be a sequence such that $\rho_{\varepsilon} \rightharpoonup \rho$ weakly in $L^1((0,T)\times \mathbb{R}^{d})$ and let $V$ satisfy Assumption \ref{ass:pot}. If $\{\rho_{\varepsilon}\, V\}_{\varepsilon}$ is uniformly bounded in $L^1((0,T)\times \mathbb{R}^{d})$  then $\rho_{\varepsilon}\, \nabla V \rightharpoonup \rho\, \nabla V$ in $L^1((0,T)\times \mathbb{R}^{d})$.
\end{lem}
\begin{proof}
    Clearly, we have $\rho\, V \in L^1((0,T)\times\mathbb{R}^d)$ because the functional $\varrho \mapsto \int_0^T \int_{\mathbb{R}^d} |\varrho|\, V \,\mathrm{d} x \,\mathrm{d} t$ is strongly lower semicontinuous on $L^1((0,T)\times\mathbb{R}^d)$ (recall that $V\geq 0$) and convex so it is weakly lower semicontinuous. For the remaining weak convergence take any test function $\psi \in L^{\infty}((0,T)\times\mathbb{R}^d;\mathbb{R}^{d})$. Then, we estimate
    \begin{align*}
        &\left|\int_0^T \int_{\mathbb{R}^{d}} (\rho_{\varepsilon} - \rho)\, \nabla V\cdot \psi \,\mathrm{d} x \,\mathrm{d} t  \right| \\
        &\leq \left|\int_0^T \int_{|x|\leq R} (\rho_{\varepsilon} - \rho)\, \nabla V\cdot \psi \,\mathrm{d} x \,\mathrm{d} t\right| + \left|\int_0^T \int_{|x| >  R} (\rho_{\varepsilon} - \rho)\, \nabla V\cdot \psi \,\mathrm{d} x \,\mathrm{d} t\right|.
    \end{align*}
    The first term converges to 0 when $\varepsilon \to 0$ as $(\nabla V\cdot\psi)\mathds{1}_{B_R}\in L^\infty((0,T)\times \mathbb{R}^{d})$. For the second term, using nonnegativity of $V$ we can estimate it as
    \begin{align*}
        \int_0^T \int_{|x| >  R} (\rho_{\varepsilon} + \rho) \, V\,  \frac{|\nabla V|}{V} \, |\psi| \,\mathrm{d} x \,\mathrm{d} t.
    \end{align*}
    We note by
    \eqref{eq:bound_nabla_V_by_V}, the fact that $p_0>1$ and that $V$ is coercive
    \begin{align*}
        \lim_{R\to\infty}\sup_{|x|\geq R}\frac{|\nabla V(x)|}{V(x)}
        \leq \lim_{R\to\infty}\sup_{|x|\geq R}C_V^{\frac{1}{p_0}}\left(V(x)^{-p_0}+V(x)^{1-p_0}\right)^{\frac{1}{p_0}}
        =0.
    \end{align*}
    This means that $\big\|\frac{|\nabla V|}{V}\,|\psi|\,\mathds{1}_{\mathbb{R}^{d}\setminus B_R}\big\|_{L_{t,x}^\infty}\to 0$ as $R\to \infty$, so using the bounds on $\rho\, V$ and $\rho_{\varepsilon} \, V$ in $L^{1}((0,T) \times \mathbb{R}^d)$ we obtain
    \begin{align*}
        \limsup_{\varepsilon \to 0} \left|\int_0^T \int_{\mathbb{R}^{d}} (\rho_{\varepsilon} - \rho)\, \nabla V\cdot \psi \,\mathrm{d} x \,\mathrm{d} t  \right| \leq  C\,\left\|\frac{|\nabla V|}{V}\,|\psi|\,\mathds{1}_{\mathbb{R}^{d}\setminus B_R}\right\|_{L_{t,x}^\infty}.
    \end{align*}
    As $R$ is arbitrary, the proof is concluded.
\end{proof}

\begin{lem}
    \label{lem:optimal_slsi_val}
    Let $f\in C^0(\mathbb{R}^{d})$ and $D\geq 1$ be such that $f(0)=1$ and
    \begin{align*}
        f(x)\geq \frac{1}{D}\frac{1}{1+|x|^2}.
    \end{align*}
    For all $\sigma\in (0,1)$ define $f_\sigma(x):=f(\sigma x)$. Then for all $\varepsilon\in (0,1)$, there exists $\sigma_\ast\in (0,1)$ such that for all $\sigma\in (0,\sigma_\ast)$ we have
    \begin{align*}
        f_\sigma(x)\geq \frac{1-\varepsilon}{1+|x|^2}.
    \end{align*}
\end{lem}
\begin{proof}
    Let $c := D(1-\varepsilon)$. Since $f$ is continuous at $0$ and $f(0)=1$, there exists $\delta > 0$ such that
    \begin{align*}
        \inf_{|y|< \delta}f(y) \geq 1 - \varepsilon.
    \end{align*}
    Now, choose $\sigma_\ast \in (0,1)$ sufficiently small such that for all $\sigma \in (0,\sigma_\ast)$
    \begin{align}
        \label{eq:sigma_choice}
        c \sigma^2 < 1
        \qquad \text{and}\qquad1 + \frac{\delta^2}{\sigma^2} \geq c(1+\delta^2).
    \end{align}
    Next we split into two cases. The first case is when $|x| < \delta/\sigma$. Then $|\sigma x| < \delta$, so
    \begin{align*}
        f_\sigma(x)=f(\sigma x) \geq 1 - \varepsilon
        \geq \frac{1-\varepsilon}{1+|x|^2}.
    \end{align*}
    It remains to consider the other case $|x| \geq \delta/\sigma$. Using the assumed bound on $f$, we have
    \begin{align*}
        f(\sigma x) \geq \frac{1}{D(1+\sigma^2|x|^2)}.
    \end{align*}
    Thus, it suffices to show
    \begin{align*}
        \frac{1}{D(1+\sigma^2|x|^2)} \geq \frac{1-\varepsilon}{1+|x|^2} \iff (1-c) + (1-c\,\sigma^2)\,|x|^2 \geq 0.
    \end{align*}
    Note that $(1-c\,\sigma^2) \geq 0$ by \eqref{eq:sigma_choice} so we can estimate using $|x|^2 \geq \delta^2/\sigma^2$
    $$
    (1-c) + (1-c\,\sigma^2)\,|x|^2 \geq (1-c) + (1-c\,\sigma^2)\,\frac{\delta^2}{\sigma^2}  = 1 + \frac{\delta^2}{\sigma^2} - c\,(1+\delta^2)\geq 0
    $$
    by the second item in \eqref{eq:sigma_choice}. The proof is concluded.
\end{proof}

\subsection*{Acknowledgements}
JAC and JS were supported by the Advanced Grant Nonlocal-CPD (Nonlocal PDEs for Complex Particle Dynamics: Phase Transitions, Patterns and Synchronization) of the European Research Council Executive Agency (ERC) under the European Union’s Horizon 2020 research and innovation programme (grant agreement No. 883363). JAC was also partially supported by the EPSRC grant numbers EP/T022132/1 and EP/V051121/1. JW was supported by the Engineering and Physical Sciences Research Council (grant number EP/W524311/1)

\bibliographystyle{abbrv}
\bibliography{fastlimit}
\end{document}